\documentclass[12pt,leqno]{amsart}
\usepackage{amsmath}
\usepackage{amssymb}
\usepackage[hidelinks]{hyperref}
\usepackage{orcidlink}
\usepackage{enumitem}
\usepackage{needspace}

\numberwithin{equation}{section} \newtheorem{thm}{Theorem}[section] \newtheorem{lem}[thm]{Lemma}  \newtheorem{prop}[thm]{Proposition} \newtheorem{cor}[thm]{Corollary} \newtheorem{defi}[thm]{Definition} \newtheorem{rmk}[thm]{Remark} 

\newcommand{\Ind}{\operatorname{Ind}}
\newcommand{\End}{\operatorname{End}}
\newcommand{\one}{\mathbf{1}}
\newcommand{\id}{\operatorname{id}}
\newcommand{\wt}{\operatorname{wt}}

\newcommand{\Res}{\operatorname{Res}}
\newcommand{\Aut}{\operatorname{Aut}}

\newcommand{\Vir}{\operatorname{Vir}}
\newcommand{\Rep}{\operatorname{Rep}}
\newcommand{\Sym}{\operatorname{Sym}}
\newcommand{\Spec}{\operatorname{Spec}}
\newcommand{\Hom}{\operatorname{Hom}}

\newcommand{\ch}{\operatorname{ch}}
\newcommand{\grch}{\operatorname{grch}}
\newcommand{\PSL}{\operatorname{PSL}}
\newcommand{\SO}{\operatorname{SO}}
\newcommand{\OO}{\mathcal O}
\newcommand{\Csq}{\mathcal C_{\mathrm{sq}}}
\newcommand{\Dk}{\mathcal D_k}

\newcommand{\Ovir}{\mathcal O_1}
\newcommand{\CAlg}{\operatorname{CAlg}}
\newcommand{\IndCMY}{\operatorname{Ind}_{\mathrm{CMY}}}
\newcommand{\PSLtwo}{\operatorname{PSL}_2(\C)}
\newcommand{\Otwo}{O_2(\C)}

\def\Z{\mathbb{Z}}   \def\C{\mathbb{C}}   

\allowdisplaybreaks

\begin{document}

\title[]{Classification of Strongly Rational Vertex Operator Algebras with Central and Effective Central Charge One}

\author{Shun Xu~\orcidlink{0009-0006-8080-8107}}

\address{School of Mathematical Sciences, Anhui University, Anhui, Hefei, 230601, China}

\email{shunxu@ahu.edu.cn}


\hypersetup{pdftitle={Classification of Strongly Rational Vertex Operator Algebras with Central and Effective Central Charge One},pdfauthor={Shun Xu}}

\begin{abstract}
We classify the strongly rational vertex operator algebras whose central charge and effective central charge are both one. Every such algebra is isomorphic to one of \(V_L,V_L^+,V_{A_1}^{A_4},V_{A_1}^{S_4},V_{A_1}^{A_5},\) where $L$ ranges over the positive-definite even lattices of rank one and $A_1=\sqrt2\Z$.
\end{abstract}

\maketitle

\section{Introduction}

The classification of rational vertex operator algebras at central charge one is a basic rank-one problem in the structure theory of vertex operator algebras. The expected examples are the positive-definite rank-one lattice vertex operator algebras, their $\mathbb Z_2$-orbifolds, and the three polyhedral orbifolds of the $A_1$ lattice vertex operator algebra. Earlier characterizations of these models often assumed that the ambient vertex operator algebra decomposes as a direct sum of highest-weight modules for the Virasoro algebra $L(1,0)$; see, for example, \cite{DJplusI,DJplusII,DJA4}. This assumption is automatic in the unitary setting, but it is not a formal consequence of rationality of the ambient vertex operator algebra.

The aim of this paper is to remove that Virasoro complete-reducibility hypothesis. We call a vertex operator algebra \emph{strongly rational} if it is simple, rational, $C_2$-cofinite, self-contragredient, and of CFT type. For a rational vertex operator algebra $V$, we write $\widetilde c$ for its effective central charge. Let $A_1=\mathbb Z\alpha$ with $\langle\alpha,\alpha\rangle=2$, and identify $V_{A_1}\cong L_{\widehat{\mathfrak{sl}}_2}(1,0)$ and $\Aut(V_{A_1})\cong\PSLtwo\cong\SO_3(\C)$. We regard $A_4,S_4,A_5$ as the tetrahedral, octahedral, and icosahedral rotation subgroups of $\SO_3(\C)$.

Our main theorem is the following.

\begin{thm}\label{thm:main}
Let $V$ be a simple, rational, $C_2$-cofinite, self-contragredient vertex operator algebra of CFT type. Assume $c(V)=\widetilde c(V)=1$. Then there exists a positive-definite even lattice $L$ of rank one such that $V$ is isomorphic to one of $V_L,V_L^+,V_{A_1}^{A_4},V_{A_1}^{S_4},V_{A_1}^{A_5}$. In particular, no assumption that $V$ is completely reducible as an $L(1,0)$-module, or that $V$ is a direct sum of Virasoro highest-weight modules, is needed.
\end{thm}

The proof has two stages. First, the modularity theorem of Dong and Jiang, together with the classical completeness result at $c=1$, restricts the vacuum character to the five families appearing in Theorem~\ref{thm:main}; this is recorded in Theorem~\ref{thm:character-reduction}. We then prove rigidity for each character. The lattice character is handled by the effective-central-charge theorem of Dong--Mason. For the four remaining characters, the main issue is that the ambient $L(1,0)$-module need not be semisimple, so the reconstruction cannot begin by decomposing $V$ globally into irreducible Virasoro modules.

The replacement is to work with semisimple subcores generated by specific primary vectors. Put $X_m=L(1,m^2)$ for $m\in\mathbb Z_{\ge0}$. The finite direct sums of the $X_m$ form a symmetric tensor category $\Csq\simeq_{\mathrm{sym}}\Rep\PSLtwo$, with $X_m$ corresponding to the $(2m+1)$-dimensional irreducible representation $W_{2m}$. A coefficient-image argument shows that the Virasoro modules generated by the coefficients of the relevant intertwining operators remain in this square subcategory. Thus the particular subcores needed in the proof are semisimple even though $V\downarrow L(1,0)$ need not be.

For the three polyhedral characters, the first deviation from the vacuum character occurs in weights $9,16,36$, respectively. The invariant bilinear form produces a primary vector at that weight, and the first-singular-vector argument identifies its cyclic Virasoro module with $X_3,X_4,X_6$. The resulting square core corresponds to a finitely generated commutative $\PSLtwo$-algebra. Affine invariant theory then identifies the relevant homogeneous orbit, and comparison with the prescribed graded character yields the fixed-point algebra $V_{A_1}^{A_4}$, $V_{A_1}^{S_4}$, or $V_{A_1}^{A_5}$.

The branch with character $\ch V_L^+$ requires a second reconstruction. A weight-four primary first recovers the Heisenberg orbifold $M(1)^+$. The first lattice primary is then extracted as an irreducible $M(1)^+$-module, and the integral-charge subcategory is normalized by the compact $O(2)$-action on the corresponding rank-one lattice VOA. This gives a symmetric tensor category equivalent to $\Rep O_2(\C)$, in which the extension is controlled by the standard two-dimensional representation and its invariant quadratic form. This proves the $V_L^+$ rigidity for lattice norm greater than $8$. The norms $2,4,6,8$ require separate low-weight arguments: the first two are covered by known uniqueness theorems, while the norms $6$ and $8$ are treated by dimension saturation and the weight-four invariant theory developed in Section~\ref{sec:low}. In the last case a regular weight-four primary has orbit $G/D_2$, and the lattice identity $V_{A_1}^{D_2}\cong V_{2A_1}^+$ completes the reconstruction.

The paper is organized as follows. Section~\ref{sec:categorical} fixes the categorical conventions and records the rank-one lattice-VOA input. Section~\ref{sec:prelim} contains the character reduction, the square tensor category, the direct-limit extension machinery, and the homogeneous-coordinate-algebra dictionary. Section~\ref{sec:squarecore} proves the exceptional character-rigidity theorem. Section~\ref{sec:weight4} reconstructs the $M(1)^+$ core, Section~\ref{sec:o2} establishes the $O_2$ charge category and treats the generic rank-one orbifold branch, and Section~\ref{sec:low} handles the four low-norm cases. Section~\ref{sec:final} proves the lattice branch and assembles the preceding results to prove Theorem~\ref{thm:main}.

\section{Categorical conventions and vertex tensor categories}
\label{sec:categorical}

We collect here the categorical conventions and auxiliary results used throughout the paper. The module-theoretic conventions and the $C_1$-cofinite Virasoro package are taken from Creutzig--Jiang--Orosz Hunziker--Ridout--Yang \cite{CJORY}; the conventions and results concerning $P(z)$-tensor products, vertex tensor categories, vertex tensor functors, and compact orbifolds are taken from McRae \cite{McRae}, together with the functorial framework of Creutzig--Kanade--McRae \cite{CKM}. The $M(1)^+$ results used below are from Adamovi\'c--Lin--Yang \cite{ALY}; the standard rank-one lattice-VOA construction and its $\mathbb Z_2$-orbifold are recalled from Frenkel--Lepowsky--Meurman \cite[Chapter~8]{FLM}, Lepowsky--Li \cite[Sections~6.4--6.5]{LepowskyLi}, and Dong--Griess \cite[Section~2]{DG98}. The invariant-bilinear-form results are from Li \cite{LiInvariant}, the level-one $\mathfrak{sl}_2$ uniqueness results are from McRae--Yang \cite{MY}, and the direct-limit completion and extension--algebra correspondence are taken from Creutzig--McRae--Yang \cite{CMY}. Only the conventions and results needed in the sequel are recalled.

\subsection{Vertex-operator-algebra module conventions}
\label{subsec:module-conventions}

We first fix the module-theoretic terminology used throughout the paper, following the conventions of \cite[Section~2.1]{CJORY}. Let $(V,Y,\one,\omega)$ be a vertex operator algebra of central charge $c$. By a \emph{vertex operator algebra extension} of $V$ we mean a vertex operator algebra $U$ together with an injective vertex-operator-algebra homomorphism $\iota:V\hookrightarrow U$ preserving the conformal vector. We usually identify $V$ with its image under $\iota$, so that $V\subset U$ and $U$ is naturally a $V$-module via the restriction of its vertex operator map.

A \emph{weak $V$-module} is a vector space $W$ equipped with a linear vertex-operator map \(Y_W:V\longrightarrow(\End W)[[x,x^{-1}]], Y_W(v,x)=\sum_{n\in\Z}v_nx^{-n-1}\), satisfying the following axioms. For $v\in V$ and $w\in W$, one has $v_nw=0$ for all sufficiently large $n$; $Y_W(\one,x)=\id_W$; and, for $u,v\in V$,
\begin{align*}
&x_0^{-1}\delta\!\left(\frac{x_1-x_2}{x_0}\right)
 Y_W(u,x_1)Y_W(v,x_2)
-x_0^{-1}\delta\!\left(\frac{x_2-x_1}{-x_0}\right)
 Y_W(v,x_2)Y_W(u,x_1)\\
&\hspace{38mm}=
 x_2^{-1}\delta\!\left(\frac{x_1-x_0}{x_2}\right)
 Y_W\bigl(Y(u,x_0)v,x_2\bigr).
\end{align*}
Writing \(Y_W(\omega,x)=\sum_{n\in\Z}L(n)x^{-n-2}\), the operators $L(n)$ satisfy \([L(m),L(n)]=(m-n)L(m+n)+\frac{m^3-m}{12}\delta_{m+n,0}c\id_W\). The $L(-1)$-derivative property \(Y_W(L(-1)v,x)=\frac{d}{dx}Y_W(v,x)\) holds. No grading on $W$ is part of the definition of a weak module.

A \emph{generalized $V$-module} is a weak $V$-module with a $\C$-grading \(W=\bigoplus_{h\in\C}W_{[h]},W_{[h]}=\{w\in W\mid(L(0)-h)^Nw=0\text{ for some }N\ge1\},\) so that $W_{[h]}$ is the generalized $L(0)$-eigenspace of eigenvalue $h$. It is \emph{lower bounded} if, for every $h\in\C$, \(W_{[h+n]}=0\) for all sufficiently negative $n\in\Z$. It is \emph{grading restricted} if it is lower bounded and \(\dim_\C W_{[h]}<\infty\) for every $h\in\C$.

An \emph{ordinary $V$-module} is a grading-restricted generalized $V$-module on which $L(0)$ acts semisimply. Thus \(W=\bigoplus_{h\in\C}W_{(h)},W_{(h)}=\{w\in W\mid L(0)w=hw\},\) with finite-dimensional weight spaces and the preceding lower-bounded condition. A generalized $V$-module $W$ has \emph{length} $\ell$ if there is a filtration \(W=W_1\supset W_2\supset\cdots\supset W_{\ell+1}=0\) by generalized $V$-submodules such that every $W_i/W_{i+1}$ is irreducible; it has \emph{finite length} if such an $\ell$ is finite.

We call \(V\) \emph{rational} if every admissible \(V\)-module is completely reducible. Likewise, a category of ordinary $V$-modules is called \emph{semisimple} if each of its objects is a direct sum of simple ordinary modules.

Let $W=\bigoplus_{h\in\C}W_{[h]}$ be a lower-bounded generalized $V$-module. Its \emph{contragredient module} is the restricted dual \(W':=\bigoplus_{h\in\C}W_{[h]}^*,\) with vertex-operator action $Y_{W'}$ characterized by
\[
\langle Y_{W'}(v,x)w',w\rangle
=
\langle w',Y_W^{\circ}(v,x)w\rangle,
\qquad
Y_W^{\circ}(v,x)
:=Y_W\!\left(e^{xL(1)}(-x^{-2})^{L(0)}v,x^{-1}\right).
\]
Here $Y_W^{\circ}$ is the \emph{opposite vertex operator}. We also write \(\overline W:=\prod_{h\in\C}W_{[h]}\) for the formal completion of $W$ with respect to its generalized $L(0)$-grading. In particular, \(V'=\bigoplus_{n\in\Z}V_n^*\) for the adjoint module of an integrally graded vertex operator algebra.

For a weak $V$-module $W$, let \(V_+:=\bigoplus_{n\in\Z_{>0}}V_n\) and \(C_1(W):=\operatorname{span}_{\C}\{u_{-1}w\mid u\in V_+,\ w\in W\}\). The module $W$ is \emph{$C_1$-cofinite} if its $C_1$-quotient $W/C_1(W)$ is finite dimensional. For the vertex operator algebra itself, set \(C_2(V):=\operatorname{span}_{\C}\{u_{-2}v\mid u,v\in V\},\) and call $V$ \emph{$C_2$-cofinite} if \(\dim_\C V/C_2(V)<\infty\). A vertex operator algebra $V$ is \emph{simple} if it has no nonzero proper ideals. It is \emph{self-contragredient} if $V\cong V'$ as a $V$-module, and it is of \emph{CFT type} if \(V_n=0\) for $n<0$ and \(V_0=\C\one\). Throughout this paper, \emph{strongly rational} means simple, rational, $C_2$-cofinite, self-contragredient, and of CFT type.

\subsection{Tensor, braided, and symmetric categories}

A \emph{tensor category} in this paper is a monoidal category $\mathcal C$ with tensor product bifunctor \(\boxtimes:\mathcal C\times\mathcal C\longrightarrow\mathcal C\), tensor unit $\mathbf 1_{\mathcal C}$, associativity isomorphisms \(\mathcal A_{X,Y,Z}:X\boxtimes(Y\boxtimes Z)\overset{\sim}{\longrightarrow}(X\boxtimes Y)\boxtimes Z,\) and left and right unit isomorphisms, satisfying the usual pentagon and triangle coherence axioms. A \emph{braiding} is a natural family of isomorphisms \(\mathcal R_{X,Y}:X\boxtimes Y\overset{\sim}{\longrightarrow}Y\boxtimes X\) satisfying the two hexagon axioms. The braided tensor category is \emph{symmetric} if \(\mathcal R_{Y,X}\circ\mathcal R_{X,Y}=\id_{X\boxtimes Y}\text{for all }X,Y\in\mathcal C.\) A strong tensor functor $F:\mathcal C\to\mathcal D$ is equipped with natural isomorphisms \(J_{X,Y}:F(X)\boxtimes F(Y)\overset{\sim}{\longrightarrow}F(X\boxtimes Y)\) and an isomorphism \(\varphi:\mathbf 1_{\mathcal D}\overset{\sim}{\longrightarrow}F(\mathbf 1_{\mathcal C})\), compatible with the associativity and left and right unit constraints. If $\mathcal C$ and $\mathcal D$ are braided, write $\mathcal R^{\mathcal C}$ and $\mathcal R^{\mathcal D}$ for their braidings. The functor is \emph{braided} if, for all $X,Y\in\mathcal C$, \(F(\mathcal R^{\mathcal C}_{X,Y})\circ J_{X,Y}=J_{Y,X}\circ\mathcal R^{\mathcal D}_{F(X),F(Y)}.\) A braided tensor functor whose underlying functor is an equivalence of categories is called a \emph{braided tensor equivalence}. When both categories are symmetric, we also call such an equivalence a \emph{symmetric tensor equivalence}.

For a complex reductive algebraic group $G$, we write $\Rep G$ for the category of finite-dimensional rational algebraic $G$-modules. It is a symmetric tensor category with the usual tensor product, tensor unit $\C$, and flip \(m\otimes n\longmapsto n\otimes m\) as braiding. We write $\Rep^{\mathrm{rat}}G$ for the category of arbitrary rational algebraic $G$-modules. Equivalently, an object of $\Rep^{\mathrm{rat}}G$ is a $G$-module that is the union of its finite-dimensional rational $G$-submodules.

If $\mathcal D$ is a full subcategory of a tensor category $\mathcal C$, we call $\mathcal D$ a \emph{full tensor subcategory} when it contains the tensor unit, is closed under the tensor product and isomorphisms, and carries the corresponding restrictions of the associativity and unit constraints. The same terminology is used for braided or symmetric tensor subcategories when the braiding restricts.

\subsection{Basic algebraic-geometric conventions}
\label{subsec:AG-conventions}

All algebraic varieties, schemes, and algebraic groups in this paper are over $\C$, and all schemes that occur below are of finite type over $\C$. For a finitely generated commutative $\C$-algebra $R$, we write $\Spec R$ for the associated affine scheme. The algebra $R$ is \emph{reduced} if it has no nonzero nilpotent elements, equivalently if its nilradical $\sqrt{(0)}$ is zero. When $R$ is reduced, we regard $X=\Spec R$ as an affine variety and write $\OO(X)=\Gamma(X,\mathcal O_X)\cong R$ for its coordinate algebra. Thus, if $X\subset\C^n$ is an affine algebraic set, then $\OO(X)=\C[x_1,\ldots,x_n]/I(X)$, and every regular function is represented by a polynomial in the ambient coordinates. A map $\varphi:X\to Y$ of affine varieties is regular precisely when it induces a pullback homomorphism $\varphi^*:\OO(Y)\to\OO(X)$. For a finite-dimensional complex vector space $E$, in particular, $\OO(E)=\Sym(E^*)$.

Closed points of $\Spec R$ correspond to maximal ideals of $R$, and over $\C$ the weak Hilbert Nullstellensatz identifies them with the ordinary $\C$-points. We also use the strong form: if $J\subseteq\C[x_1,\ldots,x_n]$ is an ideal and $V(J)$ is its zero locus, then \(I(V(J))=\sqrt J.\) Consequently, two reduced closed subschemes of an affine space having the same $\C$-points have the same radical defining ideal and hence coincide scheme-theoretically. We use these conventions in both their variety and affine-scheme forms; see \cite[Chapters~I--II]{Hartshorne}.

We record the local notions used later. Let $X=\Spec R$ and let $x\in X(\C)$ correspond to the maximal ideal $\mathfrak m\subset R$. The local ring at $x$ is \(\mathcal O_{X,x}=R_{\mathfrak m},\) with maximal ideal $\mathfrak m_x=\mathfrak mR_{\mathfrak m}$, and we write \(\dim_xX:=\dim\mathcal O_{X,x}.\) The Zariski tangent space is \(T_xX:=\Hom_\C(\mathfrak m_x/\mathfrak m_x^2,\C)\), and \(\dim_\C T_xX=\dim_\C\mathfrak m_x/\mathfrak m_x^2\). For every Noetherian local ring one has \(\dim\mathcal O_{X,x}\leq\dim_\C\mathfrak m_x/\mathfrak m_x^2;\) equality is equivalent to $\mathcal O_{X,x}$ being a regular local ring. Since $\C$ is perfect, for a finite-type $\C$-scheme a closed point is smooth over $\C$ if and only if its local ring is regular. The smooth locus is open, and a scheme smooth over $\C$ is reduced. If $X\subseteq E$ is cut out in a finite-dimensional vector space by polynomials $f_1,\ldots,f_r$, then \(T_xX=\{v\in E:(df_i)_x(v)=0\text{ for }1\leq i\leq r\}.\) We use the usual Jacobian criterion in this form: if the differentials of $r$ defining equations are linearly independent at $x$, then the common zero scheme is smooth of codimension $r$ at $x$. Finite-type $\C$-schemes are Jacobson; in particular every nonempty closed subset, and more generally every nonempty constructible subset, contains a closed point. These standard local facts will be used without further comment; see \cite[Chapter~I, \S5 and Chapter~III, \S10]{Hartshorne}.

If an algebraic group $G$ acts algebraically on an affine variety $X$, the induced action on the coordinate algebra is \((g\cdot f)(x)=f(g^{-1}x)\). A \emph{commutative rational $G$-algebra} is a commutative $\C$-algebra $B$ whose underlying $G$-module belongs to $\Rep^{\mathrm{rat}}G$ and for which multiplication and the unit are $G$-equivariant; we write $B^G$ for its invariant subalgebra. When $B$ is finitely generated, $\Spec B$ is the corresponding affine $G$-scheme. For $x\in X$, the stabilizer is \(G_x=\{g\in G:gx=x\}\). Algebraic-group orbits are locally closed, the orbit map induces an isomorphism \(G/G_x\overset{\sim}{\longrightarrow}Gx,\) and \(\dim Gx=\dim G-\dim G_x.\) A \emph{closed orbit} means an orbit that is Zariski closed in $X$. If $H\leq G$ is a closed subgroup and $G/H$ is affine, $\OO(G/H)$ denotes its coordinate algebra.

For an algebraic group $H$, we write $H^\circ$ for its identity component. An algebraic \emph{torus} is a group isomorphic to $(\C^\times)^r$ for some $r$, and a \emph{maximal torus} is one maximal among algebraic tori. For a closed subgroup $H\leq G$, its normalizer is \(N_G(H):=\{g\in G:gHg^{-1}=H\}.\) If $G$ is connected reductive and $T$ is a maximal torus, then \(W_G(T):=N_G(T)/T\) is the finite Weyl group. We use these standard facts about algebraic groups and homogeneous spaces as in \cite[Chapters~1--2]{SpringerLAG}.

Finally, let $G$ be reductive and let $X=\Spec B$ be an affine $G$-variety of finite type. The affine GIT quotient is \(X/\!/G:=\Spec(B^G)\), with quotient morphism \(\pi:X\to X/\!/G\) induced by $B^G\hookrightarrow B$. We use the standard closed-orbit theorem: every fiber of $\pi$ contains a unique closed $G$-orbit. In particular, if $B^G=\C$, then $X$ has a unique closed $G$-orbit; see \cite[Chapter~1, \S1]{MFK}. We also use Matsushima's criterion in the following direction: if $H\leq G$ is closed and $G/H$ is affine, then $H$ is reductive; see \cite{Matsushima} and \cite[Theorem~3.5]{BHC}. These are the affine invariant-theoretic conventions used below.

\subsection{\texorpdfstring{$P(z)$}{P(z)}-tensor products}

Let $V$ be a vertex operator algebra and let $W_1,W_2,W_3$ be generalized $V$-modules. Following \cite[Section~2.1]{CJORY}, a \emph{logarithmic intertwining operator of type $\binom{W_3}{W_1\,W_2}$} is a linear map \(\mathcal Y:W_1\otimes W_2\longrightarrow W_3\{x\}[\log x]\) such that, for $w_1\in W_1$ and $w_2\in W_2$, \(\mathcal Y(w_1,x)w_2=\sum_{k=0}^{K}\sum_{n\in\C}(w_1)^{\mathcal Y}_{n,k}w_2 x^{-n-1}(\log x)^k\) for some $K\ge0$, and the following conditions hold:
\begin{enumerate}[label=\textnormal{(\roman*)}]
\item for every $n\in\C$ and $0\le k\le K$,
\((w_1)^{\mathcal Y}_{n+m,k}w_2=0\) for all sufficiently large $m\in\Z$;
\item for $v\in V$,
\begin{align*}
&x_0^{-1}\delta\!\left(\frac{x_1-x_2}{x_0}\right)
Y_{W_3}(v,x_1)\mathcal Y(w_1,x_2)w_2
-x_0^{-1}\delta\!\left(\frac{x_2-x_1}{-x_0}\right)
\mathcal Y(w_1,x_2)Y_{W_2}(v,x_1)w_2\\
&\hspace{27mm}=
 x_2^{-1}\delta\!\left(\frac{x_1-x_0}{x_2}\right)
\mathcal Y\bigl(Y_{W_1}(v,x_0)w_1,x_2\bigr)w_2;
\end{align*}
\item
\(\mathcal Y(L(-1)w_1,x)=\frac{d}{dx}\mathcal Y(w_1,x).\)
\end{enumerate}
A logarithmic intertwining operator is called an \emph{intertwining operator} if no positive power of $\log x$ occurs. We write \(\mathcal V^{W_3}_{W_1 W_2}\) for the vector space of logarithmic intertwining operators of this type and \(\mathcal N^{W_3}_{W_1 W_2}:=\dim\mathcal V^{W_3}_{W_1 W_2}\) for the corresponding fusion coefficient (fusion rule).

Following \cite[\S2]{CMY}, for a logarithmic intertwining operator $\mathcal Y$ of type $\binom{W_3}{W_1\,W_2}$ we write
\[
\operatorname{Im}\mathcal Y
:=
\operatorname{span}_{\C}
\left\{
(w_1)^{\mathcal Y}_{h,k}w_2
\ \middle|\
w_1\in W_1,\ w_2\in W_2,\ h\in\C,\ k\in\Z_{\ge0}
\right\}.
\]
Thus $\operatorname{Im}\mathcal Y$ is the linear span of all formal coefficients of the series $\mathcal Y(w_1,x)w_2$.

Let $\mathcal C$ be a full subcategory of the category of grading-restricted generalized $V$-modules, with morphisms the $V$-module homomorphisms. For \(W=\bigoplus_{h\in\C}W_{[h]}\in\mathcal C\), write \(\overline W:=\prod_{h\in\C}W_{[h]}\) for its algebraic completion. Fix $z\in\C^\times$. A $P(z)$-intertwining map of type \(\binom{W_3}{W_1 W_2}\) is, in the Huang--Lepowsky--Zhang sense, a linear map \(I:W_1\otimes W_2\longrightarrow\overline{W_3}\) satisfying the $P(z)$ lower-truncation and Jacobi conditions. Equivalently, after choosing a branch of $\log z$, such a map is obtained by evaluating a logarithmic intertwining operator at $x=e^{\log z}$; see \cite[\S2.1]{McRae}. Conversely, if $I$ is a $P(z)$-intertwining map, the corresponding logarithmic intertwining operator is recovered by the $L(0)$-conjugation formula
\begin{equation}\label{eq:Pz-inverse}
\mathcal Y_I(w_1,x)w_2
=
\left(\frac{e^{\log z}}{x}\right)^{-L(0)}
I\left(
\left(\frac{e^{\log z}}{x}\right)^{L(0)}w_1
\otimes
\left(\frac{e^{\log z}}{x}\right)^{L(0)}w_2
\right).
\end{equation}
We shall use only this correspondence, the resulting universal property, and functoriality.

A \emph{$P(z)$-tensor product} of $W_1,W_2\in\mathcal C$ is a pair \(\left(W_1\boxtimes_{P(z)}W_2,\ \boxtimes_{P(z)}\right),\) where $W_1\boxtimes_{P(z)}W_2$ is an object of $\mathcal C$ and \(\boxtimes_{P(z)}:W_1\otimes W_2\longrightarrow\overline{W_1\boxtimes_{P(z)}W_2}\) is a $P(z)$-intertwining map with the following universal property: for every $W_3\in\mathcal C$ and every $P(z)$-intertwining map \(I:W_1\otimes W_2\longrightarrow\overline{W_3}\), there is a unique morphism \(f_I:W_1\boxtimes_{P(z)}W_2\longrightarrow W_3\) such that \(I=\overline{f_I}\circ\boxtimes_{P(z)}.\) Conversely, every morphism \(f:W_1\boxtimes_{P(z)}W_2\longrightarrow W_3\) determines a $P(z)$-intertwining map \(I_f:=\overline f\circ\boxtimes_{P(z)}:W_1\otimes W_2\longrightarrow\overline{W_3}.\) These two assignments are inverse to each other. Thus, throughout the paper, by the $P(z)$-intertwining map \emph{associated with} a morphism $f:W_1\boxtimes_{P(z)}W_2\to W_3$ we mean precisely $I_f$ above. For vectors, we use the standard notation \(w_1\boxtimes_{P(z)}w_2:=\boxtimes_{P(z)}(w_1\otimes w_2)\in\overline{W_1\boxtimes_{P(z)}W_2}.\) The universal property is relative to the chosen category $\mathcal C$. Thus, even when the same modules lie in two different module categories, their $P(z)$-tensor products are not to be identified without justification. This point is emphasized in \cite[Remark~2.3]{McRae}. Whenever we identify a $P(z)$-tensor product in a full subcategory with one computed in a larger ambient category, we will either cite a result guaranteeing this independence or prove that the inclusion preserves the relevant $P(z)$-tensor product.

\subsection{Vertex tensor categories and braided structure}

Following McRae, we use \emph{vertex tensor category} in the abstract sense: the category need not itself be a category of modules for a vertex operator algebra. It is equipped with a distinguished unit object $\mathbf 1_{\mathcal C}$, $P(z)$-tensor-product bifunctors $\boxtimes_{P(z)}$ for all $z\in\C^\times$, and the natural isomorphisms listed below, subject to the usual coherence conditions; see \cite[\S2.1]{McRae}. For an abstract vertex tensor category these $P(z)$-tensor products are part of the categorical structure. When $\mathcal C$ is a suitable full subcategory of grading-restricted generalized $V$-modules, they are realized by the $P(z)$-tensor products of Huang--Lepowsky--Zhang defined above, and $\mathbf 1_{\mathcal C}=V$.

The structure includes the following natural isomorphisms:
\begin{enumerate}[label=\textnormal{(\roman*)}]
\item parallel transport isomorphisms
\(T_\gamma:\boxtimes_{P(z_1)}\overset{\sim}{\longrightarrow}\boxtimes_{P(z_2)}\) for paths $\gamma$ in $\C^\times$ from $z_1$ to $z_2$;

\item $P(z)$-left and right unit isomorphisms
\[
\ell_{P(z);X}:
\mathbf 1_{\mathcal C}\boxtimes_{P(z)}X
\overset{\sim}{\longrightarrow}
X,
\qquad
r_{P(z);X}:
X\boxtimes_{P(z)}\mathbf 1_{\mathcal C}
\overset{\sim}{\longrightarrow}
X;
\]

\item $P(z_1,z_2)$-associativity isomorphisms
\[
\mathcal A_{P(z_1),P(z_2);X,Y,Z}:
X\boxtimes_{P(z_1)}
\bigl(Y\boxtimes_{P(z_2)}Z\bigr)
\overset{\sim}{\longrightarrow}
\bigl(X\boxtimes_{P(z_1-z_2)}Y\bigr)
\boxtimes_{P(z_2)}Z,
\]
for \(|z_1|>|z_2|>|z_1-z_2|>0;\)
\item $P(z)$-braiding isomorphisms
\(\mathcal R_{P(z);X,Y}:X\boxtimes_{P(z)}Y\overset{\sim}{\longrightarrow}Y\boxtimes_{P(-z)}X.\)
\end{enumerate}
These data satisfy the corresponding unit, associativity, braiding, parallel-transport, pentagon, and hexagon coherence conditions. We refer to \cite[\S2.1]{McRae} for the complete formulation.

The vertex tensor structure canonically determines an ordinary braided tensor category structure. Following \cite[\S2.1]{McRae}, throughout this paper the unadorned tensor product means \(X\boxtimes Y:=X\boxtimes_{P(1)}Y.\) The tensor unit is $\mathbf 1_{\mathcal C}$ (equal to $V$ in the VOA-module realization), the unit isomorphisms are the $P(1)$-unit isomorphisms, and the ordinary braiding is \(\mathcal R_{X,Y}=T_{-1\to1;Y,X}\circ\mathcal R_{P(1);X,Y}.\) The associativity isomorphism is obtained from the $P(z_1,z_2)$-associativity isomorphism by parallel transport. Thus a formula involving $\boxtimes$ without a $P(z)$ subscript always refers to the braided tensor product induced from the $P(1)$-tensor product.

\subsection{Vertex tensor functors and the compact-orbifold theorem for ordinary VOAs}
\label{subsec:McRae-VOA}

Let $\mathcal C_1$ and $\mathcal C_2$ be vertex tensor categories. A \emph{vertex tensor functor} consists of a functor $F:\mathcal C_1\to\mathcal C_2$, an isomorphism of tensor units \(\varphi:\mathbf 1_{\mathcal C_2}\overset{\sim}{\longrightarrow}F(\mathbf 1_{\mathcal C_1})\), and, for every $z\in\C^\times$, natural isomorphisms \(J_{P(z);M,N}:F(M)\boxtimes_{P(z)}F(N)\overset{\sim}{\longrightarrow}F(M\boxtimes_{P(z)}N),\) compatible with parallel transport, unit, associativity, and braiding; see \cite[\S2.1]{McRae} and \cite[\S3.6]{CKM}. Such a functor induces, at $z=1$, a braided tensor functor between the underlying braided tensor categories. If its underlying functor is an equivalence, we call it a \emph{vertex tensor equivalence}. If, in addition, the induced braided tensor equivalence is symmetric, we call it a \emph{symmetric vertex tensor equivalence}. The adjective ``symmetric'' therefore refers to the underlying braided tensor category; it is not an additional kind of $P(z)$-tensor product.

We now record, in the form used later, the specialization of McRae's compact-orbifold results to an ordinary vertex operator algebra. We regard the ordinary VOA case as the specialization with trivial abelian grading group, so the abelian $3$-cocycle and the corresponding cocycle modification of the representation category disappear from the general statements of \cite{McRae}.

Let $U$ be a simple vertex operator algebra and let $K$ be a compact Lie group acting on $U$ by automorphisms. We say that the action is \emph{continuous} if the induced action on each finite-dimensional homogeneous subspace $U_n$ is continuous. Write $\widehat K$ for the set of isomorphism classes of irreducible finite-dimensional continuous complex $K$-modules, and choose a representative $M_\chi$ for each $\chi\in\widehat K$. The compact-orbifold Schur--Weyl decomposition is
\begin{equation}\label{eq:McRae-SW-VOA}
 U\cong
 \bigoplus_{\chi\in\widehat K}M_\chi\otimes U_\chi
\end{equation}
as a $K\times U^K$-module, where every $U_\chi$ is a nonzero irreducible $U^K$-module and the modules $U_\chi$ are pairwise nonisomorphic; see \cite[\S3.1]{McRae}. Let \(\mathcal C_U:=\left\{\text{finite direct sums of the }U_\chi,\ \chi\in\widehat K\right\}.\) Thus $\mathcal C_U$ is the full semisimple subcategory of $U^K$-modules generated by the multiplicity spaces in \eqref{eq:McRae-SW-VOA}.

For a finite-dimensional continuous $K$-module $M$, define \(\Phi_K(M):=(M\otimes U)^K\), where $K$ acts diagonally, and for a $K$-homomorphism $f:M\to N$ set \(\Phi_K(f):=(f\otimes\id_U)\big|_{(M\otimes U)^K}.\) Then \(U_\chi\cong\Phi_K(M_\chi^*)\) and \(\Phi_K:\Rep_{\mathrm{fd}}K\overset{\sim}{\longrightarrow}\mathcal C_U\) is an equivalence of semisimple abelian categories; this is \cite[Proposition~3.5]{McRae}. Here $\Rep_{\mathrm{fd}}K$ denotes the category of finite-dimensional continuous complex representations of $K$.

For finite-dimensional continuous $K$-modules $M_1,M_2,M_3$, there is also a canonical linear map
\begin{equation}\label{eq:McRae-intertwiner-map}
 J^{M_3}_{M_1,M_2}:
 \Hom_K(M_1\otimes M_2,M_3)
 \longrightarrow
 \mathcal V^{\Phi_K(M_3)}_{\Phi_K(M_1)\,\Phi_K(M_2)}.
\end{equation}
For $f\in\Hom_K(M_1\otimes M_2,M_3)$, the corresponding intertwining operator is the restriction to $K$-fixed points of \(\mathcal Y_f(m_1\otimes u_1,x)(m_2\otimes u_2)=f(m_1\otimes m_2)\otimes Y_U(u_1,x)u_2.\) The map \eqref{eq:McRae-intertwiner-map} is always injective \cite[\S3.2]{McRae}.

\begin{thm}
\label{thm:McRae-VOA-orbifold}
Retain the preceding assumptions on the simple vertex operator algebra $U$ and the compact Lie group $K$.

\begin{enumerate}[label=\textnormal{(\roman*)}]
\item
If the maps \eqref{eq:McRae-intertwiner-map} are surjective for all finite-dimensional continuous $K$-modules $M_1,M_2,M_3$, then $\mathcal C_U$ admits a vertex tensor category structure. In this structure the maps \eqref{eq:McRae-intertwiner-map} are isomorphisms; see \cite[Theorem~4.1]{McRae}.

\item
Suppose that there is a vertex tensor category $\mathcal C$ of grading-restricted generalized $U^K$-modules containing every multiplicity module $U_\chi$. Then, for every $z\in\C^\times$, $\mathcal C_U$ is closed under the $P(z)$-tensor product computed in $\mathcal C$. Since $\mathcal C_U$ is full and contains the tensor unit $U^K=\Phi_K(\C)$, the $P(z)$-tensor products and the unit, parallel-transport, associativity, and braiding isomorphisms of $\mathcal C$ restrict to $\mathcal C_U$. Thus $\mathcal C_U$ is a full vertex tensor subcategory of $\mathcal C$.

McRae first proves that \(\bigl(\Phi_K,\{J_{P(z)}\}_{z\in\C^\times},\varphi_1\bigr)\) is a lax vertex tensor functor; see \cite[Theorem~4.5]{McRae}. By \cite[Corollary~4.8]{McRae}, the constraint at $z=1$ is an isomorphism and the induced functor is a braided tensor equivalence onto $\mathcal C_U$. Compatibility with parallel transport then implies that, for every $z\in\C^\times$ and $M,N\in\Rep_{\mathrm{fd}}K$, the canonical map \(J_{P(z);M,N}: \Phi_K(M)\boxtimes^{\mathcal C}_{P(z)}\Phi_K(N) \overset{\sim}{\longrightarrow} \Phi_K(M\otimes N)\) is an isomorphism; this is precisely the observation in \cite[Remark~4.10]{McRae}. Thus the $J_{P(z)}$ are natural isomorphisms compatible with the unit, parallel transport, associativity, and braiding, and \(\Phi_K:\Rep_{\mathrm{fd}}K\overset{\sim}{\longrightarrow}\mathcal C_U\) is a vertex tensor equivalence.

In the ordinary-VOA specialization the abelian grading group is trivial and $F=\Omega=1$. McRae's vertex tensor structure on $\Rep_{\mathrm{fd}}K$ therefore has \(M\boxtimes_{P(z)}N=M\otimes N\) for every $z\in\C^\times$, trivial parallel transport, and the usual unit, associativity, and braiding isomorphisms of finite-dimensional $K$-modules. Its underlying braided tensor category is consequently the usual symmetric category $\Rep_{\mathrm{fd}}K$. Hence the vertex tensor equivalence above induces a symmetric braided tensor equivalence; in the terminology fixed above, it is a symmetric vertex tensor equivalence.

Moreover, this vertex tensor structure on $\mathcal C_U$ is independent of the choice of the ambient vertex tensor category. More precisely, if $\mathcal C$ and $\mathcal D$ are two admissible ambient vertex tensor categories containing the same multiplicity modules, then, for $X=\Phi_K(M)$ and $Y=\Phi_K(N)$, the two $P(z)$-tensor products are canonically identified through
\[
 X\boxtimes^{\mathcal C}_{P(z)}Y
 \xrightarrow{\ J^{\mathcal C}_{P(z);M,N}\ }
 \Phi_K(M\otimes N)
 \xleftarrow{\ J^{\mathcal D}_{P(z);M,N}\ }
 X\boxtimes^{\mathcal D}_{P(z)}Y.
\]
The independence statement at the level of tensor products is \cite[Remark~4.9]{McRae}, while the full $P(z)$-tensor equivalence is summarized in \cite[Remark~4.10]{McRae}.
\end{enumerate}
\end{thm}

\begin{cor}\label{cor:McRae-VOA-fusion}
Under the hypotheses of Theorem~\ref{thm:McRae-VOA-orbifold}(ii), for all $M_1,M_2,M_3\in\Rep_{\mathrm{fd}}K$ one has \(\mathcal V^{\Phi_K(M_3)}_{\Phi_K(M_1) \Phi_K(M_2)}\cong\Hom_K(M_1\otimes M_2,M_3).\) Hence tensor-product multiplicities in $\mathcal C_U$ are exactly the Clebsch--Gordan multiplicities for $K$. In particular, if \(M_1\otimes M_2\cong\bigoplus_i N_i^{\oplus a_i}\), then \(\Phi_K(M_1)\boxtimes\Phi_K(M_2)\cong\bigoplus_i\Phi_K(N_i)^{\oplus a_i}.\)
\end{cor}

\begin{proof}
Under the hypotheses of Theorem~\ref{thm:McRae-VOA-orbifold}(ii), the canonical maps \(J^{M_3}_{M_1,M_2}: \Hom_K(M_1\otimes M_2,M_3) \longrightarrow \mathcal V^{\Phi_K(M_3)}_{\Phi_K(M_1)\,\Phi_K(M_2)}\) are surjective; see \cite[\S4]{McRae}. Since these maps are always injective by \cite[\S3.2]{McRae}, they are isomorphisms. This proves the first assertion. The statements about tensor-product multiplicities and decompositions follow immediately from the symmetric tensor equivalence in Theorem~\ref{thm:McRae-VOA-orbifold}(ii).
\end{proof}

\begin{rmk}\label{rmk:compact-complexification}
We use the standard notation $SO(3)$ and $O(2)$ for the compact real orthogonal groups, and $SO_3(\C)$ and $O_2(\C)$ for their complex algebraic counterparts. If $G$ is a complex reductive algebraic group and $K\subset G$ is a maximal compact subgroup, restriction gives the usual symmetric tensor equivalence $ \Res^G_K:\Rep G\overset{\sim}{\longrightarrow}\Rep_{\mathrm{fd}}K. $ Transport McRae's vertex tensor structure on $\Rep_{\mathrm{fd}}K$ along this equivalence. Then the composite of $\Res^G_K$ with the compact-orbifold functor is a symmetric vertex tensor equivalence from $\Rep G$. We shall use this convention without further comment for $ SO(3)\subset SO_3(\C)\cong\PSLtwo $ and $ O(2)\subset O_2(\C). $
\end{rmk}

\subsection{The \texorpdfstring{$C_1$}{C1}-cofinite Virasoro tensor category}
\label{subsec:CJORY-package}

We record here the Virasoro conventions and results from Creutzig--Jiang--Orosz Hunziker--Ridout--Yang \cite[Section~2.2]{CJORY} that will be used later. The Virasoro algebra is \(\Vir=\bigoplus_{n\in\Z}\C L_n\oplus\C\mathbf c,\) with \([L_m,L_n]=(m-n)L_{m+n}+\frac{m^3-m}{12}\delta_{m+n,0}\mathbf c,[\Vir,\mathbf c]=0.\) Put \(\Vir_{\ge0}:=\bigoplus_{n\ge0}\C L_n\oplus\C\mathbf c\). For $c,h\in\C$, let $\C\one_{c,h}$ be the one-dimensional $\Vir_{\ge0}$-module on which \(L_0\one_{c,h}=h\one_{c,h},\mathbf c\one_{c,h}=c\one_{c,h},L_n\one_{c,h}=0(n>0).\) The \emph{Virasoro Verma module} of central charge $c$ and highest weight $h$ is \(V(c,h):=U(\Vir)\otimes_{U(\Vir_{\ge0})}\C\one_{c,h}.\) It has a unique maximal proper submodule; its irreducible quotient is denoted by $L(c,h)$. More generally, a highest-weight Virasoro module of central charge $c$ and highest weight $h$ is a cyclic $\Vir$-module generated by a vector $v$ satisfying \(\mathbf c v=cv\), \(L_0v=hv\), and \(L_nv=0\) for all $n>0$. A nonzero $L_0$-eigenvector killed by all $L_n$, $n>0$, is called a \emph{singular vector}; in a highest-weight module it is \emph{nontrivial} if it is not proportional to the chosen cyclic highest-weight vector.

When $h=0$, $L_{-1}\one_{c,0}$ is singular and \(\Vir_c:=V(c,0)/\langle L_{-1}\one_{c,0}\rangle\) is the universal Virasoro vertex operator algebra of central charge $c$. For $t\in\C^\times$ and $r,s\in\Z_{\ge1}$, set \(c(t):=13-6t-6t^{-1},h_{r,s}(t):=\frac{r^2-1}{4}t-\frac{rs-1}{2}+\frac{s^2-1}{4}t^{-1}.\) For $c=c(t)$, define \(H_c:=\{h\in\C\mid V(c,h)\text{ is reducible}\}.\) By the Feigin--Fuchs reducibility criterion recalled in \cite[Section~2.2]{CJORY}, \(H_c=\{h_{r,s}(t)\mid r,s\ge1\}\). Let $\mathcal O_c^{\mathrm{fin}}$ denote the category of finite-length $\Vir_c$-modules whose irreducible composition factors are $L(c,h)$ with $h\in H_c$.

\begin{prop}
\label{prop:CJORY-package}
The following statements hold.
\begin{enumerate}[label=\textnormal{(\roman*)}]
\item Every highest-weight Virasoro module of central charge $c$ is an
$\Vir_c$-module. Moreover, \(\Vir_c=L(c,0)\) unless \(c=c_{p,q}=1-\frac{6(p-q)^2}{pq}\) for coprime integers $p,q\ge2$; see \cite[Section~2.2 and Corollary~2.2.5(4)]{CJORY}.

\item The irreducible module $L(c,h)$ is $C_1$-cofinite if and only if
$h\in H_c$; more generally, a highest-weight $\Vir_c$-module of highest weight $h$, necessarily a quotient of $V(c,h)$, is $C_1$-cofinite if and only if it is a proper quotient of $V(c,h)$; see \cite[Corollary~2.2.7]{CJORY}.

\item The category $\mathcal O_c^{\mathrm{fin}}$ coincides with the category of lower-bounded $C_1$-cofinite generalized $\Vir_c$-modules; see \cite[Theorem~3.1.4]{CJORY}. Since its objects have finite length with composition factors $L(c,h)$, they are grading restricted. Moreover, $\mathcal O_c^{\mathrm{fin}}$ is abelian.

\item The Huang--Lepowsky--Zhang construction carried out in
\cite[Sections~4.1--4.2]{CJORY} endows $\mathcal O_c^{\mathrm{fin}}$ with a vertex tensor category structure. In particular, $\mathcal O_c^{\mathrm{fin}}$ is closed under $P(z)$-tensor products for every $z\in\mathbb C^\times$, and the associated ordinary tensor category is braided; see \cite[Corollary~4.1.8 and Theorem~4.2.6]{CJORY}.
\end{enumerate}
\end{prop}

At central charge one, set \(\mathcal O_1:=\mathcal O_1^{\mathrm{fin}}\). Since $c_{p,q}=1$ forces $p=q$, whereas the parameters $p,q\ge2$ in Proposition~\ref{prop:CJORY-package}(i) are required to be coprime, central charge one is not exceptional. Hence \(\Vir_1=L(1,0)\). For every vertex operator algebra $V$ of central charge one, the universal property of \(\Vir_1\) gives a canonical homomorphism \(\Vir_1\longrightarrow V,\one_{1,0}\longmapsto\one,\omega\longmapsto\omega_V\), whose image is the Virasoro subVOA generated by \(\omega_V\). Since \(\Vir_1=L(1,0)\) is simple and this homomorphism is nonzero, it is injective. We therefore identify \(\operatorname{Vir}(\omega_V)=L(1,0)\subset V\) whenever \(V\) has central charge one. Thus \(\Ovir\) denotes the vertex tensor category of grading-restricted $C_1$-cofinite generalized $L(1,0)$-modules. Also, in the standard parametrization \(h_{r,s}(1)=\frac{(r-s)^2}{4},\) so for every $m\ge0$, \(m^2=h_{2m+1,1}(1)\in H_1.\) Hence
\begin{equation}\label{eq:square-C1-cofinite}
L(1,m^2)\in\Ovir\qquad(m\ge0),
\end{equation}
and every $L(1,m^2)$ is $C_1$-cofinite.

For later use, put \(X_m:=L(1,m^2)(m\ge0).\) We shall also use the following standard structure of the degenerate central-charge-one Verma modules. For $N\in\Z_{\ge0}$, the weight-$h+N$ subspace of $V(c,h)$ is called its \emph{level-$N$} subspace. A bilinear form $\langle\ ,\ \rangle$ on a Virasoro module is \emph{contravariant} if \(\langle L(n)x,y\rangle=\langle x,L(-n)y\rangle\ (n\in\Z).\) On $V(c,h)$ there is a unique contravariant bilinear form $\langle\ ,\ \rangle_{\rm Sh}$ normalized by $\langle\one_{c,h},\one_{c,h}\rangle_{\rm Sh}=1$; this is the normalized \emph{Shapovalov form}. Its radical is the unique maximal proper submodule of $V(c,h)$. Consequently, every contravariant bilinear form on $V(c,h)$ that is nonzero on the highest-weight line is a nonzero scalar multiple of $\langle\ ,\ \rangle_{\rm Sh}$.

\begin{prop}\label{prop:c1-verma}
At central charge one, $V(1,h)$ is reducible if and only if \(h=\frac{n^2}{4}\) for some \(n\in\Z_{\ge0}.\) Moreover, for every $m\ge0$, the Verma module $V(1,m^2)$ has, up to scalar, a unique singular vector at level $2m+1$. This is its first nontrivial singular vector, and it generates the unique maximal proper submodule, which is isomorphic to $V(1,(m+1)^2)$. Equivalently, there is an exact sequence \(0\longrightarrow V(1,(m+1)^2)\longrightarrow V(1,m^2)\longrightarrow X_m\longrightarrow0\). The embedded Verma module is also the radical of the normalized Shapovalov form on $V(1,m^2)$.
\end{prop}

\begin{proof}
The reducibility criterion follows from $H_1=\{h_{r,s}(1):r,s\ge1\}$ and $h_{r,s}(1)=(r-s)^2/4$, recalled above from the Feigin--Fuchs criterion. The more precise embedding statement is \cite[Proposition~2.1]{Milas}, with the parameter there specialized to the even integer $2m$; see also the Feigin--Fuchs embedding description in \cite{KR}. The singular level is $2m+1$, so its absolute conformal weight is \((m+1)^2.\) The final assertion follows from the Shapovalov-form convention above and the description of the unique maximal proper submodule.
\end{proof}

\begin{cor}\label{cor:c1-verma-low-weight}
Let $m\ge1$, and let $W$ be a nonzero highest-weight quotient of $V(1,m^2)$ whose highest-weight vector is nonzero. Then the canonical map \(V(1,m^2)\twoheadrightarrow W\) is injective in every absolute conformal weight $n<(m+1)^2$. Consequently, \(\dim W_n=\dim(X_m)_n\bigl(n<(m+1)^2\bigr).\)
\end{cor}

\begin{proof}
Let $ \pi:V(1,m^2)\twoheadrightarrow W $ be the quotient map. Suppose, for contradiction, that \(\ker\pi\) is nonzero in some conformal weight strictly below \((m+1)^2\). Since the highest-weight vector survives in \(W\), the level-zero part of \(\ker\pi\) is zero. Choose a nonzero homogeneous vector $ x\in\ker\pi $ of minimal positive level \(N\). Since \(x\) has conformal weight \(m^2+N<(m+1)^2\), we have $ 0<N<2m+1. $ We claim that \(x\) is singular. Since \(\ker\pi\) is a Virasoro submodule, \(L(n)x\in\ker\pi\) for every \(n>0\). If \(0<n<N\), then \(L(n)x\) has level \(N-n\), which is positive and strictly smaller than \(N\); hence \(L(n)x=0\) by the minimality of \(N\). If \(n=N\), then \(L(N)x\) lies in the level-zero part of \(\ker\pi\), which is zero. Finally, if \(n>N\), then $ L(n)x\in V(1,m^2)_{N-n}=0. $ Thus $L(n)x=0(n>0),$ so \(x\) is a nontrivial singular vector of level \(N<2m+1\). This contradicts Proposition~\ref{prop:c1-verma}, according to which the first nontrivial singular vector in \(V(1,m^2)\) occurs at level \(2m+1\), equivalently at conformal weight $ (m+1)^2. $ Therefore \(\pi\) is injective in every conformal weight strictly below \((m+1)^2\). Applying this conclusion to the canonical quotient $ V(1,m^2)\twoheadrightarrow X_m=L(1,m^2) $ shows that both \(W\) and \(X_m\) have the same graded dimensions as \(V(1,m^2)\) below conformal weight \((m+1)^2\). Hence $ \dim W_n=\dim (X_m)_n$ for $n<(m+1)^2. $
\end{proof}

When we say below that a compact-orbifold functor is a \emph{symmetric vertex tensor equivalence}, we use the definition fixed above: it is a vertex tensor equivalence whose induced braided tensor equivalence is symmetric. In particular, this records the full $P(z)$-tensor data and not merely an isomorphism of fusion rings.

\begin{cor}
\label{cor:McRae-A1-SO3}
Let \(Q=\Z\alpha\) with \(\langle\alpha,\alpha\rangle=2\), so that \(V_Q=V_{A_1}\), and let \(K=SO(3)\) be the standard compact automorphism group obtained from the \(SU(2)\)-action on the \(\mathfrak{sl}_2\) weight-lattice abelian intertwining algebra. Then \(V_{A_1}^{K}=L(1,0).\) If \(W_{2m}\) denotes the irreducible \(SO(3)\)-module of dimension \(2m+1\), then \(V_{A_1}\cong\bigoplus_{m\ge0}W_{2m}\otimes L(1,m^2)\) as a \(K\times L(1,0)\)-module. The full semisimple category \(\mathcal C_{V_{A_1}}=\left\{\text{finite direct sums of }L(1,m^2),\ m\ge0\right\}\) carries the compact-orbifold vertex tensor category structure, and its underlying braided tensor category is symmetric. Under the compact-orbifold functor there is a symmetric tensor equivalence \(\Rep_{\mathrm{fd}}SO(3) \overset{\sim}{\longrightarrow} \mathcal C_{V_{A_1}}, W_{2m}\longmapsto L(1,m^2)\).
\end{cor}

\begin{proof}
McRae's central-charge-one example identifies \(V_Q^{SO(3)}=L(1,0)\) and, using the weight-lattice abelian intertwining algebra \(V_P\), identifies the \((n+1)\)-dimensional irreducible \(SU(2)\)-module \(V(n)\) with the multiplicity module \(L(1,n^2/4)\); see \cite[Example~4.11]{McRae}. The irreducible \(SU(2)\)-representations occurring in \(V_Q\) are precisely those of even highest weight \(n=2m\), equivalently the odd-dimensional ones. Hence \(V_Q\cong\bigoplus_{m\ge0}W_{2m}\otimes L(1,m^2)\) as an \(SO(3)\times L(1,0)\)-module.

By \eqref{eq:square-C1-cofinite}, every multiplicity module \(L(1,m^2)\) belongs to the ambient vertex tensor category \(\Ovir\). We may therefore apply Theorem~\ref{thm:McRae-VOA-orbifold}(ii) with \(U=V_Q\), \(K=SO(3)\), and \(\mathcal C=\Ovir\). It follows that \(\mathcal C_{V_{A_1}}\) is a full vertex tensor subcategory of \(\Ovir\) and that the compact-orbifold functor is a symmetric vertex tensor equivalence \(\Rep_{\mathrm{fd}}SO(3) \overset{\sim}{\longrightarrow} \mathcal C_{V_{A_1}}, W_{2m}\longmapsto L(1,m^2)\). In particular, the induced equivalence of the underlying braided tensor categories is symmetric.
\end{proof}

\subsection{The \texorpdfstring{$C_1$}{C1}-cofinite \texorpdfstring{$M(1)^+$}{M(1)+} tensor category}
\label{subsec:ALY-package}

We record the rank-one Heisenberg-orbifold notation and the results of Adamovi\'c--Lin--Yang that will be used later; compare \cite[Section~2.1]{ALY}. Let $\mathfrak h=\C h$ be a one-dimensional complex vector space equipped with a nondegenerate symmetric bilinear form $\langle\,\cdot\,,\,\cdot\,\rangle$, normalized by $\langle h,h\rangle=1$. Put
\[
\widehat{\mathfrak h}
=
\mathfrak h\otimes\C[t,t^{-1}]\oplus\C K,
\qquad
[a(m),b(n)]
=
m\langle a,b\rangle\delta_{m+n,0}K,
\]
where $a(n):=a\otimes t^n$. We identify $\mathfrak h$ with $\mathfrak h^*$ by the bilinear form. For $\lambda\in\mathfrak h$, let $\C e^\lambda$ be the one-dimensional $\mathfrak h\otimes\C[t]\oplus\C K$-module on which \(a(n)e^\lambda=0\) for \(n>0\), \(a(0)e^\lambda=\langle a,\lambda\rangle e^\lambda\), and \(Ke^\lambda=e^\lambda\). For $\lambda\ne0$, define the Heisenberg Fock module \(M(1,\lambda):= U(\widehat{\mathfrak h}) \otimes_{U(\mathfrak h\otimes\C[t]\oplus\C K)} \C e^\lambda\). For zero charge, the same induced construction is the rank-one Heisenberg vertex operator algebra, which we denote by $M(1)$; its vacuum is $\one$ and its conformal vector is $\omega=\frac12h(-1)^2\one$. For $\lambda\ne0$, the vector $e^\lambda$ is the lowest conformal-weight vector of $M(1,\lambda)$, of conformal weight $\langle\lambda,\lambda\rangle/2$.

The isometry $-\id_{\mathfrak h}$ lifts to an involutive automorphism $\theta$ of $M(1)$, explicitly
\[
\theta\bigl(a_1(-n_1)\cdots a_r(-n_r)\one\bigr)=(-1)^r a_1(-n_1)\cdots a_r(-n_r)\one
\]
for $a_i\in\mathfrak h$ and $n_i\in\Z_{>0}$. We write \(M(1)^+:=M(1)^{\langle\theta\rangle}\) and \(M(1)^-:=\{v\in M(1)\mid\theta v=-v\}\). Thus $M(1)^+$ is the $\mathbb Z_2$-orbifold vertex operator algebra and $M(1)^-$ is its $(-1)$-eigenspace. For $\lambda\ne0$, restriction of the Fock module $M(1,\lambda)$ to $M(1)^+$ is irreducible, and $M(1,\lambda)\cong M(1,-\lambda)$ as $M(1)^+$-modules.

For completeness, we also fix the twisted-module notation used in the classification of irreducible $M(1)^+$-modules. Let \(\widehat{\mathfrak h}[\theta]=\mathfrak h\otimes t^{1/2}\C[t,t^{-1}]\oplus\C K\), with the same Heisenberg commutation relation, now for $m,n\in\frac12+\Z$, and put $\widehat{\mathfrak h}[\theta]_{\ge0} :=\mathfrak h\otimes t^{1/2}\C[t]\oplus\C K$. Let $\C\one_{\rm tw}$ be the one-dimensional $\widehat{\mathfrak h}[\theta]_{\ge0}$-module on which all positive half-integer modes act as zero and $K$ acts as the identity. The canonical $\theta$-twisted Heisenberg module is
\[
M(1)(\theta)
:=
U(\widehat{\mathfrak h}[\theta])
\otimes_{U(\widehat{\mathfrak h}[\theta]_{\ge0})}
\C\one_{\rm tw}
\cong
S\!\left(\mathfrak h\otimes t^{-1/2}\C[t^{-1}]\right)
\]
as a vector space. The parity operator
\[
\theta\bigl(a_1(-r_1)\cdots a_s(-r_s)\one_{\rm tw}\bigr)
=
(-1)^s a_1(-r_1)\cdots a_s(-r_s)\one_{\rm tw},
\qquad r_i\in\tfrac12+\Z_{\ge0},
\]
commutes with the $M(1)^+$-action; its $\pm1$ eigenspaces are denoted by $M(1)(\theta)^\pm$.

The irreducible grading-restricted $M(1)^+$-modules are precisely \(M(1)^\pm\), \(M(1,\lambda)\) for \(\lambda\in\mathfrak h\setminus\{0\}\), and \(M(1)(\theta)^\pm\); see the original rank-one classification \cite{DongNagatomo} and the summary in \cite[Section~2.1]{ALY}. Their lowest conformal weights are, respectively, \(0\), \(1\), \(\frac{\langle\lambda,\lambda\rangle}{2}\), \(\frac1{16}\), and \(\frac9{16}\). In particular, the only non-Fock irreducible modules have lowest conformal weights $0$, $1$, $1/16$, and $9/16$.

We finally make explicit the highest-weight convention used below. For a grading-restricted generalized $M(1)^+$-module $M$, its Zhu top space is
\[
\Omega(M):=\left\{w\in M\middle|a_nw=0\text{ for every homogeneous }a\in M(1)^+\text{ and }n>\wt(a)-1\right\}.
\] By Zhu's top-space theorem \cite{Zhu}, the Zhu algebra $A(M(1)^+)$ acts on $\Omega(M)$ by zero modes $o(a):=a_{\wt(a)-1}$. Following the usage in \cite[Section~3.4]{ALY}, we call $M$ a \emph{highest-weight $M(1)^+$-module} if it is generated by a nonzero vector $v\in\Omega(M)$ such that $\C v$ is a one-dimensional irreducible $A(M(1)^+)$-module. We call $v$ a \emph{highest-weight vector}; its $L(0)$-eigenvalue is the conformal weight of the highest weight.

\begin{prop}\label{prop:ALY-package}
Let $\mathcal C_1(M(1)^+)$ be the category of grading-restricted $C_1$-cofinite generalized $M(1)^+$-modules. Then:
\begin{enumerate}[label=\textnormal{(\roman*)}]
\item The irreducible $C_1$-cofinite $M(1)^+$-modules are precisely
\(M(1)^\pm\text{ and }M(1,\lambda)\ (\lambda\ne0)\). In particular, every module of these three types is an object of $\mathcal C_1(M(1)^+)$; see \cite[Corollary~3.4 and Theorems~3.5--3.6]{ALY}.

\item The category $\mathcal C_1(M(1)^+)$ is exactly the category of
finite-length $M(1)^+$-modules whose simple composition factors are of the types in \textnormal{(i)}; see \cite[Theorem~4.2]{ALY}. Moreover, it carries a vertex tensor category structure and the induced braided tensor category structure; see \cite[Theorem~4.3]{ALY}.

\item Let $M$ be a highest-weight $M(1)^+$-module and let $\Delta$ be
the conformal weight of a highest-weight vector. If $\Delta\ne0$, then $M$ is irreducible. If $\Delta=0$, then either $M\cong M(1)^+$ or $M$ has length two and fits into a non-split exact sequence \(0\longrightarrow M(1)^-\longrightarrow M\longrightarrow M(1)^+\longrightarrow0.\) This is \cite[Theorem~3.8]{ALY}.

\item For nonzero charges $\lambda,\mu\in\mathfrak h$ one has
\begin{align*}
M(1)^-\boxtimes M(1)^-&\cong M(1)^+,\\
M(1)^-\boxtimes M(1,\lambda)&\cong M(1,\lambda),\\
M(1,\lambda)\boxtimes M(1,\mu)
&\cong M(1,\lambda+\mu)\oplus M(1,\lambda-\mu),
\end{align*}
where, when $\lambda+\mu=0$ or $\lambda-\mu=0$, the corresponding zero-charge summand is the Heisenberg VOA $M(1)$, which restricts to $M(1)^+\oplus M(1)^-$ as an $M(1)^+$-module. These are the fusion-product decompositions of \cite[Theorem~4.6]{ALY}.
\end{enumerate}
\end{prop}

\subsection{Rank-one lattice vertex operator algebras}
\label{subsec:rank-one-lattice}

We record the concrete rank-one lattice realization and isolate the precise consequences used later. We follow the standard lattice-VOA construction of Frenkel--Lepowsky--Meurman \cite[Chapter~8]{FLM}; for the rank-one specialization, including the involution induced by the isometry $-1$, see \cite[Section~2]{DG98}. See also \cite[Sections~6.4--6.5]{LepowskyLi} for a textbook account of the lattice construction.

Let \(L=\Z\alpha,\ \langle\alpha,\alpha\rangle=2k,\ k\in\Z_{>0}.\) Fix a normalized lattice $2$-cocycle $\varepsilon$ for the standard central extension of $L$, so that \(\varepsilon(\beta,\gamma)\varepsilon(\gamma,\beta)^{-1} =(-1)^{\langle\beta,\gamma\rangle} (\beta,\gamma\in L)\), and let $\C_\varepsilon[L]$ be the corresponding twisted group algebra, with basis $e^\beta$ ($\beta\in L$). Since $L$ has rank one and is even, $\langle L,L\rangle\subset 2\Z$; in particular, the cocycle may be chosen bimultiplicative. The associated lattice vertex operator algebra is \(V_L=M(1)\otimes\C_\varepsilon[L]=\bigoplus_{r\in\Z}M(1,r\alpha)\), where \(M(1,r\alpha):=M(1)\otimes e^{r\alpha}\), where, for $r\ne0$, the last notation agrees with the Heisenberg Fock-module notation of Subsection~\ref{subsec:ALY-package}. Its conformal vector is the Heisenberg conformal vector, and hence \(\wt(e^{r\alpha})=\frac{\langle r\alpha,r\alpha\rangle}{2}=kr^2.\) On a lattice top vector, the standard lattice vertex-operator formula specializes to
\begin{equation*}
Y(e^\beta,z)e^\gamma
=
\varepsilon(\beta,\gamma)
 z^{\langle\beta,\gamma\rangle}
 \exp\!\left(\sum_{n\ge1}\frac{\beta(-n)}{n}z^n\right)
 e^{\beta+\gamma}
\qquad(\beta,\gamma\in L).
\end{equation*}
This is the usual lattice formula of \cite[Chapter~8]{FLM}; it is the only part of the explicit vertex-operator construction used below.

Let $\theta$ be the standard lift of the lattice isometry $-1$. In the rank-one normalization of \cite[Section~2]{DG98}, it acts by
\[
\theta(e^\beta)=e^{-\beta},\qquad
\theta\bigl(h_1(-n_1)\cdots h_s(-n_s)e^\beta\bigr)
=(-1)^s h_1(-n_1)\cdots h_s(-n_s)e^{-\beta},
\]
for $h_i\in\mathfrak h$ and $n_i>0$. Thus the restriction of $\theta$ to $M(1)$ is exactly the involution fixed in Subsection~\ref{subsec:ALY-package}. Put \(V_L^+:=V_L^{\langle\theta\rangle},\ F_r:=e^{r\alpha}+e^{-r\alpha}\ (r\ge1),\) and \(V_L^+[r]:=\bigl(M(1,r\alpha)\oplus M(1,-r\alpha)\bigr)^\theta\ (r\ge1).\)

\begin{prop}[Rank-one lattice package]\label{prop:rank-one-lattice-package}
With the notation above, the following hold.
\begin{enumerate}[label=\textnormal{(\roman*)}]
\item There is an $M(1)^+$-module decomposition
\(V_L^+=M(1)^+\oplus\bigoplus_{r\ge1}V_L^+[r].\) For every $r\ge1$, the map \(\varphi_r:M(1,r\alpha)\longrightarrow V_L^+[r],\ u\longmapsto u+\theta u,\) is an $M(1)^+$-module isomorphism. Consequently, \(V_L^+\cong M(1)^+\oplus\bigoplus_{r\ge1}M(1,r\alpha)\) as an $M(1)^+$-module. The lowest-weight space of $V_L^+[r]$ is $\C F_r$, and its conformal weight is $kr^2$.

\item As an $L(1,0)$-module,
\(M(1)^+\cong\bigoplus_{j\ge0}L(1,4j^2).\) In particular, the first nonvacuum Virasoro summand begins in conformal weight $4$. This is exactly \cite[Theorem~2.7(1)]{DG98}.

\item The vertex operator algebra $V_L^+$ is generated by $M(1)^+$ together
with its first nonzero charge sector:
\begin{equation}\label{eq:VLplus-generated-first-charge}
V_L^+=\langle M(1)^+,F_1\rangle_{\mathrm{VOA}}
     =\langle M(1)^+,V_L^+[1]\rangle_{\mathrm{VOA}}.
\end{equation}
More precisely, under the identification $L\cong L_{2k}=\Z\sqrt{2k}\,\beta$ used in \cite[Section~2]{DG98}, the vector denoted $e^k$ there is $F_1$ here. \cite[Theorem~2.7(2)]{DG98} shows that $M(1)^+$ is generated by $\omega$ and the weight-four primary $u^4$, while \cite[Theorem~2.9]{DG98} shows that $V_{L_{2k}}^+$ is generated by $\omega,u^4,e^k$. Hence \eqref{eq:VLplus-generated-first-charge} follows.

\item Put
\(P(q):=\prod_{n\ge1}(1-q^n)^{-1},\ T(q):=\prod_{n\ge1}(1+q^n)^{-1},\) and for a nonnegatively $\Z$-graded vector space $W=\bigoplus_{n\ge0}W_n$ with finite-dimensional homogeneous pieces write $\grch W:=\sum_{n\ge0}(\dim W_n)q^n$. Then
\begin{equation}\label{eq:rank-one-character-interface}
\grch M(1)^+=\frac{P(q)+T(q)}2,
\qquad
\grch V_L^+
=\frac{P(q)+T(q)}2+P(q)\sum_{r\ge1}q^{kr^2}.
\end{equation}

\item Let
$ U(1):=\{z\in\mathbb C:|z|=1\}, $ viewed as a group under multiplication, and set $ T_L:=\Hom(L,U(1)). $ Since $L=\mathbb Z\alpha$, evaluation at $\alpha$ identifies $T_L$ with $U(1)$. We let $T_L$ act on $V_L$ by \(\chi\cdot(u\otimes e^\beta) =\chi(\beta)\,u\otimes e^\beta (\chi\in T_L,\ u\in M(1),\ \beta\in L)\). Then $V_L^{T_L}=M(1)$ and $\theta\chi\theta^{-1}=\chi^{-1}$ for every $\chi\in T_L$. Hence the subgroup \(K_L:=\langle T_L,\theta\rangle\le\Aut(V_L)\) is the semidirect product $T_L\rtimes\langle\theta\rangle\cong O(2)$, and $V_L^{K_L}=M(1)^+$. For $r\ge1$, let $\rho_r^{\mathrm{cpt}}$ be the two-dimensional irreducible complex representation of $K_L\cong O(2)$ whose restriction to $T_L$ is the sum of the two characters $\chi\mapsto\chi(r\alpha)$ and $\chi\mapsto\chi(-r\alpha)$. Via the identification \(K_L\cong O(2)\), let \(\det\) denote the determinant character of \(K_L\). Then, with respect to the commuting actions of \(K_L\) and \(M(1)^+\), where \(K_L\) acts on the first tensor factors and \(M(1)^+\) on the second tensor factors,
\begin{equation}\label{eq:rank-one-O2-Schur-Weyl}
V_L\cong
\C\otimes M(1)^+
\oplus\det\otimes M(1)^-
\oplus\bigoplus_{r\ge1}
\rho_r^{\mathrm{cpt}}\otimes M(1,r\alpha).
\end{equation}
\end{enumerate}
\end{prop}

\begin{proof}
For \textnormal{(i)}, the charge decomposition of $V_L$ is preserved by $M(1)^+$, while $\theta$ preserves the zero-charge summand and interchanges the $r\alpha$ and $-r\alpha$ summands. Hence \(V_L^+=M(1)^+\oplus\bigoplus_{r\ge1}\bigl(M(1,r\alpha)\oplus M(1,-r\alpha)\bigr)^\theta\). Because every element of $M(1)^+$ is fixed by $\theta$, the map $\varphi_r$ is $M(1)^+$-linear. Projection onto the $r\alpha$ charge summand is an inverse to $\varphi_r$, so $\varphi_r$ is an isomorphism. Since the lowest conformal-weight space of $M(1,r\alpha)$ is $\C e^{r\alpha}$, its image is $\C F_r$ and has weight $kr^2$. This is the rank-one fixed-point decomposition used in \cite[Section~2]{DG98}.

Part~\textnormal{(ii)} is \cite[Theorem~2.7(1)]{DG98}, and the precise translation of \cite[Theorems~2.7(2) and~2.9]{DG98} giving part~\textnormal{(iii)} was recorded in the statement.

For \textnormal{(iv)}, since $ M(1)\cong S\!\left(\bigoplus_{n\ge1}\mathfrak h\otimes t^{-n}\right) $ as a graded vector space and $\dim\mathfrak h=1$, we have $\grch M(1)=P(q)$. On the one-dimensional oscillator space in mode $-n$, the involution $\theta$ acts by $-1$; hence its graded trace on the corresponding polynomial algebra is $ \sum_{m\ge0}(-1)^m q^{nm}=(1+q^n)^{-1}. $ Taking the product over all $n\ge1$ therefore gives, as an identity of formal power series, $ \operatorname{tr}_{M(1)}\!\left(\theta q^{L(0)}\right) =\prod_{n\ge1}(1+q^n)^{-1}=T(q). $ For each homogeneous subspace $M(1)_N$, the involution $\theta$ has eigenvalues $\pm1$, so $ \dim M(1)_N^+ =\frac12\left(\dim M(1)_N+ \operatorname{tr}\!\left(\theta|_{M(1)_N}\right)\right). $ Summing over $N\ge0$ yields $ \grch M(1)^+=\frac{P(q)+T(q)}2. $ For $r\ge1$, part~\textnormal{(i)} gives a grading-preserving $M(1)^+$-module isomorphism $V_L^+[r]\cong M(1,r\alpha)$. Since $M(1,r\alpha)=M(1)\otimes e^{r\alpha}$ and $\wt(e^{r\alpha})=kr^2$, it follows that $ \grch V_L^+[r] =\grch M(1,r\alpha) =q^{kr^2}P(q). $ Using the decomposition in part~\textnormal{(i)} and summing over $r\ge1$ gives $ \grch V_L^+ =\frac{P(q)+T(q)}2+ P(q)\sum_{r\ge1}q^{kr^2}, $ which is the second formula in \eqref{eq:rank-one-character-interface}.

For \textnormal{(v)}, the lattice charge decomposition $ V_L=\bigoplus_{\beta\in L}M(1,\beta) $ is compatible with the vertex operator in the sense that the charge of a product is the sum of the charges. Since both \(\one\) and \(\omega\) have charge zero, the displayed action fixes the vacuum and conformal vector; hence it is by vertex-operator-algebra automorphisms. Since \(L=\mathbb Z\alpha\), evaluation at \(\alpha\) identifies \(T_L\) with \(U(1)\). If \(\beta=m\alpha\ne0\), one may choose \(z\in U(1)\) with \(z^m\ne1\) and then the character determined by \(\chi(\alpha)=z\) satisfies \(\chi(\beta)\ne1\). Thus the characters in \(T_L\) separate the elements of \(L\). It follows from the charge decomposition that a \(T_L\)-fixed vector has only charge zero, and hence $ V_L^{T_L}=M(1). $ For \(u\in M(1)\) and \(\beta\in L\), using \(\theta^{-1}=\theta\) and the standard lift formula above gives $$
\begin{aligned}
\theta\chi\theta^{-1}(u\otimes e^\beta)
&=\theta\chi(\theta u\otimes e^{-\beta})=\chi(-\beta)\,\theta(\theta u\otimes e^{-\beta})\\
&=\chi(-\beta)\,u\otimes e^\beta=\chi(\beta)^{-1}u\otimes e^\beta.
\end{aligned}
$$ Therefore $ \theta\chi\theta^{-1}=\chi^{-1} (\chi\in T_L). $ Moreover, \(\theta\notin T_L\), since \(T_L\) fixes \(M(1)\) pointwise, whereas \(\theta\) acts as \(-1\) on \(M(1)_1\). Since \(\theta^2=1\), we obtain the internal semidirect product $ K_L=T_L\rtimes\langle\theta\rangle. $ Under the identification \(T_L\cong U(1)\) given by evaluation at \(\alpha\), conjugation by \(\theta\) is inversion. Hence $ K_L\cong U(1)\rtimes_{\mathrm{inv}}\mathbb Z_2\cong O(2), $ with \(T_L\) corresponding to \(SO(2)\) and \(\theta\) to a reflection. Consequently, \(V_L^{K_L}=(V_L^{T_L})^{\langle\theta\rangle}=M(1)^{\langle\theta\rangle}=M(1)^+\). On the zero-charge sector, $ M(1)=M(1)^+\oplus M(1)^-. $ The torus \(T_L\) acts trivially on both summands, while \(\theta\) acts as \(+1\) on \(M(1)^+\) and as \(-1\) on \(M(1)^-\). Thus their \(K_L\)-multiplicity representations are respectively the trivial character and the determinant character. Fix \(r\ge1\) and set \(M_r:=M(1,r\alpha)\). Let \(E_r\) be the two-dimensional vector space with basis \(e_+,e_-\), on which $ \chi e_+=\chi(r\alpha)e_+, \chi e_-=\chi(-r\alpha)e_- (\chi\in T_L), $ and $ \theta e_+=e_-, \theta e_-=e_+. $ The relation \(\theta\chi\theta^{-1}=\chi^{-1}\) shows that these formulas define a representation of \(K_L\). Its two \(T_L\)-weights are \(\chi\mapsto\chi(\pm r\alpha)\), and these weights are distinct for \(r\ge1\). Since \(\theta\) interchanges the corresponding weight lines, \(E_r\) is irreducible. Hence, by the definition in the statement, $ E_r\cong\rho_r^{\mathrm{cpt}}. $ Regard \(E_r\otimes M_r\) as a module for the commuting actions of \(K_L\) and \(M(1)^+\), with \(K_L\) acting on \(E_r\) and trivially on \(M_r\), and with \(M(1)^+\) acting trivially on \(E_r\) and in the usual way on \(M_r\). Now define \(\Phi_r:E_r\otimes M_r\longrightarrow M(1,r\alpha)\oplus M(1,-r\alpha)\) by $ \Phi_r(e_+\otimes u)=u, \Phi_r(e_-\otimes u)=\theta u. $ Because \(\theta\) fixes \(M(1)^+\) pointwise, \(\Phi_r\) is \(M(1)^+\)-linear; the formulas above also show that it is \(K_L\)-equivariant. It is plainly bijective. Therefore \(M(1,r\alpha)\oplus M(1,-r\alpha)\cong\rho_r^{\mathrm{cpt}}\otimes M(1,r\alpha)\) as a \(K_L\times M(1)^+\)-module. Combining this with the zero-charge decomposition gives $$ V_L\cong \mathbb C\otimes M(1)^+ \oplus\det\otimes M(1)^- \oplus\bigoplus_{r\ge1} \rho_r^{\mathrm{cpt}}\otimes M(1,r\alpha), $$ which is \eqref{eq:rank-one-O2-Schur-Weyl}.
\end{proof}

\subsection{Invariant bilinear forms}
\label{subsec:Li-invariant}

We also record the invariant-form results used throughout the reconstruction arguments. Recall that a bilinear form $(\,\cdot\,,\,\cdot\,)$ on a vertex operator algebra $V$ is \emph{invariant} if \(\bigl(Y(a,z)u,v\bigr)=\bigl(u,Y(e^{zL(1)}(-z^{-2})^{L(0)}a,z^{-1})v\bigr)\) for all $a,u,v\in V$. Let $V'$ denote the contragredient module of $V$.

\begin{prop}
\label{prop:Li-invariant-package}
The following statements hold.
\begin{enumerate}[label=\textnormal{(\roman*)}]
\item Every invariant bilinear form on a vertex operator algebra is
symmetric, and the space of invariant bilinear forms on $V$ is naturally isomorphic to \(\bigl(V_0/L(1)V_1\bigr)^*.\) Under this isomorphism a form $(\,\cdot\,,\,\cdot\,)$ corresponds to the functional $v\mapsto(\one,v)$ on $V_0$; see \cite[Proposition~2.6 and Theorem~3.1]{LiInvariant}.

\item Suppose that $V$ is simple and of CFT type.  Then \(V\cong V'\Longleftrightarrow L(1)V_1=0.\)
When these equivalent conditions hold, there is a unique invariant bilinear form normalized by \((\one,\one)=1,\) and this form is symmetric and nondegenerate.

\item For an invariant bilinear form one has \((L(n)u,v)=(u,L(-n)v)(n\in\mathbb Z).\)
Consequently $L(0)$ is self-adjoint and distinct conformal-weight spaces are orthogonal. If the form is nondegenerate, then its restriction to every homogeneous space $V_m$ is nondegenerate.
\end{enumerate}
\end{prop}

\begin{proof}
Part~\textnormal{(i)} follows from \cite[Proposition~2.6 and Theorem~3.1]{LiInvariant}. For a simple VOA of CFT type, $V_0=\C\one$. If $L(1)V_1=0$, then \cite[Corollary~3.2]{LiInvariant} gives a nondegenerate symmetric invariant bilinear form on $V$, and hence $V\cong V'$.

Conversely, invariant bilinear forms on $V$ are naturally identified with $V$-module homomorphisms $V\to V'$; see \cite[p.~284]{LiInvariant}. Thus $V\cong V'$ gives a nondegenerate invariant bilinear form. By \cite[Theorem~3.1]{LiInvariant}, its associated functional \(f:V_0\longrightarrow\C, f(v)=(\one,v)\), annihilates $L(1)V_1$. Since the form is nondegenerate and $V_0=\C\one$, one has $f\ne0$, hence $\ker f=0$. Therefore $L(1)V_1=0$.

Under these equivalent conditions, \(\bigl(V_0/L(1)V_1\bigr)^*=V_0^*\) is one-dimensional, so the normalization $(\one,\one)=1$ determines the invariant bilinear form uniquely; it is symmetric and nondegenerate.

Finally, the Virasoro adjointness relation and the orthogonality of distinct conformal-weight spaces are \cite[(2.21)--(2.23)]{LiInvariant}. If the form is nondegenerate and $u\in V_m$ is orthogonal to $V_m$, then weight-space orthogonality shows that $u$ is orthogonal to all of $V$; hence $u=0$. Thus the restriction of the form to each $V_m$ is nondegenerate.
\end{proof}

\subsection{Uniqueness of the level-one \texorpdfstring{$\mathfrak{sl}_2$}{sl2} extension}
\label{subsec:MY-L1-uniqueness}

We record the two uniqueness results from McRae--Yang \cite[Appendix~B]{MY} that will be used later. Recall that their irreducible module \(\mathcal L^{(1)}_{2m+1,1}\) has lowest conformal weight \(m^2\); hence, in the notation used in this paper, \(\mathcal L^{(1)}_{2m+1,1}\cong L(1,m^2).\) We use the standard identification \(L_1(\mathfrak{sl}_2)\cong V_{A_1}.\)
\begin{thm}
\label{thm:MY-L1-uniqueness}
Let \(U\) be a simple vertex operator algebra extension of \(L(1,0)\).
\begin{enumerate}[label=\textnormal{(\roman*)}]
\item
Suppose that, as an \(L(1,0)\)-module, \(U\cong\bigoplus_{m\ge0}\C^{2m+1}\otimes L(1,m^2)\). Then \(U\cong L_1(\mathfrak{sl}_2)\cong V_{A_1}\) as vertex operator algebras.

\item
Let \(G=\PSLtwo\), and for \(m\ge0\) let \(W_{2m}\) denote the irreducible \(G\)-module of dimension \(2m+1\). Suppose, in addition, that \(G\) acts on \(U\) by vertex operator algebra automorphisms and that \(U\cong\bigoplus_{m\ge0}W_{2m}\otimes L(1,m^2)\) as a \(G\times L(1,0)\)-module. Then the isomorphism in \textnormal{(i)} may be chosen \(G\)-equivariantly, where \(V_{A_1}\cong L_1(\mathfrak{sl}_2)\) carries its standard \(G\)-action.
\end{enumerate}
\end{thm}

\begin{proof}
After identifying \(\mathcal L^{(1)}_{2m+1,1}\cong L(1,m^2),\) part~\textnormal{(i)} is exactly \cite[Theorem~B.1]{MY}, while part~\textnormal{(ii)} is \cite[Theorem~B.2]{MY}.
\end{proof}

\subsection{Direct-limit completions and commutative algebra objects}

Let $V$ be a vertex operator algebra and let $\mathcal C$ be a full subcategory of grading-restricted generalized $V$-modules satisfying the standing hypotheses for the direct-limit construction of Creutzig--McRae--Yang; see \cite[Assumption~4.1]{CMY}. Following \cite[Definition~4.4]{CMY}, we write \(\IndCMY(\mathcal C)\) for the concrete direct-limit completion of $\mathcal C$ in the category of weak $V$-modules: its objects are the weak $V$-modules isomorphic to direct limits of direct systems in $\mathcal C$. Such direct limits exist in the weak-module category by \cite[Proposition~3.5]{CMY}, and \cite[Proposition~3.6]{CMY} shows that a direct limit is the union of the images of its structure maps.

By \cite[Proposition~4.6]{CMY}, equivalently, \(\IndCMY(\mathcal C)\) is the full subcategory of generalized $V$-modules $X$ satisfying
\[
X=\bigcup_{\substack{M\subseteq X\\ M\in\mathcal C}}M.
\]
For such an $X$, let \(I_X:=\{X_0\subseteq X\mid X_0\text{ is a $V$-submodule and }X_0\in\mathcal C\}\), ordered by inclusion. Then $X=\bigcup_{X_0\in I_X}X_0$, and $I_X$ is directed: if $X_0,X_1\in I_X$, then $X_0+X_1$ is a quotient of $X_0\oplus X_1$, hence belongs to $\mathcal C$, and contains both $X_0$ and $X_1$.

Whenever $\mathcal C$ carries the tensor product bifunctor used below, we use its CMY direct-limit extension on the completion. Thus, for $X,Y\in\IndCMY(\mathcal C)$,
\[
X\boxtimes_{\IndCMY(\mathcal C)}Y:=\varinjlim_{(X_0,Y_0)\in I_X\times I_Y}\bigl(X_0\boxtimes_{\mathcal C}Y_0\bigr),
\]
where the direct limit is taken in the category of weak $V$-modules. Here $I_X\times I_Y$ has the componentwise inclusion order, and the transition morphism associated to $X_0\subseteq X_1$ and $Y_0\subseteq Y_1$ is \(i_{X_0,X_1}\boxtimes_{\mathcal C}i_{Y_0,Y_1}: X_0\boxtimes_{\mathcal C}Y_0\longrightarrow X_1\boxtimes_{\mathcal C}Y_1\). Thus every vector of $X$ lies in a $V$-submodule belonging to $\mathcal C$, and morphisms in $\IndCMY(\mathcal C)$ are $V$-module homomorphisms. When the ordinary categorical $\operatorname{Ind}(\mathcal C)$ is needed, its identification with this concrete completion will be proved separately.

A functor between cocomplete categories is called \emph{cocontinuous} if it preserves all small colimits. For functors between abelian categories, \emph{exact} means preserving short exact sequences.

For a braided tensor category $\mathcal C$, we write $\CAlg(\mathcal C)$ for the category of commutative associative algebra objects in the sense of \cite[Definition~7.4]{CMY}. An object \((A,\mu_A,\iota_A)\in\CAlg(\mathcal C)\) consists of $A\in\mathcal C$, a multiplication and a unit \(\mu_A:A\boxtimes A\longrightarrow A,\iota_A:\mathbf 1_{\mathcal C}\longrightarrow A\), satisfying associativity and the left and right unit axioms, together with braided commutativity \(\mu_A\circ\mathcal R_{A,A}=\mu_A.\) Morphisms in $\CAlg(\mathcal C)$ are morphisms in $\mathcal C$ that preserve the multiplication and the unit. These are the algebra objects to which the extension--algebra correspondences used later apply.

\subsection{The Creutzig--McRae--Yang direct-limit package}
\label{subsec:CMY-package}

We shall use the following results of Creutzig--McRae--Yang \cite{CMY}. They are stated here once so that later applications need only verify their hypotheses.

\begin{thm}
\label{thm:CMY-direct-limit}
Let $V$ be a vertex operator algebra and let $\mathcal C$ be a category of grading-restricted generalized $V$-modules. Assume:
\begin{enumerate}[label=\textnormal{(\arabic*)}]
\item $V\in\mathcal C$;
\item $\mathcal C$ is closed under submodules, quotients, and finite
direct sums;
\item every object of $\mathcal C$ is finitely generated;
\item $\mathcal C$ carries the required $P(z)$-vertex and braided
tensor category structures;
\item for every logarithmic intertwining operator
\(\mathcal Y:\ W_1\otimes W_2\longrightarrow X\{x\}[\log x], W_1,W_2\in\mathcal C,X\in\IndCMY(\mathcal C)\), the coefficient image $\operatorname{Im}\mathcal Y$ is an object of $\mathcal C$.
\end{enumerate}
Then the direct-limit bifunctor defined above makes $\IndCMY(\mathcal C)$ into a category carrying $P(z)$-vertex and braided tensor category structures extending those on $\mathcal C$. Moreover, the inclusion \(\mathcal C\hookrightarrow\IndCMY(\mathcal C)\) is a braided tensor functor.
\end{thm}

\begin{proof}
The existence of the $P(z)$-vertex and braided tensor category structures on the direct-limit completion is \cite[Theorem~1.1]{CMY}. The embedding $\mathcal C\hookrightarrow\IndCMY(\mathcal C)$ is a braided tensor functor by \cite[Theorem~5.4]{CMY}; hence the resulting structure restricts to the original tensor structure on $\mathcal C$.
\end{proof}

\begin{prop}
\label{prop:CMY-coefficient-image}
Let $V$ be a vertex operator algebra, let $W_1,W_2$ be $C_1$-cofinite grading-restricted generalized $V$-modules, let $X$ be an arbitrary generalized $V$-module, and let $\mathcal Y$ be a logarithmic intertwining operator of type \(\binom{X}{W_1 W_2}.\) Then \(\operatorname{Im}\mathcal Y\) is a $C_1$-cofinite grading-restricted generalized $V$-module.
\end{prop}

\begin{proof}
This is \cite[Corollary~2.14]{CMY}.
\end{proof}

\begin{prop}
\label{prop:CMY-tensor-universal}
Under the hypotheses of Theorem~\ref{thm:CMY-direct-limit}, for $X_1,X_2\in\IndCMY(\mathcal C)$ there is a canonical tensor-product logarithmic intertwining operator \(\mathcal Y_{X_1,X_2}:X_1\otimes X_2\longrightarrow(X_1\boxtimes X_2)\{x\}[\log x]\) such that, for every $X_3\in\IndCMY(\mathcal C)$, composition with $\mathcal Y_{X_1,X_2}$ gives an isomorphism \(\Hom_V(X_1\boxtimes X_2,X_3)\overset{\sim}{\longrightarrow}\mathcal V^{X_3}_{X_1 X_2}\).
\end{prop}

\begin{proof}
The canonical tensor-product logarithmic intertwining operator is \cite[Theorem~6.1(1)]{CMY}, and the displayed universal property is \cite[Theorem~6.1(2)]{CMY}.
\end{proof}

\begin{thm}
\label{thm:CMY-extension}
Assume Theorem~\ref{thm:CMY-direct-limit} applies. Then the following two categories are naturally isomorphic:
\begin{enumerate}[label=\textnormal{(\arabic*)}]
\item vertex operator algebras $(A,Y,\one_A,\omega_A)$ such that:
\begin{itemize}
\item $A$ is a $V$-module in $\IndCMY(\mathcal C)$;
\item $Y_A(v,x)=Y(v_{-1}\one_A,x)$ for $v\in V$; and
\item $\omega_A=L(-2)\one_A(=\omega_{-1}\one_A)$.
\end{itemize}
\item commutative associative algebra objects
\((A,\mu_A,\iota_A)\in\CAlg(\IndCMY(\mathcal C))\) for which $L(0)$ acts semisimply with integral eigenvalues, the eigenvalues are bounded below, and every $L(0)$-eigenspace is finite dimensional.
\end{enumerate}
\end{thm}

\begin{proof}
This is \cite[Theorem~7.5]{CMY}; see also \cite[Theorem~1.3]{CMY}.
\end{proof}

\section{The \texorpdfstring{$c=1$}{c=1} input and the square category}\label{sec:prelim}

\subsection{The classical character reduction}

We use the module-theoretic and finiteness conventions of Subsection~\ref{subsec:module-conventions} and the rank-one lattice notation of Subsection~\ref{subsec:rank-one-lattice}. For a subgroup $H\leq\Aut(V_L)$, we use the fixed-point notation \(V_L^H:=\{v\in V_L\mid hv=v\text{ for all }h\in H\}.\) For the root lattice $A_1=\Z\alpha$, $\langle\alpha,\alpha\rangle=2$, we use the standard tetrahedral, octahedral, and icosahedral rotation subgroups \(A_4,\ S_4,\ A_5\subset SO(3)\subset\SO_3(\C)\cong\Aut(V_{A_1}),\) so that $V_{A_1}^{A_4}$, $V_{A_1}^{S_4}$, and $V_{A_1}^{A_5}$ denote the corresponding fixed-point vertex operator algebras.

For a rational vertex operator algebra $V$, let $\lambda_{\min}$ be the least conformal weight among its irreducible modules and recall \(\widetilde c(V)=c(V)-24\lambda_{\min}.\) The hypothesis $c=\widetilde c=1$ is essential in the character classification used below.

\begin{thm}\label{thm:DJ-modularity}
Let $V$ be simple, rational, $C_2$-cofinite and self-contragredient. Let $M^1,\dots,M^r$ be the irreducible ordinary $V$-modules and write \(Z_i(\tau)=\operatorname{tr}_{M^i}q^{L(0)-c/24}.\) Then every $Z_i$ is a modular function for a congruence subgroup of $SL_2(\Z)$ and \(\sum_{i=1}^r|Z_i(\tau)|^2\) is $SL_2(\Z)$-invariant.
\end{thm}

\begin{proof}
This is \cite[Theorem~2.6]{DJA4}. Its hypotheses are rationality, $C_2$-cofiniteness, self-duality, and simplicity. In the terminology of that paper, self-duality means $V\cong V'$ as a $V$-module, which is precisely self-contragredience.
\end{proof}

For a vertex operator algebra $V$, we write \(\ch V:=\operatorname{tr}_V q^{L(0)-c(V)/24}\) for its vacuum character.

\begin{thm}\label{thm:character-reduction}
Let $V$ satisfy the hypotheses of Theorem~\ref{thm:main}. Then there is a positive-definite even rank-one lattice $L$ such that \(\ch V\in \left\{ \ch V_L,\ \ch V_L^+,\ \ch V_{A_1}^{A_4},\ \ch V_{A_1}^{S_4},\ \ch V_{A_1}^{A_5} \right\}\).
\end{thm}

\begin{proof}
Theorem~\ref{thm:DJ-modularity} supplies exactly the two modularity assumptions in \cite[Theorem~2.9]{DJA4}. Since $V$ is of CFT type and $c=\widetilde c=1$, that theorem gives the five vacuum characters displayed above. Its character-theoretic input is the classical $c=1$ completeness result of Kiritsis \cite{Kiritsis}. As Dong and Jiang note immediately after Theorem~2.9, their Theorem~2.6 supplies the required modularity.
\end{proof}

\subsection{The square Virasoro category}

Let $L(1,h)$ denote the irreducible highest-weight module for the Virasoro vertex operator algebra $L(1,0)$ of central charge one and highest weight $h$. We retain the notation $X_m=L(1,m^2)$ fixed in Subsection~\ref{subsec:CJORY-package}. Recall from Corollary~\ref{cor:McRae-A1-SO3} that \(\Csq:=\mathcal C_{V_{A_1}}=\left\{\text{finite direct sums of }X_m,\ m\ge0\right\}\) is a full vertex tensor subcategory of $\Ovir$. Thus, for $A,B\in\Csq$ and $z\in\C^\times$, the $P(z)$-tensor product in $\Csq$ is the one computed in $\Ovir$, with the corresponding unit, parallel-transport, associativity, and braiding isomorphisms obtained by restriction.

Set $G:=\PSLtwo\cong\SO_3(\C)$. We write $W_{2m}$ for the irreducible algebraic representation of $G$ of highest weight $2m$, equivalently $\Sym^{2m}(\C^2)$, of dimension $2m+1$. All tensor-category terminology and the conventions for $\boxtimes_{P(z)}$ and $\boxtimes=\boxtimes_{P(1)}$ are those of Section~\ref{sec:categorical}.

\begin{cor}
\label{thm:square-category}
There is a symmetric vertex tensor equivalence \(\Phi:\Rep G\overset{\sim}{\longrightarrow}\Csq, W_{2m}\longmapsto X_m\). For every $z\in\C^\times$ and $M,N\in\Rep G$, its vertex-tensor constraints are natural isomorphisms \(J_{P(z);M,N}:\Phi(M)\boxtimes_{P(z)}\Phi(N)\overset{\sim}{\longrightarrow}\Phi(M\otimes N)\) compatible with the unit, parallel transport, associativity, and braiding. In particular, for all $a,b\ge0$, \(X_a\boxtimes X_b\cong\bigoplus_{j=|a-b|}^{a+b}X_j\).
\end{cor}

\begin{proof}
Corollary~\ref{cor:McRae-A1-SO3} gives a symmetric vertex tensor equivalence \(\Rep_{\mathrm{fd}}SO(3)\overset{\sim}{\longrightarrow}\Csq\) sending \(W_{2m}\) to \(X_m\). Composing with the complexification equivalence of Remark~\ref{rmk:compact-complexification} gives the asserted symmetric vertex tensor equivalence from $\Rep G$. The Clebsch--Gordan rule \(W_{2a}\otimes W_{2b}\cong\bigoplus_{j=|a-b|}^{a+b}W_{2j}\) then gives the displayed tensor-product decomposition.
\end{proof}

\subsection{The full \texorpdfstring{$C_1$}{C1}-cofinite Virasoro category and direct limits}

We use the notation and results of Subsection~\ref{subsec:CJORY-package}. In particular, the universal Virasoro vertex operator algebra at central charge one is $L(1,0)$, and $\Ovir$ is the ambient vertex tensor category of grading restricted $C_1$-cofinite generalized $L(1,0)$-modules. By \eqref{eq:square-C1-cofinite}, every simple object $X_m$ of $\Csq$ belongs to $\Ovir$ and is $C_1$-cofinite.

The VOA extensions considered below need not be finite-length Virasoro modules, so we shall work in a direct-limit completion of $\Csq$. If $A,B\in\Csq$ and a logarithmic intertwining operator takes values in an arbitrary generalized module $X$, one cannot invoke the ambient tensor-product universal property before placing the relevant target inside $\Ovir$. The coefficient-image lemma below resolves this point: it replaces $X$ by the smallest Virasoro submodule generated by the coefficients of the logarithmic intertwining operator and proves that this submodule actually belongs to $\Csq$. This is the technical input needed later for the Creutzig--McRae--Yang direct-limit completion.

\begin{lem}\label{lem:coeff-image}
Let $A,B\in\Csq$, let $X$ be any generalized $L(1,0)$-module, and let $\mathcal Y$ be a logarithmic intertwining operator of type \(\binom{X}{A B}.\) Then \(I:=\operatorname{Im}\mathcal Y\) belongs to $\Csq$. More precisely, $I\in\Ovir$, and for every $z\in\C^\times$ the corresponding $P(z)$-intertwining map induces an epimorphism \(A\boxtimes^{\Ovir}_{P(z)}B\twoheadrightarrow I.\)
\end{lem}

\begin{proof}
By \eqref{eq:square-C1-cofinite}, the simple summands $X_m$ are $C_1$-cofinite grading-restricted $L(1,0)$-modules. Hence the finite direct sums $A$ and $B$ are $C_1$-cofinite grading-restricted $L(1,0)$-modules. Proposition~\ref{prop:CMY-coefficient-image} therefore gives \(I=\operatorname{Im}\mathcal Y\in\Ovir.\) Since all coefficients of $\mathcal Y(a,x)b$ lie in $I$, the same formal series corestricts to a logarithmic intertwining operator \(\mathcal Y^I:A\otimes B\longrightarrow I\{x\}[\log x]\), and, by the definition of $\operatorname{Im}\mathcal Y$, its coefficients span $I$.

Fix $z\in\C^\times$ and a branch of $\log z$, and let \(I_{\mathcal Y^I}:A\otimes B\longrightarrow\overline I\) be the corresponding $P(z)$-intertwining map. Since $I\in\Ovir$, the universal property of $A\boxtimes^{\Ovir}_{P(z)}B$ gives a unique $L(1,0)$-module morphism \(f:A\boxtimes^{\Ovir}_{P(z)}B\longrightarrow I\) such that \(I_{\mathcal Y^I}=\overline f\circ\boxtimes_{P(z)}\). Applying the inverse correspondence \eqref{eq:Pz-inverse} between $P(z)$-intertwining maps and logarithmic intertwining operators gives \(\mathcal Y^I=f\circ\mathcal Y_{\boxtimes}\). Therefore every coefficient of $\mathcal Y^I(a,x)b$ lies in $\operatorname{Im}f$. Since these coefficients span $I$, the map $f$ is surjective.

Finally, Corollary~\ref{cor:McRae-A1-SO3} says that $\Csq$ is a full vertex tensor subcategory of $\Ovir$, so $A\boxtimes^{\Ovir}_{P(z)}B$ is an object of $\Csq$. Equivalently, by Corollary~\ref{thm:square-category}, it is a finite direct sum of simple square modules $X_m$. Hence it is a semisimple $L(1,0)$-module, and every quotient is a finite direct sum of the same simple types. Since $I$ is a quotient of this source, \(I\in\Csq.\)
\end{proof}

\begin{lem}\label{lem:square-ambient-Serre}
The full subcategory $\Csq\subset\Ovir$ is closed under ambient $L(1,0)$-submodules and quotients, and the inclusion $\Csq\hookrightarrow\Ovir$ is exact.
\end{lem}

\begin{proof}
Let \(M=\bigoplus_{m\in F}X_m^{\oplus n_m}\in\Csq,F\subset\Z_{\ge0}\text{ finite}.\) As an ordinary $L(1,0)$-module, $M$ is a finite direct sum of simple modules and is therefore semisimple. Every ambient $L(1,0)$-submodule of $M$ is consequently a direct summand and a finite direct sum of simple modules isomorphic to some of the $X_m$. The same is true of every ambient quotient. Hence $\Csq$ is closed under ambient submodules and quotients. Now let $f:M\to N$ be a morphism in $\Csq$. Its kernel in $\Ovir$ is an ambient submodule of $M$, and its cokernel in $\Ovir$ is an ambient quotient of $N$; hence both belong to $\Csq$. Since $\Csq$ is a full subcategory of $\Ovir$, these ambient kernel and cokernel objects satisfy the same universal properties in $\Csq$. Thus the inclusion $\Csq\hookrightarrow\Ovir$ preserves kernels and cokernels and is therefore exact.
\end{proof}

We use the notation of Section~\ref{sec:categorical}. In particular, $\IndCMY(\Csq)$ denotes the Creutzig--McRae--Yang direct-limit completion of $\Csq$, and $\CAlg(\IndCMY(\Csq))$ denotes its category of commutative associative algebra objects. Once the tensor structure on $\IndCMY(\Csq)$ is established below, its tensor unit is $L(1,0)$, so the unit of an algebra object has the form \(\iota_A:L(1,0)\longrightarrow A\).
\begin{thm}\label{thm:CMY-square}
The category $\Csq$ satisfies the hypotheses of Theorem~\ref{thm:CMY-direct-limit}. Consequently, $\IndCMY(\Csq)$ carries the $P(z)$-vertex tensor category and braided tensor category structures constructed by Creutzig--McRae--Yang, extending those on $\Csq$.

Moreover, by Theorem~\ref{thm:CMY-extension}, the following two categories are isomorphic:
\begin{enumerate}[label=\textnormal{(\arabic*)}]
\item Vertex operator algebras $(A,Y_A,\one_A,\omega_A)$ such that
\begin{itemize}
\item $A$ is an $L(1,0)$-module in $\IndCMY(\Csq)$;
\item writing $Y_A^{L(1,0)}$ for the given $L(1,0)$-module vertex
operator on $A$, \(Y_A^{L(1,0)}(v,x)=Y_A(v_{-1}\one_A,x),v\in L(1,0);\)
\item \(\omega_A=L(-2)\one_A=\omega_{-1}\one_A.\)
\end{itemize}

\item Commutative associative algebra objects \((A,\mu_A,\iota_A)\in\CAlg\bigl(\IndCMY(\Csq)\bigr)\)
such that \(A=\bigoplus_{n\in\Z}A_n,A_n=\{a\in A\mid L(0)a=na\},\) with $A_n=0$ for all sufficiently negative $n$ and $\dim_{\C}A_n<\infty$ for every $n\in\Z$.
\end{enumerate}
\end{thm}

\begin{proof}
We verify the five hypotheses of Theorem~\ref{thm:CMY-direct-limit}.

For condition~(1), the tensor unit \(X_0=L(1,0)\) is an object of $\Csq$ by definition.

For condition~(2), Lemma~\ref{lem:square-ambient-Serre} shows that $\Csq$ is closed under $L(1,0)$-submodules and quotients in the ambient generalized-module category. Closure under finite direct sums is part of the definition of $\Csq$.

For condition~(3), every object of $\Csq$ is a finite direct sum of the irreducible highest-weight modules \(X_m=L(1,m^2),m\ge0.\) Each $X_m$ is generated by a highest-weight vector, and hence every object of $\Csq$ is finitely generated.

For condition~(4), Corollary~\ref{cor:McRae-A1-SO3} already shows that $\Csq=\mathcal C_{V_{A_1}}$ is a full vertex tensor subcategory of $\Ovir$. Hence $\Csq$ carries the required $P(z)$-vertex tensor and braided tensor category structures, and these are precisely the ambient Virasoro structures restricted from $\Ovir$.

It remains to verify condition~(5). Let $A,B\in\Csq$, let $X\in\IndCMY(\Csq)$, and let $\mathcal Y$ be a logarithmic intertwining operator of type \(\binom{X}{A B}.\) By Lemma~\ref{lem:coeff-image}, \(\operatorname{Im}\mathcal Y\in\Csq,\) which is exactly condition~(5) of Theorem~\ref{thm:CMY-direct-limit}.

All five hypotheses are therefore satisfied, so Theorem~\ref{thm:CMY-direct-limit} gives the asserted $P(z)$-vertex tensor and braided tensor category structures on $\IndCMY(\Csq)$, extending those on $\Csq$.

Finally, Theorem~\ref{thm:CMY-extension} identifies the two categories appearing in the statement. This proves the asserted extension--algebra correspondence.
\end{proof}

\begin{lem}
\label{lem:ind-equivalence}
Let $G$ be a reductive algebraic group over $\C$, let $V$ be a vertex operator algebra, and let $\mathcal C$ be a semisimple full vertex tensor subcategory of generalized $V$-modules such that every object of $\mathcal C$ is a finite direct sum of simple objects. Assume that $\mathcal C$ satisfies the hypotheses of Theorem~\ref{thm:CMY-direct-limit}, and let \(\Psi:\Rep G\overset{\sim}{\longrightarrow}\mathcal C\) be a symmetric vertex tensor equivalence.

Then $\Psi$ extends to an exact cocontinuous symmetric monoidal equivalence \(\widehat\Psi:\Rep^{\mathrm{rat}}G\overset{\sim}{\longrightarrow}\IndCMY(\mathcal C).\) More explicitly, choose representatives ${S_\lambda}(\lambda\in\Lambda)$ of the isomorphism classes of simple objects of $\Rep G$ and put $ T_\lambda:=\Psi(S_\lambda). $ For $T\in\mathcal C$ and a complex vector space $E$, let $T\otimes E$ denote the multiplicity-space tensor product, with $V$ acting only on the first factor; this is not the vertex-tensor product $\boxtimes$. Then, for every $M\in\Rep^{\mathrm{rat}}G$, \(M\cong\bigoplus_{\lambda\in\Lambda}S_\lambda\otimes\Hom_G(S_\lambda,M),\) and \(\widehat\Psi(M)\cong\bigoplus_{\lambda\in\Lambda}T_\lambda\otimes\Hom_G(S_\lambda,M).\)
\end{lem}

\begin{proof}
We first identify the two categories in the statement as ordinary Ind-completions.

Since $G$ is reductive over $\C$, every finite-dimensional rational $G$-module is semisimple; this is the usual complete-reducibility theorem for reductive algebraic groups over a field of characteristic zero (see, for example, \cite{SpringerLAG}). If $M\in\Rep^{\mathrm{rat}}G$, then every vector of $M$ is contained in a finite-dimensional rational $G$-submodule and hence in a finite direct sum of simple rational $G$-modules. Thus $M$ is the sum of its simple rational $G$-submodules. A module that is a sum of simple submodules is semisimple, so $M$ is a direct sum of simple rational $G$-modules. Equivalently, the canonical evaluation map \(\bigoplus_{\lambda\in\Lambda} S_\lambda\otimes\Hom_G(S_\lambda,M) \longrightarrow M, s\otimes f\longmapsto f(s)\), is an isomorphism. Consequently \(\Rep^{\mathrm{rat}}G\simeq\Ind(\Rep G).\) We next prove the analogous statement for $\mathcal C$. Let $X\in\IndCMY(\mathcal C)$. By the definition of the CMY completion, $X$ is the directed union of its submodules belonging to $\mathcal C$. Consider the evaluation morphism
\[
\operatorname{ev}_X:
\bigoplus_{\lambda\in\Lambda}
T_\lambda\otimes
\Hom_{\IndCMY(\mathcal C)}(T_\lambda,X)
\longrightarrow X,
\qquad
t\otimes f\longmapsto f(t).
\]
It is surjective: every $x\in X$ belongs to some $X_0\subseteq X$ with $X_0\in\mathcal C$, and $X_0$ is a finite direct sum of the simple objects $T_\lambda$.

It is also injective. Any relation in the source involves only finitely many vectors and finitely many morphisms $T_\lambda\to X$. Because the objects $T_\lambda$ are finitely generated and $X$ is a directed union of its $\mathcal C$-subobjects, the images of these finitely many morphisms are contained in a common subobject $X_0\in\mathcal C$. The relation therefore already occurs in \(\bigoplus_{\lambda\in\Lambda} T_\lambda\otimes\Hom_{\mathcal C}(T_\lambda,X_0) \longrightarrow X_0\), which is injective by the semisimplicity of $\mathcal C$ and Schur's lemma. Hence \(X\cong\bigoplus_{\lambda\in\Lambda}T_\lambda\otimes\Hom_{\IndCMY(\mathcal C)}(T_\lambda,X)\). Thus \(\IndCMY(\mathcal C)\simeq\Ind(\mathcal C).\) In particular, for any complex vector space $E$,
\[
T\otimes E
=
\bigcup_{\substack{E_0\subseteq E\\ \dim E_0<\infty}}
T\otimes E_0
\in\IndCMY(\mathcal C),
\]
since $T\otimes E_0\cong T^{\oplus\dim E_0}\in\mathcal C$. The equivalence $\Psi:\Rep G\to\mathcal C$ therefore has its usual Ind-extension \(\widehat\Psi:\Ind(\Rep G)\longrightarrow\Ind(\mathcal C).\) Under the preceding identifications, this is precisely \(\widehat\Psi(M)=\bigoplus_{\lambda\in\Lambda}T_\lambda\otimes\Hom_G(S_\lambda,M),\) with morphisms acting naturally on the multiplicity spaces. It follows immediately from the isotypic decompositions on the two sides that $\widehat\Psi$ is fully faithful and essentially surjective. Hence it is an equivalence. Both sides are semisimple abelian categories, so $\widehat\Psi$ is exact; and, being an equivalence between cocomplete categories, it preserves all small colimits. Thus it is cocontinuous.

It remains only to extend the tensor structure. Write $M$ and $N$ as directed unions of their finite-dimensional rational $G$-submodules. Since the tensor product of rational $G$-modules commutes with these directed unions, \(M\otimes N\cong\varinjlim_{M_0,N_0}(M_0\otimes N_0),\) where $M_0\subseteq M$ and $N_0\subseteq N$ range over finite-dimensional rational $G$-submodules. On the CMY side, the definition gives \(\widehat\Psi(M)\boxtimes\widehat\Psi(N) = \varinjlim_{M_0,N_0} \bigl(\Psi(M_0)\boxtimes\Psi(N_0)\bigr)\). The tensor constraints of $\Psi$ therefore induce a natural isomorphism \(\widehat\Psi(M)\boxtimes\widehat\Psi(N) \overset{\sim}{\longrightarrow} \widehat\Psi(M\otimes N)\). The unit, associativity, and symmetry diagrams commute because, after restriction to finite-dimensional submodules, they are exactly the corresponding coherence diagrams for $\Psi$. Hence $\widehat\Psi$ is a symmetric monoidal equivalence extending $\Psi$.
\end{proof}

By Corollary~\ref{thm:square-category} and Theorem~\ref{thm:CMY-square}, Lemma~\ref{lem:ind-equivalence} applies to \(\Phi:\Rep G\overset\sim\longrightarrow\Csq\). We henceforth write \(\widehat\Phi:\Rep^{\mathrm{rat}}G\overset\sim\longrightarrow\IndCMY(\Csq)\) for the resulting exact cocontinuous symmetric monoidal equivalence.

\begin{lem}\label{lem:ind-equivariance}
Let $H$ be a reductive algebraic group over $\C$, and let $M$ be a rational $G\times H$-module. Then the commuting $H$-action on $M$ induces a rational $H$-action on $\widehat\Phi(M)$ by morphisms in $\IndCMY(\Csq)$, and there is a canonical isomorphism \(\widehat\Phi(M^H)\cong\widehat\Phi(M)^H\), where the fixed-point object on the right is taken with respect to this induced action.

If, moreover, $M$ is a commutative rational $G$-algebra and $H$ acts by algebra automorphisms, then $H$ acts on $\widehat\Phi(M)$ by commutative-algebra-object automorphisms, and the displayed isomorphism is an isomorphism of commutative algebra objects.
\end{lem}

\begin{proof}
Write \(M\cong\bigoplus_{m\ge0}W_{2m}\otimes E_m\), where \(E_m:=\Hom_G(W_{2m},M)\). Since the $G$- and $H$-actions on $M$ commute, each $E_m$ carries the rational $H$-action \((h\cdot f)(w):=h\cdot f(w)\) for $h\in H$, $f\in E_m$, and $w\in W_{2m}$. With this action, the canonical evaluation isomorphism above is $G\times H$-equivariant. Hence \(M^H\cong\bigoplus_{m\ge0}W_{2m}\otimes E_m^H\). By Lemma~\ref{lem:ind-equivalence}, \(\widehat\Phi(M)\cong\bigoplus_{m\ge0}X_m\otimes E_m\), and the $H$-actions on the multiplicity spaces induce a rational $H$-action on $\widehat\Phi(M)$ by morphisms in $\IndCMY(\Csq)$. Under this decomposition its fixed-point object is \(\widehat\Phi(M)^H\cong\bigoplus_{m\ge0}X_m\otimes E_m^H\). Applying the same formula to $M^H$ therefore gives the canonical isomorphism \(\widehat\Phi(M^H)\cong\widehat\Phi(M)^H\). Now suppose that $M$ is a commutative rational $G$-algebra and that $H$ acts by algebra automorphisms. Then $M^H$ is a commutative rational $G$-subalgebra of $M$, and the inclusion \(i:M^H\hookrightarrow M\) is a morphism of commutative algebra objects in $\Rep^{\mathrm{rat}}G$. Since $\widehat\Phi$ is symmetric monoidal, $\widehat\Phi(i)$ is a morphism of commutative algebra objects. Under the multiplicity-space decompositions above, its image is precisely \(\bigoplus_{m\ge0}X_m\otimes E_m^H=\widehat\Phi(M)^H\). Consequently the canonical isomorphism \(\widehat\Phi(M^H)\cong\widehat\Phi(M)^H\) is an isomorphism of commutative algebra objects.
\end{proof}

\subsection{Extension-admissible algebras and ideals}

The algebra-object/extension correspondence will be used for infinite direct sums. Motivated by Theorem~\ref{thm:CMY-extension}, we isolate the commutative algebra objects that reconstruct genuine CFT-type vertex operator algebra extensions.

\begin{defi}\label{def:extension-admissible}
Let $V_{\mathrm{base}}$ be a vertex operator algebra and $\mathcal C$ a vertex tensor category of $V_{\mathrm{base}}$-modules satisfying the hypotheses of Theorem~\ref{thm:CMY-direct-limit}. We equip $\IndCMY(\mathcal C)$ with the braided tensor category structure supplied by that theorem; its tensor unit is $V_{\mathrm{base}}$. A commutative algebra object $(A,\mu_A,\iota_A)\in\CAlg(\IndCMY(\mathcal C))$ will be called \emph{extension-admissible} if:
\begin{enumerate}[label=(\roman*),leftmargin=2.1em]
\item the unit \(\iota_A:V_{\mathrm{base}}\to A\) is injective and
\(\Hom_{V_{\mathrm{base}}}(V_{\mathrm{base}},A)=\C\iota_A\);
\item $L(0)$ acts semisimply on $A$ with integral eigenvalues;
\item the $L(0)$-spectrum of $A$ is bounded below and every $L(0)$-eigenspace is finite dimensional.
\end{enumerate}
\end{defi}

By Theorem~\ref{thm:CMY-extension}, conditions~(ii)--(iii) are precisely the grading conditions needed for the commutative algebra object $A$ to reconstruct a grading-restricted vertex operator algebra with the same conformal vector as $V_{\mathrm{base}}$. Under this correspondence, the algebra unit $\iota_A:V_{\mathrm{base}}\to A$ is the structural vertex-operator-algebra homomorphism from the base VOA. Thus the injectivity required in~(i) identifies $V_{\mathrm{base}}$ with a vertex operator subalgebra of the reconstructed VOA, so that the latter is a genuine extension of $V_{\mathrm{base}}$.

In the two direct-limit categories used below, the tensor unit $V_{\mathrm{base}}$ is the unique simple object of lowest conformal weight zero, while every other simple object has strictly positive lowest conformal weight. Consequently, $\Hom_{\IndCMY(\mathcal C)}(V_{\mathrm{base}},A)=\C$ implies that the weight-zero space of the reconstructed extension is $\C\one_A$. Moreover, all simple summands occurring in these categories have nonnegative conformal weights. Hence an extension-admissible algebra reconstructs a vertex operator algebra of CFT type.

\begin{prop}\label{prop:subalgebra-interface}
Let $V_{\mathrm{base}}$ be a vertex operator algebra and let $\mathcal C$ be a full vertex tensor subcategory of grading-restricted generalized $V_{\mathrm{base}}$-modules satisfying the hypotheses of Theorem~\ref{thm:CMY-direct-limit}. Assume moreover that $\mathcal C$ is closed under ambient $V_{\mathrm{base}}$-submodules. Let $A\in\CAlg(\IndCMY(\mathcal C))$ correspond under Theorem~\ref{thm:CMY-extension} to a grading-restricted VOA extension $U$ of $V_{\mathrm{base}}$. If $i:B\to A$ is a morphism of commutative algebra objects whose underlying morphism in $\IndCMY(\mathcal C)$ is a monomorphism, then $i$ is injective as a $V_{\mathrm{base}}$-module map and identifies $B$ with a grading-restricted vertex operator subalgebra of $U$ containing $V_{\mathrm{base}}$.
\end{prop}

\begin{proof}
Let $K$ be the kernel of the underlying $V_{\mathrm{base}}$-module map in the ambient module category. Since $B$ is a directed union of $\mathcal C$-subobjects and $\mathcal C$ is closed under ambient $V_{\mathrm{base}}$-submodules, \(K=\bigcup_{F\subset B,F\in\mathcal C}(K\cap F)\in\IndCMY(\mathcal C)\). Thus $K\hookrightarrow B$ is a morphism in $\IndCMY(\mathcal C)$ whose composite with $i$ is zero. Since the underlying morphism $i$ is a monomorphism in $\IndCMY(\mathcal C)$, $K=0$.

Identify $B$ with its image in $U$. The grading hypotheses in Theorem~\ref{thm:CMY-extension} are inherited from the grading-restricted VOA $U$. Moreover $i\circ\iota_B=\iota_A$, so the algebra unit of $B$ is injective and identifies the same copy of $V_{\mathrm{base}}$. The natural extension--algebra correspondence of Theorem~\ref{thm:CMY-extension} therefore sends $i$ to a VOA homomorphism into $U$. Under this correspondence the underlying $V_{\mathrm{base}}$-module map is unchanged. Hence this VOA homomorphism is injective and has image $B$.
\end{proof}

\begin{lem}\label{lem:G-simple-homogeneous}
Let $G$ be a complex algebraic group acting transitively on a nonempty reduced affine variety $X$. Then the coordinate algebra $\OO(X)$ is $G$-simple: its only $G$-stable ideals are $0$ and $\OO(X)$. In particular, if $H\le G$ is a closed subgroup and the homogeneous space $G/H$ is affine and reduced, then $\OO(G/H)$ has no nonzero proper $G$-stable ideals.
\end{lem}

\begin{proof}
Let $I\subsetneq\OO(X)$ be a $G$-stable ideal. By the Nullstellensatz conventions of Subsection~\ref{subsec:AG-conventions}, its zero locus \(V_X(I):=\{x\in X\mid f(x)=0\text{ for every }f\in I\}\) is nonempty. The induced action on regular functions recalled there shows that $V_X(I)$ is $G$-stable: if $x\in V_X(I)$, $g\in G$, and $f\in I$, then $g^{-1}\cdot f\in I$ and hence \(f(gx)=(g^{-1}\cdot f)(x)=0\). Transitivity therefore gives $V_X(I)=X$. The Nullstellensatz now yields \(\sqrt I=\sqrt{(0)}=0\), because $X$ is reduced. Hence $I=0$.
\end{proof}

\subsection{The regular algebra and the fixed-point dictionary}

For $G=\PSLtwo$, let $\OO(G)$ be its coordinate algebra. For $g,x\in G$ and $f\in\OO(G)$, set \((\lambda(g)f)(x)=f(g^{-1}x)\) and \((\rho(g)f)(x)=f(xg)\). We use $\lambda$ for the categorical $G$-action. The right action $\rho$ commutes with $\lambda$ and acts by algebra automorphisms. It will become the automorphism action on the reconstructed VOA.

\begin{lem}\label{lem:admissible-homogeneous}
Let $H\le G=\PSLtwo$ be a closed reductive algebraic subgroup and suppose that $G/H$ is affine. Then \(A:=\widehat\Phi(\OO(G/H))\), viewed as a commutative algebra object of \(\IndCMY(\Csq)\), is extension-admissible. More explicitly, under the left regular $G$-action, \(\OO(G/H) \cong \bigoplus_{m\ge0}W_{2m}^*\otimes W_{2m}^H \cong \bigoplus_{m\ge0}(\dim W_{2m}^H)\,W_{2m}\), and hence \(A\cong\bigoplus_{m\ge0}(\dim W_{2m}^H) X_m.\)
\end{lem}

\begin{proof}
The left translation action makes $\OO(G/H)$ a commutative rational $G$-algebra. Since $\widehat\Phi$ is a symmetric monoidal equivalence, $A=\widehat\Phi(\OO(G/H))$ is therefore a commutative algebra object in $\IndCMY(\Csq)$.

With the conventions above, algebraic Peter--Weyl gives \(\OO(G)\cong\bigoplus_{m\ge0}W_{2m}^*\otimes W_{2m},\) where the first factor carries the left regular action and the second the right regular action. Pullback along the quotient morphism $G\to G/H$ identifies the regular functions on the affine homogeneous space with the right $H$-invariants: \(\OO(G/H)\cong\OO(G)^{H_{\rm right}}\cong\bigoplus_{m\ge0}W_{2m}^*\otimes W_{2m}^H\) as left $G$-modules. Since each $W_{2m}$ is self-dual, this yields, noncanonically as a left $G$-module, \(\OO(G/H)\cong\bigoplus_{m\ge0}(\dim W_{2m}^H) W_{2m}.\) The multiplicity-space formula in Lemma~\ref{lem:ind-equivalence} therefore gives \(A\cong\bigoplus_{m\ge0}(\dim W_{2m}^H) X_m.\) We verify Definition~\ref{def:extension-admissible}. Since $W_0=\C$ and $W_0^H=\C$, the trivial representation occurs in $\OO(G/H)$ with multiplicity one. By the multiplicity-space description of $\widehat\Phi$ in Lemma~\ref{lem:ind-equivalence}, \(\Hom_{L(1,0)}(L(1,0),A)=\Hom_{\IndCMY(\Csq)}(X_0,A)\cong\C.\) The algebra unit of $\OO(G/H)$ is the nonzero morphism $\eta:\C\to\OO(G/H)$, $1\mapsto1$. Under the monoidal unit isomorphism $X_0\cong\widehat\Phi(\C)$, the algebra unit $\iota_A:X_0\to A$ is the morphism induced by $\widehat\Phi(\eta)$. Because $\widehat\Phi$ is faithful, $\iota_A\ne0$. Hence it spans the preceding one-dimensional Hom-space. Since $X_0=L(1,0)$ is simple, this nonzero module homomorphism is injective. Thus \(\Hom_{L(1,0)}(L(1,0),A)=\C\iota_A\), and condition~\textnormal{(i)} holds.

For every $m\ge0$, the ordinary module $X_m=L(1,m^2)$ is $L(0)$-semisimple and has spectrum contained in $m^2+\Z_{\ge0}$. The displayed direct-sum decomposition therefore shows that $L(0)$ acts semisimply on $A$ with nonnegative integral eigenvalues. If $N\in\Z_{\ge0}$, only summands with $m^2\le N$ can contribute to the $N$-eigenspace. There are only finitely many such $m$; moreover, $W_{2m}^H$ is finite dimensional and every homogeneous subspace of the ordinary module $X_m$ is finite dimensional. Hence every $L(0)$-eigenspace of $A$ is finite dimensional, and the spectrum is bounded below by $0$. This proves conditions~\textnormal{(ii)}--\textnormal{(iii)}, so $A$ is extension-admissible.
\end{proof}

\begin{lem}\label{lem:regular-simple}
Let $U_{\mathrm{reg}}$ be the vertex operator algebra extension corresponding to $\widehat\Phi(\OO(G))$; it is well defined by Lemma~\ref{lem:admissible-homogeneous} with $H=\{1\}$. Then $U_{\mathrm{reg}}$ is simple.
\end{lem}

\begin{proof}
Put $A=\widehat\Phi(\OO(G))$ and identify its underlying $L(1,0)$-module with $U_{\mathrm{reg}}$. Let $J$ be a VOA ideal of $U_{\mathrm{reg}}$. Then $J$ is, in particular, an $L(1,0)$-submodule. For every $\Csq$-subobject $F\subset U_{\mathrm{reg}}$, the intersection $J\cap F$ is an ambient $L(1,0)$-submodule of $F$, hence belongs to $\Csq$ by Lemma~\ref{lem:square-ambient-Serre}. Since $U_{\mathrm{reg}}\in\IndCMY(\Csq)$ is the union of its $\Csq$-subobjects, \(J=\bigcup_F(J\cap F).\) Thus the concrete characterization of $\IndCMY(\Csq)$ gives $J\in\IndCMY(\Csq)$, and the inclusion $i:J\hookrightarrow A$ is a morphism in $\IndCMY(\Csq)$.

Because $J$ is an ideal, the vertex operator of $U_{\mathrm{reg}}$ corestricts to an intertwining operator \(Y^J:A\otimes J\longrightarrow J\{x\},Y^J(a,x)u:=Y_{U_{\mathrm{reg}}}(a,x)u,\) of type $\binom{J}{A\,J}$. Proposition~\ref{prop:CMY-tensor-universal} therefore gives a unique morphism \(\mu_J:A\boxtimes J\to J\) corresponding to $Y^J$. Likewise, under the extension--algebra correspondence, the multiplication $\mu_A:A\boxtimes A\to A$ corresponds to the vertex operator of $U_{\mathrm{reg}}$. The identity of intertwining operators \(i\circ Y^J=Y_{U_{\mathrm{reg}}}\circ(\id_A\otimes i)\) and functoriality of the tensor-product intertwining operator show that both morphisms \(i\circ\mu_J\) and \(\mu_A\circ(\id_A\boxtimes i)\) correspond, under the universal bijection of Proposition~\ref{prop:CMY-tensor-universal}, to the same intertwining operator of type $\binom{A}{A\,J}$. Hence \(i\circ\mu_J=\mu_A\circ(\id_A\boxtimes i).\) Choose a symmetric monoidal quasi-inverse $\mathcal F$ of $\widehat\Phi$, together with a monoidal natural isomorphism \(\varepsilon:\mathcal F\circ\widehat\Phi \overset{\sim}{\Longrightarrow}\id_{\Rep^{\mathrm{rat}}G}\). Set \(I=\mathcal F(J)\) and \(j:=\varepsilon_{\OO(G)}\circ\mathcal F(i):I\longrightarrow\OO(G).\) Since $\mathcal F$ is an equivalence, $j$ is a monomorphism in $\Rep^{\mathrm{rat}}G$. Let \(\widetilde I:=j(I)\subseteq\OO(G)\); because $j$ is $G$-equivariant, $\widetilde I$ is a $G$-stable subspace. Since $\varepsilon$ is monoidal, $\mu_A$ is transported to the ordinary multiplication $m:\OO(G)\otimes\OO(G)\to\OO(G)$. Applying $\mathcal F$ to $\mu_J$ and to the preceding identity, and transporting along the monoidal constraints and $\varepsilon$, gives a $G$-equivariant map \(\nu:\OO(G)\otimes I\to I\) such that \(j\circ\nu=m\circ(\id_{\OO(G)}\otimes j).\) Therefore \(m(\OO(G)\otimes\widetilde I)\subseteq\widetilde I\), so $\widetilde I$ is a $G$-stable ideal of $\OO(G)$.

By Lemma~\ref{lem:G-simple-homogeneous}, applied to the left translation action of $G$ on the reduced affine variety $G$, one has \(\widetilde I=0\) or \(\widetilde I=\OO(G)\). In the first case, monicity of $j$ gives $I=0$, and hence $J=0$ because $\mathcal F$ is an equivalence. In the second case, $j$ is surjective as well as injective, hence is an isomorphism; consequently $\mathcal F(i)$ and then $i$ are isomorphisms. Thus $J=U_{\mathrm{reg}}$. Therefore $U_{\mathrm{reg}}$ has no nonzero proper ideals and is simple.
\end{proof}

\begin{thm}\label{thm:regular-algebra}
Let $U_{\mathrm{reg}}$ be the vertex operator algebra extension of $L(1,0)$ corresponding to the commutative algebra object $\widehat\Phi(\OO(G))$. Then there is a $G$-equivariant isomorphism of VOA extensions \(U_{\mathrm{reg}}\overset\sim\longrightarrow V_{A_1}\), where the $G$-action on $U_{\mathrm{reg}}$ is obtained, via $\widehat\Phi$ and the extension--algebra correspondence, from right translations on $\OO(G)$, and $V_{A_1}$ carries its standard automorphism action.
\end{thm}

\begin{proof}
As a $G_{\rm left}\times G_{\rm right}$-module, algebraic Peter--Weyl gives \(\OO(G)\cong\bigoplus_{m\ge0}W_{2m}^*\otimes W_{2m},\) where the first factor carries the left regular action and the second the right regular action. Since every $W_{2m}$ is self-dual and $\widehat\Phi$ extends $\Phi$ with $\Phi(W_{2m})=X_m$, the multiplicity-space description in Lemma~\ref{lem:ind-equivalence} gives \(\widehat\Phi(\OO(G))\cong\bigoplus_{m\ge0}X_m\otimes W_{2m}\) as an $L(1,0)\times G_{\rm right}$-module. In particular, after forgetting the right action, its underlying Virasoro module is \(\bigoplus_{m\ge0}(2m+1)L(1,m^2).\) Taking $H=\{1\}$ in Lemma~\ref{lem:admissible-homogeneous} shows that $\widehat\Phi(\OO(G))$ is extension-admissible, and hence reconstructs the grading-restricted VOA extension $U_{\mathrm{reg}}$. By Lemma~\ref{lem:regular-simple}, $U_{\mathrm{reg}}$ is simple.

Apply Lemma~\ref{lem:ind-equivariance} with the second copy $G_{\rm right}$ as the commuting group. The right translations on $\OO(G)$ therefore induce a rational algebraic $G$-action on $\widehat\Phi(\OO(G))$ by commutative-algebra-object automorphisms. Since Theorem~\ref{thm:CMY-extension} is an isomorphism of categories, this action corresponds to an action of $G$ on $U_{\mathrm{reg}}$ by VOA automorphisms. Moreover, each algebra-object automorphism preserves the unit morphism $L(1,0)\to\widehat\Phi(\OO(G))$, so the resulting $G$-action fixes the embedded copy of $L(1,0)$ pointwise. Reordering the preceding bicovariant decomposition, we thus have \(U_{\mathrm{reg}}\cong\bigoplus_{m\ge0}W_{2m}\otimes L(1,m^2)\) as a $G\times L(1,0)$-module.

Consequently the hypotheses of Theorem~\ref{thm:MY-L1-uniqueness}\textnormal{(i)--(ii)} are satisfied. Thus there is a $G$-equivariant vertex operator algebra isomorphism \(U_{\mathrm{reg}}\overset\sim\longrightarrow L_1(\mathfrak{sl}_2)\cong V_{A_1},\) where $V_{A_1}$ carries the standard $G$-action. A vertex operator algebra isomorphism preserves the conformal vector, and hence restricts to the identity on the Virasoro subalgebra generated by that conformal vector, namely the distinguished copy of $L(1,0)$ in both extensions. Therefore the displayed map is an isomorphism of VOA extensions. We fix one such isomorphism for the remainder of the paper.
\end{proof}

\begin{cor}\label{cor:fixed-point-dictionary}
Let $H\le G$ be a closed reductive algebraic subgroup and suppose that $G/H$ is affine. Let $U_H$ be the VOA extension of $L(1,0)$ corresponding to the commutative algebra object \(\widehat\Phi(\OO(G/H))\). Then \(U_H\cong V_{A_1}^{H}\) as VOA extensions of $L(1,0)$. Consequently, any VOA extension of $L(1,0)$ whose corresponding commutative algebra object is isomorphic, as a commutative algebra object in $\IndCMY(\Csq)$, to \(\widehat\Phi(\OO(G/H))\) is isomorphic to $V_{A_1}^{H}$ as a VOA extension of $L(1,0)$.
\end{cor}

\begin{proof}
By Lemma~\ref{lem:admissible-homogeneous}, \(A_H:=\widehat\Phi(\OO(G/H))\) is extension-admissible, so the extension $U_H$ is well defined. Pullback along the quotient morphism $G\to G/H$ identifies \(\OO(G/H)=\OO(G)^{H_{\rm right}}\) as a commutative rational $G$-subalgebra of $\OO(G)$. Let \(j:\OO(G/H)\hookrightarrow\OO(G)\) denote this inclusion. Since $\widehat\Phi$ is a symmetric monoidal equivalence, \(\widehat\Phi(j):A_H\longrightarrow A_{\rm reg}:=\widehat\Phi(\OO(G))\) is a morphism of commutative algebra objects, and its underlying morphism in $\IndCMY(\Csq)$ is a monomorphism. By Lemma~\ref{lem:ind-equivariance}, under the canonical identification \(A_H\cong A_{\rm reg}^{H}\), the image of $\widehat\Phi(j)$ is exactly the fixed subobject $A_{\rm reg}^{H}$ for the induced right $H$-action.

By Lemma~\ref{lem:square-ambient-Serre}, $\Csq$ is closed under ambient $L(1,0)$-submodules. Hence Proposition~\ref{prop:subalgebra-interface} applies to $\widehat\Phi(j)$ and identifies $U_H$ with the vertex operator subalgebra of $U_{\rm reg}$ whose underlying $L(1,0)$-module is $A_{\rm reg}^{H}$. The right $H$-action on $A_{\rm reg}$ corresponds, under Theorem~\ref{thm:CMY-extension}, to the action of $H$ on $U_{\rm reg}$ by VOA automorphisms constructed in Theorem~\ref{thm:regular-algebra}. Therefore the fixed subobject $A_{\rm reg}^{H}$ is precisely the underlying $L(1,0)$-module of the fixed-point vertex operator subalgebra $U_{\rm reg}^{H}$, and hence \(U_H=U_{\rm reg}^{H}.\) The $G$-equivariant isomorphism of VOA extensions in Theorem~\ref{thm:regular-algebra} restricts to the $H$-fixed subalgebras, giving \(U_H=U_{\rm reg}^{H}\cong V_{A_1}^{H}\) as VOA extensions of $L(1,0)$.

Finally, let $U$ be a VOA extension of $L(1,0)$ whose corresponding commutative algebra object $B\in\CAlg(\IndCMY(\Csq))$ is isomorphic to $A_H$. Since Theorem~\ref{thm:CMY-extension} is an isomorphism of categories, an isomorphism $B\cong A_H$ of commutative algebra objects corresponds to an isomorphism $U\cong U_H$ of VOA extensions. Combining this with the preceding identification of $U_H$ proves the last assertion.
\end{proof}

\section{Square cores and the exceptional character branches}\label{sec:squarecore}

\subsection{The invariant form and the first primary}

Throughout this section $V$ is a simple self-contragredient VOA of CFT type and central charge one. By Proposition~\ref{prop:Li-invariant-package}, there is a unique normalized nondegenerate symmetric invariant bilinear form on $V$; we fix the normalization $(\one,\one)=1$. The same proposition implies that $L(0)$ is self-adjoint, distinct conformal-weight spaces are orthogonal, and the form restricts nondegenerately to every $V_n$.

\begin{defi}
For $h\in\Z_{\ge0}$ set \(P_h(V)=\{v\in V_h\mid L(n)v=0\text{ for all }n>0\}.\) A nonzero vector in \(P_h(V)\) is called a primary vector of weight $h$. A primary vector $v$ is called nonisotropic if \((v,v)\ne0\). The first positive primary weight of $V$ is the least $h>0$ such that \(P_h(V)\ne0\), whenever such an $h$ exists.
\end{defi}

\begin{lem}\label{lem:primary-complement}
Let $h>0$ and suppose that \(V_n=L(1,0)_n(0\le n<h).\) Then \(V_h=L(1,0)_h\mathbin{\perp}P_h(V).\) In particular, the restriction of the invariant form to $P_h(V)$ is nondegenerate and \(\dim P_h(V)=\dim V_h-\dim L(1,0)_h.\)
\end{lem}

\begin{proof}
The restriction of the invariant form to the conformal subVOA \(L(1,0)\) is again invariant and is nonzero since \((\one,\one)=1\). Its radical is therefore an ideal of \(L(1,0)\); as \(L(1,0)\) is simple, the restriction is nondegenerate. Moreover, \(L(0)\) is self-adjoint, so distinct homogeneous subspaces of \(L(1,0)\) are mutually orthogonal. Hence the form restricts nondegenerately to each \(L(1,0)_n\). Let $Q=L(1,0)_h^\perp\cap V_h$. If $v\in Q$, then, for $1\le n\le h$ and $w\in V_{h-n}=L(1,0)_{h-n}$, invariance gives \((L(n)v,w)=(v,L(-n)w)=0.\) The form on $V_{h-n}$ is nondegenerate, so $L(n)v=0$. For $n>h$ the target has negative conformal weight and vanishes. Thus $Q\subseteq P_h(V)$.

Conversely, let $v\in P_h(V)$. Every vector in $L(1,0)_h$ is a linear combination of vacuum descendants \(L(-n_1)\cdots L(-n_r)\one,n_i>0.\) Conversely, let \(v\in P_h(V)\). Every vector in \(L(1,0)_h\) is a linear combination of vacuum descendants \(L(-n_1)\cdots L(-n_r)\one\), where \(r\ge1\) and \(n_i>0\). By invariance, \(\bigl(L(-n_1)\cdots L(-n_r)\one,v\bigr)=\bigl(L(-n_2)\cdots L(-n_r)\one,L(n_1)v\bigr)=0\), since \(L(n_1)v=0\). Thus \(v\) is orthogonal to \(L(1,0)_h\), and hence \(P_h(V)\subseteq Q\). We have therefore proved $ Q=P_h(V). $ Since the form on \(V_h\) is nondegenerate and its restriction to \(L(1,0)_h\) is nondegenerate, we have $ V_h=L(1,0)_h\mathbin{\perp}P_h(V). $ It follows that the restriction of the form to \(P_h(V)\) is nondegenerate, and hence $ \dim P_h(V)=\dim V_h-\dim L(1,0)_h. $
\end{proof}

\begin{lem}\label{lem:first-primary}
Let $h>0$. Assume \(V_n=L(1,0)_n(0\le n<h),\dim V_h=\dim L(1,0)_h+1.\) Then there is, uniquely up to nonzero scalar, a vector $v\in P_h(V)$ such that \(V_h=L(1,0)_h\oplus\C v,(v,v)\ne0.\)
\end{lem}

\begin{proof}
By Lemma~\ref{lem:primary-complement}, \(V_h=L(1,0)_h\mathbin{\perp}P_h(V),\dim P_h(V)=1\), and the restriction of the invariant form to \(P_h(V)\) is nondegenerate. Hence any \(0\ne v\in P_h(V)\) satisfies \((v,v)\ne0\) and $ V_h=L(1,0)_h\oplus\C v. $ Since \(P_h(V)\) is one-dimensional, such \(v\) is unique up to nonzero scalar.
\end{proof}

\begin{lem}\label{lem:vacuum-channel}
Let $u,v\in P_h(V)$, where $h\in\Z_{\ge0}$. Then \(u_{2h-1}v=(-1)^h(u,v)\one.\) In particular, the weight-zero (vacuum) coefficient of the product of two primary vectors is nonzero whenever their invariant pairing is nonzero.
\end{lem}

\begin{proof}
By invariance of the bilinear form, \((Y(u,z)v,\one)=\bigl(v,Y(e^{zL(1)}(-z^{-2})^{L(0)}u,z^{-1})\one\bigr)\). Since \(u\in P_h(V)\), we have \(L(n)u=0\) for every \(n>0\), in particular \(L(1)u=0\), while \(L(0)u=hu\). Hence \(e^{zL(1)}(-z^{-2})^{L(0)}u=(-z^{-2})^h u=(-1)^hz^{-2h}u\), where we use \(h\in\Z_{\ge0}\). Therefore, by the standard identity \(Y(u,x)\one=e^{xL(-1)}u\), \((Y(u,z)v,\one)=(-1)^hz^{-2h}\bigl(v,e^{z^{-1}L(-1)}u\bigr)\). Expanding the exponential gives \(\bigl(v,e^{z^{-1}L(-1)}u\bigr)=\sum_{k\ge0}\frac{z^{-k}}{k!}\bigl(v,L(-1)^ku\bigr)\). Now \(v\in V_h\), whereas \(L(-1)^ku\in V_{h+k}\). Since distinct conformal-weight spaces are orthogonal with respect to the invariant form, every term with \(k>0\) vanishes. Thus \((Y(u,z)v,\one)=(-1)^hz^{-2h}(v,u)=(-1)^hz^{-2h}(u,v)\), the last equality following from symmetry of the invariant form. Comparing the coefficient of \(z^{-2h}\) gives $ (u_{2h-1}v,\one)=(-1)^h(u,v). $ Since $ \wt(u_{2h-1}v) = 0, $ we have \(u_{2h-1}v\in V_0=\C\one\). As the form is normalized by \((\one,\one)=1\), the preceding equality therefore implies $ u_{2h-1}v=(-1)^h(u,v)\one. $
\end{proof}

\subsection{The first singular vector}

We use the Verma-module, level, singular-vector, and Shapovalov-form conventions fixed in Subsection~\ref{subsec:CJORY-package}, in particular Proposition~\ref{prop:c1-verma}.

\begin{lem}\label{lem:first-singular}
Fix $m\ge1$ and set \(h=m^2,H=(m+1)^2.\) Assume
\begin{enumerate}[label=(\roman*),leftmargin=2.1em]
\item $V_n=L(1,0)_n$ for $0\le n<h$;
\item $\dim V_h=\dim L(1,0)_h+1$;
\item for some $\varepsilon\in\{0,1\}$, \(\dim V_H=\dim L(1,0)_H+\dim(X_m)_H+\varepsilon.\)
\end{enumerate}
Then the vector $v$ from Lemma~\ref{lem:first-primary} generates \(U(\Vir)v\cong X_m.\)
\end{lem}

\begin{proof}
Let $C=U(\Vir)v$ and let \(\phi:V(1,m^2)\twoheadrightarrow C\) be the canonical Virasoro homomorphism sending the Verma highest-weight vector to $v$. Let $\widetilde s\in V(1,m^2)_H$ be a nonzero first singular vector and put $s=\phi(\widetilde s)$. If $s=0$, Proposition~\ref{prop:c1-verma} shows that $\phi$ factors through $X_m$. The induced map $X_m\to C$ is nonzero and surjective; since $X_m$ is simple, it is an isomorphism. Suppose therefore that $s\ne0$.

By Proposition~\ref{prop:c1-verma}, there is no nontrivial singular vector of positive level below weight $H$, and the singular space at weight $H$ is one-dimensional. Since $\phi(\one_{1,m^2})=v\ne0$, the level-zero part of $\ker\phi$ vanishes. If $\ker\phi$ were nonzero at some positive level below $2m+1$, choose a nonzero homogeneous vector $x$ in $\ker\phi$ of minimal positive level. Since $\ker\phi$ is a Virasoro submodule, if $x$ has level $N>0$, then $L(n)x\in\ker\phi$ has level $N-n$. For $0<n<N$ this vanishes by the minimality of $N$, while for $n=N$ it vanishes because the level-zero part of $\ker\phi$ is zero; for $n>N$ there is no negative-level subspace. Hence $L(n)x=0$ for every $n>0$, by the minimality of its level, so $x$ would be a nontrivial singular vector, a contradiction. Moreover, if $x\in(\ker\phi)_H$, then $L(n)x\in(\ker\phi)_{H-n}=0$ for every $n>0$, so every nonzero such $x$ is singular. Hence \((\ker\phi)_H\subseteq\C\widetilde s\). Since $\phi(\widetilde s)=s\ne0$, we have $(\ker\phi)_H=0$. Thus $\phi$ is injective through weight $H$. Since the maximal proper submodule is $V(1,H)$ and its weight-$H$ space is its one-dimensional highest-weight line, \(\dim C_H=\dim V(1,m^2)_H=\dim(X_m)_H+1.\)

Let $\langle\ ,\ \rangle_{\rm Sh}$ be the normalized Shapovalov form on $V(1,m^2)$. Its radical is the unique maximal proper submodule and, by Proposition~\ref{prop:c1-verma}, contains $\widetilde s$. Pull the invariant form on $V$ back along $\phi$ by setting \(\langle x,y\rangle_\phi=(\phi(x),\phi(y)).\) By Proposition~\ref{prop:Li-invariant-package}(iii), $\langle\ ,\ \rangle_\phi$ is contravariant, and \(\langle\one_{1,m^2},\one_{1,m^2}\rangle_\phi=(v,v)\ne0.\) Hence it is a nonzero scalar multiple of $\langle\ ,\ \rangle_{\rm Sh}$. Therefore, for every $x\in V(1,m^2)$, \((s,\phi(x))=(\phi(\widetilde s),\phi(x))=\langle\widetilde s,x\rangle_\phi=0.\) Surjectivity of $\phi$ gives $(s,C)=0$. Since $\widetilde s$ is singular and $\phi$ is a Virasoro-module homomorphism, for every $n>0$ we have $ L(n)s=L(n)\phi(\widetilde s) =\phi(L(n)\widetilde s)=0. $ Thus, if $s\ne0$, then $s$ is a primary vector of weight $H$. Moreover $C\cap L(1,0)=0$. Indeed, the intersection is a Virasoro submodule of the simple module $L(1,0)$, so it is either zero or all of $L(1,0)$. The latter is impossible because $C$ is supported in conformal weights at least $m^2>0$, whereas $L(1,0)$ contains the vacuum in weight zero.

If $\varepsilon=0$, then, since $C_H\cap L(1,0)_H=0$ and $\dim C_H=\dim(X_m)_H+1$, we have \(\dim V_H\ge\dim L(1,0)_H+\dim C_H=\dim L(1,0)_H+\dim(X_m)_H+1,\) contradicting condition~\textnormal{(iii)}.

Assume $\varepsilon=1$. Together with $C\cap L(1,0)=0$, the displayed dimension formula and condition~\textnormal{(iii)} force \(V_H=L(1,0)_H\oplus C_H.\) We already know $s\perp C_H$. Since $s$ is primary, it is orthogonal to every vacuum descendant. Indeed, if $ w=L(-n_1)\cdots L(-n_r)\one\in L(1,0)_H, $ then invariance gives $ (s,w) =(L(n_1)s,L(-n_2)\cdots L(-n_r)\one)=0, $ since $n_1>0$ and $s$ is primary. Thus $s\perp V_H$, contradicting nondegeneracy of the invariant form on $V_H$. Hence $s=0$.
\end{proof}

\subsection{Semisimple square cores}

The following closure statement replaces any appeal to global Virasoro semisimplicity. For a subset $T\subset V$, write \(\langle T\rangle_{\mathrm{VOA}}\) for the vertex operator subalgebra of $V$ generated by $T$.

\begin{thm}\label{thm:square-core}
Let $V$ be a VOA extension of $L(1,0)$. Let \(S=S_0\oplus S_1\oplus\cdots\oplus S_r\subset V\) be a finite direct sum of Virasoro submodules such that \(S_0=L(1,0)\) and \(S_i\cong X_{m_i}\) for \(1\le i\le r\). Set \(U=\langle S\rangle_{\mathrm{VOA}}\). Then \(U\in\IndCMY(\Csq)\). Equivalently, $U$ is an algebraic direct sum of copies of square simple modules as a Virasoro module. No nonsplit square-block extension is created by the vertex-algebra multiplication.
\end{thm}

\begin{proof}
Set \(F_0=S\). Suppose that \(F_n\in\Csq\) has been constructed and choose an internal decomposition \(F_n=\bigoplus_{\lambda\in\Lambda_n}S_\lambda,S_\lambda\cong X_{m(\lambda)},\) where $\lambda$ records the actual embedded copy, not merely its isomorphism class. For each ordered pair \(\lambda,\mu\in\Lambda_n\), the restriction \(Y_V|_{S_\lambda\otimes S_\mu}\) is an $L(1,0)$-intertwining operator of type \(\binom{V}{S_\lambda S_\mu},\) because $V$ is a VOA extension of $L(1,0)$. Let \(I_{\lambda,\mu}:=\operatorname{Im}\bigl(Y_V|_{S_\lambda\otimes S_\mu}\bigr)\) denote its coefficient image. By Lemma~\ref{lem:coeff-image}, \(I_{\lambda,\mu}\in\Csq\).

Define \(F_{n+1}=F_n+\sum_{\lambda,\mu\in\Lambda_n}I_{\lambda,\mu}\subset V.\) Consider the natural Virasoro-module epimorphism
\[
\psi_n:
F_n\oplus
\bigoplus_{\lambda,\mu\in\Lambda_n}I_{\lambda,\mu}
\twoheadrightarrow F_{n+1},
\qquad
\psi_n\bigl(x,(y_{\lambda,\mu})\bigr)
=
x+\sum_{\lambda,\mu}y_{\lambda,\mu}.
\]
Its source belongs to $\Csq$, since $\Lambda_n$ is finite and each summand belongs to $\Csq$. By the quotient-closure in Lemma~\ref{lem:square-ambient-Serre}, \(F_{n+1}\in\Csq\). Thus we obtain an increasing sequence of Virasoro submodules \(F_0\subseteq F_1\subseteq F_2\subseteq\cdots.\) Put \(F_\infty:=\bigcup_{n\ge0}F_n.\) We claim that $F_\infty$ is a vertex operator subalgebra of $V$. Indeed, let \(a,b\in F_\infty\). Choose \(n\) such that \(a,b\in F_n\), and write them with respect to the above decomposition of \(F_n\) as finite sums of vectors from the embedded simple summands \(S_\lambda\). By bilinearity of $Y_V$ and the definition of \(I_{\lambda,\mu}\), every coefficient of \(Y_V(a,x)b\) belongs to \(F_n+\sum_{\lambda,\mu\in\Lambda_n}I_{\lambda,\mu}=F_{n+1}\subseteq F_\infty.\) Thus $F_\infty$ is closed under all mode products. Moreover, \(F_\infty\supseteq F_0=S\supseteq L(1,0)\), so it contains the vacuum and the conformal vector. Therefore $F_\infty$ is a vertex operator subalgebra of $V$ containing $S$. By the definition of $U$, \(U\subseteq F_\infty.\) Conversely, \(F_0=S\subseteq U\). If \(F_n\subseteq U\), then every summand \(S_\lambda\) is contained in $U$, and since $U$ is a vertex operator subalgebra, every coefficient of \(Y_V(S_\lambda,x)S_\mu\) lies in $U$. Hence \(I_{\lambda,\mu}\subseteq U\) for all \(\lambda,\mu\in\Lambda_n\), and therefore \(F_{n+1}\subseteq U\). By induction, \(F_n\subseteq U\) for all \(n\), so \(F_\infty\subseteq U.\) Consequently, \(U=F_\infty=\bigcup_{n\ge0}F_n.\) Hence \(U=\bigcup_{\substack{M\subseteq U\\M\in\Csq}}M,\) by the characterization of \(\IndCMY(\Csq)\) recalled above, \(U\in\IndCMY(\Csq)\). Finally, Lemma~\ref{lem:ind-equivalence} gives complex vector spaces \(E_m\) such that \(U\cong\bigoplus_{m\ge0}X_m\otimes E_m\) as an $L(1,0)$-module. Hence $U$ is an algebraic direct sum of copies of square simple modules, and in particular no nonsplit square-block extension occurs.
\end{proof}

\subsection{A closed-orbit reconstruction lemma}

We shall transport square cores through the equivalence \(\widehat\Phi^{-1}:\IndCMY(\Csq)\longrightarrow\Rep^{\mathrm{rat}}G.\) We first record a general finite-generation lemma, using the notation $\Psi$ and $\widehat\Psi$ of Lemma~\ref{lem:ind-equivalence}; it will then be applied to $\Phi$ (and later to other tensor equivalences).

\begin{lem}\label{lem:finite-generation}
In the setting of Lemma~\ref{lem:ind-equivalence}, let $U$ be a VOA extension of the base VOA $V$ whose underlying $V$-module lies in $\IndCMY(\mathcal C)$. Suppose that there is a subobject $E\subseteq U$ with $E\in\mathcal C$ such that $U$ is generated as a vertex operator algebra by $V$ and $E$; equivalently, $U$ is the smallest vertex operator subalgebra of itself containing both $V$ and $E$. Via the extension--algebra correspondence of Theorem~\ref{thm:CMY-extension}, regard $U$ as a commutative algebra object in $\IndCMY(\mathcal C)$, and put $ B:=\widehat\Psi^{-1}(U). $ Then $B$ is generated, as an ordinary commutative $\C$-algebra, by the finite-dimensional $G$-module $\Psi^{-1}(E)$. In particular, $B$ is finitely generated.
\end{lem}

\begin{proof}
Put $W:=\Psi^{-1}(E)\subseteq B$. Since $E\in\mathcal C$ and $\Psi:\Rep G\overset{\sim}{\longrightarrow}\mathcal C$ is an equivalence, $W$ is a finite-dimensional rational $G$-module. Let $B_W\subseteq B$ be the ordinary unital commutative $\C$-subalgebra generated by $W$. Since the multiplication and unit of $B$ are $G$-equivariant and $W$ is $G$-stable, $B_W$ is $G$-stable. Hence \(B_W\) is a commutative algebra object in \(\Rep^{\mathrm{rat}}G\), and the inclusion $ B_W\hookrightarrow B $ is a monomorphism of commutative algebra objects. Since \(\widehat\Psi\) is a symmetric monoidal equivalence, the inclusion \(B_W\hookrightarrow B\) is sent to a monomorphism of commutative algebra objects $ C:=\widehat\Psi(B_W)\hookrightarrow \widehat\Psi(B)=U. $ Moreover, \(B_W\) contains both the tensor unit \(\C\) and \(W=\Psi^{-1}(E)\). Since \(\widehat\Psi\) extends \(\Psi\) and preserves the tensor unit, it follows that \(C\) contains \(V\cong\widehat\Psi(\C)\) and \(E\cong\widehat\Psi(W)\). As a $V$-submodule of the grading-restricted VOA $U$, the object $C$ is itself grading restricted. Proposition~\ref{prop:subalgebra-interface} therefore identifies $C$ with a vertex operator subalgebra of $U$. Since $C$ contains both $V$ and $E$, while $U$ is generated as a vertex operator algebra by $V$ and $E$, we obtain $C=U$. Because $\widehat\Psi$ is an equivalence, it follows that $B_W=B$. Thus $W$ generates $B$ as an ordinary commutative $\C$-algebra.
\end{proof}

\begin{lem}\label{lem:first-invariant-degrees}
Let $G=\PSLtwo\cong\SO_3(\C)$, let $T\le G$ be a maximal torus, and write \(N(T):=N_G(T)\) for its normalizer. For $m\ge0$, the first positive $m$ for which $W_{2m}^H\ne0$ is
\[
\begin{array}{c|ccccccc}
H&T&N(T)&C_n&D_n&A_4&S_4&A_5\\ \hline
\text{first }m&1&2&1&2&3&4&6.
\end{array}
\]
Here $C_n$ denotes a finite cyclic rotation subgroup, and $D_n$ a dihedral rotation subgroup with $n\ge2$.
\end{lem}

\begin{proof}
Since all maximal tori of $G$ are conjugate, we may assume that $T$ is the standard diagonal torus. For every $m\ge1$, under the realization $W_{2m}\cong\Sym^{2m}(\C^2)$, its zero-weight line is spanned by $x^my^m$. A representative $w=\begin{pmatrix}0&-1\\1&0\end{pmatrix}$ of the nontrivial element of $N_G(T)/T$ sends $x^my^m$ to $(-1)^m x^my^m$. Hence the first positive $T$-invariant degree is $1$, while the first positive $N(T)$-invariant degree is $2$.

For the finite rotation groups, use their standard realizations inside the compact real form \(SO(3)\subset G\). Under the standard three-dimensional representation, identify \(W_2\cong\C^3,W_4\cong\Sym^2_0(\C^3),\) where \(\Sym^2_0(\C^3)\) denotes the space of traceless symmetric \(3\times3\) matrices. A cyclic rotation group \(C_n\) fixes its rotation axis, and hence fixes a nonzero vector in \(W_2\cong\C^3\). Thus \(W_2^{C_n}\ne0\), so its first positive invariant degree is \(1\). For a dihedral rotation group $D_n$ with $n\ge2$, let $e$ be a unit vector along the principal rotation axis. The subgroup $C_n\le D_n$ fixes precisely the line $\C e$ in $W_2$, whereas a half-turn about an axis perpendicular to $e$ sends $e$ to $-e$. Thus $W_2^{D_n}=0$. On the other hand, the traceless quadratic tensor \(Q_e:=e\otimes e-\frac13 I\in\Sym^2_0(\C^3)\cong W_4\) is fixed by every element of $D_n$, since each such element sends the principal axis to itself and hence $e$ to $\pm e$. Here \(e\otimes e=ee^{\mathsf T}\) denotes the rank-one symmetric matrix associated with \(e\), and \(I=I_3\) is the identity matrix. Therefore the first positive invariant degree of $D_n$ is $2$.

For the three Platonic groups, let $\chi_m(\theta):=\operatorname{tr}\rho_m(R_\theta),$ where $R_\theta\in SO(3)$ is any rotation through angle $\theta$ and $\rho_m$ denotes the representation on $W_{2m}$. For a rotation through angle \(\theta\), the eigenvalues on \(W_{2m}\) are \(e^{ij\theta}\) for \(-m\le j\le m\). Hence \(\chi_m(\theta)=\sum_{j=-m}^{m}e^{ij\theta}=1+2\sum_{j=1}^{m}\cos(j\theta).\) For \(\theta\not\equiv0\pmod{2\pi}\), the geometric-series identity gives \(\chi_m(\theta)=\frac{\sin((m+\tfrac12)\theta)}{\sin(\theta/2)},\) while \(\chi_m(0)=2m+1\). The averaging formula gives \(\dim W_{2m}^H=|H|^{-1}\sum_{h\in H}\chi_m(h)\). The angle distributions are
\[
\begin{array}{c|l}
A_4&1\cdot0,\ 3\cdot\pi,\ 8\cdot\frac{2\pi}{3},\\
S_4&1\cdot0,\ 9\cdot\pi,\ 8\cdot\frac{2\pi}{3},\ 6\cdot\frac{\pi}{2},\\
A_5&1\cdot0,\ 15\cdot\pi,\ 20\cdot\frac{2\pi}{3},\ 12\cdot\frac{2\pi}{5},\ 12\cdot\frac{4\pi}{5}.
\end{array}
\]
Substitution yields
\[
\bigl(\dim W_{2m}^{A_4}\bigr)_{m=1}^{3}=(0,0,1),\qquad \bigl(\dim W_{2m}^{S_4}\bigr)_{m=1}^{4}=(0,0,0,1),
\]
and
\[
\bigl(\dim W_{2m}^{A_5}\bigr)_{m=1}^{6}=(0,0,0,0,0,1).
\] Thus the first positive invariant degrees are respectively $3$, $4$, and $6$.
\end{proof}

\begin{lem}\label{lem:reductive-subgroups}
Let $H$ be a proper closed reductive algebraic subgroup of $G=\PSLtwo$ with $\dim H>0$. Then the identity component $H^\circ$ is a maximal torus $T$ of $G$, and \(H=T\text{ or }H=N_G(T)=N(T).\) In particular, up to conjugacy these are the only positive-dimensional proper closed reductive algebraic subgroups of $G$.
\end{lem}

\begin{proof}
Since $H$ is reductive, its identity component $H^\circ$ is a connected reductive algebraic subgroup of $G$. Its derived subgroup $[H^\circ,H^\circ]$ is connected and semisimple. Suppose that $[H^\circ,H^\circ]$ is nontrivial. Since it is connected, it then has positive dimension, and hence \(\operatorname{Lie}[H^\circ,H^\circ]\neq0.\) Moreover, because $[H^\circ,H^\circ]$ is semisimple and is a closed algebraic subgroup of $G$, its Lie algebra is a nonzero semisimple Lie subalgebra of \(\operatorname{Lie}G=\mathfrak{sl}_2.\) A nonzero complex semisimple Lie algebra has dimension at least three, whereas $\dim\mathfrak{sl}_2=3$. Hence \(\operatorname{Lie}[H^\circ,H^\circ]=\mathfrak{sl}_2.\) Therefore $\dim[H^\circ,H^\circ]=3=\dim G$. Since $[H^\circ,H^\circ]$ is a closed connected subgroup of the connected algebraic group $G$, it follows that $[H^\circ,H^\circ]=G$, contradicting the properness of $H$. Thus $[H^\circ,H^\circ]=1$, so $H^\circ$ is a torus.

Because $\dim H>0$, the torus $H^\circ$ is positive dimensional. Since $G=\PSL_2(\C)$ has rank one, $H^\circ$ is therefore a maximal torus; write $T=H^\circ$. The identity component is normal in $H$, so every element of $H$ normalizes $T$. Hence \(T\subseteq H\subseteq N_G(T).\) Passing to quotients gives an injection \(H/T\hookrightarrow N_G(T)/T.\) Since the Weyl group of $G$ has order two, $N_G(T)/T\cong\Z/2\Z$. Thus $H/T$ is either trivial or all of $N_G(T)/T$. In the first case $H=T$, while in the second case $H=N_G(T)=N(T)$. Finally, all maximal tori of the connected reductive group $G$ are conjugate. Moreover, if $T'=gTg^{-1}$ for some $g\in G$, then \(N_G(T')=gN_G(T)g^{-1}.\) Hence the possibilities $T$ and $N_G(T)$ are unique up to conjugacy. This proves the last assertion.
\end{proof}

\begin{thm}\label{thm:closed-orbit}
Let $m\in{3,4,6}$. Suppose $V$ contains a nonisotropic primary vector $v$ generating an embedded $X_m$, and assume \(P_{j^2}(V)=0(0<j<m).\) Let \(U=\langle L(1,0),X_m\rangle_{\mathrm{VOA}}\). Then \(U\in\IndCMY(\Csq)\), so we may set \(B:=\widehat\Phi^{-1}(U)\in\CAlg(\Rep^{\mathrm{rat}}G)\). The algebra $B$ is finitely generated as a commutative $G$-algebra. Moreover, $\Spec B$ has a unique closed $G$-orbit \(Gx\cong G/H\), where \(H:=G_x\). Restriction to this orbit induces a surjective $G$-algebra homomorphism \(B\twoheadrightarrow\OO(G/H)\), and $$ H\text{ is conjugate in }G\text{ to }
\begin{cases}
A_4,&m=3,\\
S_4,&m=4,\\
A_5,&m=6.
\end{cases}
$$
\end{thm}

\begin{proof}
Theorem~\ref{thm:square-core} gives $U\in\IndCMY(\Csq)$. Under \(\widehat\Phi^{-1}\), the chosen embedded copy of $X_m$ corresponds to $E=W_{2m}\subset B$, and Lemma~\ref{lem:finite-generation} shows that $B$ is generated as a commutative $\C$-algebra by $E$. In particular, $B$ is finitely generated.

Under the tensor equivalence, \(\dim B^G=\dim\Hom_G(\C,B)=\dim\Hom_{L(1,0)}(X_0,U).\) Since $U\subseteq V$ and $V$ is of CFT type, $U_0=\C\one$. A Virasoro homomorphism $X_0=L(1,0)\to U$ is determined by the image of the vacuum top vector, so the last Hom-space has dimension at most one. The unit embedding $X_0=L(1,0)\hookrightarrow U$ is nonzero, hence this dimension is exactly one. Therefore \(B^G=\C\).

Let \(N=\sqrt{(0)}\subset B\). Since \(G\) acts on \(B\) by algebra automorphisms, \(N\) is \(G\)-stable. We first show that \((B/N)^G=\C\). Let \(\overline b\in(B/N)^G\) and choose a lift \(b\in B\). Since \(B\) is a rational \(G\)-module, \(b\) lies in a finite-dimensional rational \(G\)-submodule \(M\subset B\). If \(\overline M\) denotes the image of \(M\) in \(B/N\), then \(0\longrightarrow M\cap N\longrightarrow M\longrightarrow\overline M\longrightarrow0\) is an exact sequence of finite-dimensional rational \(G\)-modules. Since \(G\) is reductive over \(\C\), these modules are completely reducible, so the sequence splits \(G\)-equivariantly. Hence the invariant vector \(\overline b\in\overline M^G\) has a \(G\)-invariant lift in \(M\). Thus the natural map \(B^G\to(B/N)^G\) is surjective. Its kernel is \(B^G\cap N=N^G\). Since \(B^G=\C\) and no nonzero scalar is nilpotent, \(N^G=0\). Therefore $ (B/N)^G\cong B^G/N^G=\C. $

Every prime ideal of $B$ contains $N$. Hence the closed immersion $\Spec(B/N)\hookrightarrow\Spec B$ has image $V(N)=\Spec B$, and is therefore a homeomorphism on underlying topological spaces. Since $N$ is $G$-stable, this homeomorphism is $G$-equivariant; it therefore identifies $G$-orbits and their closures. Thus $\Spec(B/N)$ and $\Spec B$ have the same closed $G$-orbits.

Set $X=\Spec(B/N)$. It is a nonempty reduced affine $G$-variety of finite type, and \(X/\!/G=\Spec((B/N)^G)=\Spec\C.\) By the closed-orbit theorem recalled in Subsection~\ref{subsec:AG-conventions}, $X$ has a unique closed $G$-orbit. Choose $x$ on this orbit and set $H:=G_x$. Then $Gx\cong G/H$. Since $Gx$ is closed in the affine variety $X$, it is affine, and Matsushima's criterion recalled in the same subsection implies that $H$ is reductive. The closed immersion $Gx\hookrightarrow X$, together with the identification $Gx\cong G/H$, induces a surjection \(B/N=\OO(X)\twoheadrightarrow\OO(G/H).\) Composing with $B\twoheadrightarrow B/N$ gives a surjective $G$-algebra homomorphism \(\rho:B\twoheadrightarrow\OO(G/H).\)

We next prove that the generating copy $E=W_{2m}$ survives on the closed orbit. Since $G$ is reductive over $\mathbb C$, hence linearly reductive, $B$ admits a $G$-equivariant projection onto its trivial isotypic component $B^G=\mathbb C$: $p_0:B\longrightarrow B^G=\mathbb C.$ Put \(\beta:=p_0\circ\mu_B|_{E\otimes E}:E\otimes E\longrightarrow\C.\) Let $\pi_0:U\to X_0$ be the morphism corresponding to $p_0$ under $\widehat\Phi$. Let $i_0:X_0\hookrightarrow U$ denote the vacuum embedding. Under $\widehat\Phi$, it corresponds to the inclusion $B^G=\C\hookrightarrow B$. Since $p_0$ is a projection onto $B^G$, we have $\pi_0\circ i_0=\id_{X_0}.$ Hence $\pi_0$ is the identity on $U_0=\C\one$. Let $i_m:X_m\hookrightarrow U$ denote the chosen embedding corresponding under $\widehat\Phi$ to $E\hookrightarrow B$, and let \(\widehat J_{M,N}:\widehat\Phi(M)\boxtimes\widehat\Phi(N) \overset\sim\longrightarrow\widehat\Phi(M\otimes N)\) denote the symmetric monoidal constraint of $\widehat\Phi$. In the notation fixed in Section~\ref{sec:categorical}, let \(I_{\mu_U}:=\overline{\mu_U}\circ\boxtimes_{P(1)}:U\otimes U\to\overline U\) be the $P(1)$-intertwining map associated with the commutative-algebra-object multiplication $\mu_U:U\boxtimes U\to U$ in $\IndCMY(\Csq)$. Since the algebra-object structure on $U=\widehat\Phi(B)$ is transported from that on $B$, naturality of $\widehat J$ gives, after the fixed identifications $\widehat\Phi(E)=X_m$ and $\widehat\Phi(\C)=X_0$, \(\widehat\Phi(\beta)\circ\widehat J_{E,E} =\pi_0\circ\mu_U\circ(i_m\boxtimes i_m): X_m\boxtimes X_m\longrightarrow X_0\). Thus $\beta$ corresponds under the tensor equivalence to the displayed morphism. By naturality of the canonical $P(1)$-intertwining maps, its associated $P(1)$-intertwining map is \(\overline{\pi_0}\circ I_{\mu_U}\circ(i_m\otimes i_m): X_m\otimes X_m\longrightarrow\overline{X_0}\). We henceforth regard $X_m$ as its image under $i_m$. For the chosen highest-weight vector $v\in X_m$, of weight $h=m^2$, Lemma~\ref{lem:vacuum-channel} gives \(\operatorname{pr}_{U_0}I_{\mu_U}(v\otimes v)=v_{2h-1}v=(-1)^h(v,v)\one\ne0.\) Since $\pi_0$ is the identity on $U_0$, the corresponding $X_0$-valued intertwining map is nonzero. Hence $\beta\ne0$.

Since $B$ is commutative and the braiding on $\Rep^{\mathrm{rat}}G$ is the usual flip, $\mu_B(x\otimes y)=\mu_B(y\otimes x)$ for $x,y\in E$. Hence $\beta(x,y)=p_0(xy)=p_0(yx)=\beta(y,x),$ so $\beta$ is symmetric. The $G$-invariant form $\beta$ defines a nonzero $G$-homomorphism $E\to E^*$. Since $E$ and $E^*$ are irreducible and $E\cong E^*$, Schur's lemma implies that this map is an isomorphism. Hence $\beta$ is nondegenerate. Since $\beta$ is nondegenerate, it induces a $G$-equivariant isomorphism \(\beta^\sharp:E\overset{\sim}{\longrightarrow}E^*, x\longmapsto\beta(x,-)\). Let $q_\beta\in E\otimes E$ be the tensor corresponding to $(\beta^\sharp)^{-1}:E^*\to E$. Equivalently, if $\{e_i\}$ and $\{e^i\}$ are $\beta$-dual bases, then $q_\beta=\sum_i e_i\otimes e^i.$ Since $\beta$ is symmetric and $G$-invariant, $q_\beta\in(\Sym^2E)^G$, and $\beta(q_\beta)=\dim E=2m+1.$ As $\mu_B:B\otimes B\to B$ is $G$-equivariant, $\mu_B(q_\beta)\in B^G$. Hence, since $p_0|_{B^G}=\id$, \(\mu_B(q_\beta)=p_0(\mu_B(q_\beta))=\beta(q_\beta)=2m+1\ne0\), where the scalar is identified with the corresponding constant in $B$. If $\rho|_E=0$, then, since $q_\beta\in E\otimes E$, $(\rho\otimes\rho)(q_\beta)=0.$ As $\rho$ is a unital $G$-algebra homomorphism, $\rho\bigl(\mu_B(q_\beta)\bigr) = \mu_{\OO(G/H)} \bigl((\rho\otimes\rho)(q_\beta)\bigr) =0.$ On the other hand, $\mu_B(q_\beta)=(2m+1)1_B$, so $\rho\bigl(\mu_B(q_\beta)\bigr) =(2m+1)1_{\OO(G/H)}\ne0,$ a contradiction. Thus $\rho|_E\ne0$. Since $\rho|_E$ is a nonzero $G$-homomorphism and $E$ is irreducible, its kernel is zero; hence $\rho(E)\cong E$. By Lemma~\ref{lem:admissible-homogeneous}, the multiplicity of $E=W_{2m}$ in $\OO(G/H)$ is $\dim E^H$. Therefore $\dim E^H>0$, i.e. $E^H\ne0$.

We claim that no $W_{2j}$ with $0<j<m$ has an $H$-fixed vector. Suppose to the contrary that $W_{2j}^H\ne0$. By Lemma~\ref{lem:admissible-homogeneous}, $\OO(G/H)$ contains a $G$-submodule $W'\cong W_{2j}$. Choose a basis $w_1,\ldots,w_r$ of $W'$ and elements $b_i\in B$ such that $\rho(b_i)=w_i$. Since $B$ is a rational $G$-module, each $b_i$ is contained in a finite-dimensional rational $G$-submodule $M_i\subset B$. Hence $M:=M_1+\cdots+M_r$ is a finite-dimensional rational $G$-submodule of $B$. Since $\rho$ is $G$-equivariant, $\rho(M)$ is a finite-dimensional rational $G$-submodule, and it contains every $w_i$; thus $W'\subseteq\rho(M)$. Complete reducibility of the finite-dimensional rational $G$-module $\rho(M)$ gives a $G$-equivariant projection $r:\rho(M)\to W'$ with $r|_{W'}=\id_{W'}$. Hence $q:=r\circ\rho|_M:M\twoheadrightarrow W'$ is a surjective homomorphism of finite-dimensional rational $G$-modules. Since $G$ is reductive, $q$ admits a $G$-equivariant section, and therefore $B$ contains a $G$-submodule isomorphic to $W'\cong W_{2j}$. Applying the exact equivalence $\widehat\Phi$ gives an embedded copy of $X_j=L(1,j^2)$ in $U\subseteq V$. The image of a nonzero highest-weight vector of $X_j$ is a nonzero vector of $P_{j^2}(V)$, contradicting $P_{j^2}(V)=0$. Hence $W_{2j}^H=0$ for every $0<j<m$. Together with $W_{2m}^H=E^H\ne0$, this shows that $m$ is the first positive invariant degree of $H$.

The subgroup $H$ is proper: if $H=G$, then the nontrivial irreducible $G$-module $E$ would have $E^H=E^G=0$, contrary to the preceding paragraph. If $H$ has positive dimension, Lemma~\ref{lem:reductive-subgroups} gives $H=T$ or $H=N(T)$ for a maximal torus $T$. Their first positive invariant degrees are $1$ and $2$, respectively, by Lemma~\ref{lem:first-invariant-degrees}, so neither can occur because $m\in\{3,4,6\}$. Thus $H$ is zero dimensional, hence finite. The classical classification of finite subgroups of $\PSLtwo$, up to conjugacy, gives the cyclic, dihedral, tetrahedral, octahedral, and icosahedral groups. Lemma~\ref{lem:first-invariant-degrees} now identifies $H$, up to conjugacy, as $A_4$, $S_4$, or $A_5$ according as $m=3$, $4$, or $6$.
\end{proof}

\subsection{Polyhedral multiplicities and exceptional rigidity}

For later use we record the first spherical multiplicities. For $H=A_4,S_4,A_5$ set \(a_m(H)=\dim W_{2m}^H.\) The character average in the proof of Lemma~\ref{lem:first-invariant-degrees} gives
\[
\begin{array}{c|rrrrrrr}
H\backslash m&1&2&3&4&5&6&7\\ \hline
A_4&0&0&1&1&0&2&1\\
S_4&0&0&0&1&0&1&0\\
A_5&0&0&0&0&0&1&0.
\end{array}
\]
Summing the character averages uses \(\sum_{m\ge0}\chi_m(\theta)t^m=\frac{1+t}{1-2t\cos\theta+t^2}\). Substitution of the three angle distributions and elementary simplification give the Hilbert series \(\sum_{m\ge0}a_m(A_4)t^m=\frac{1-t^{12}}{(1-t^3)(1-t^4)(1-t^6)},\) \(\sum_{m\ge0}a_m(S_4)t^m=\frac{1-t^{18}}{(1-t^4)(1-t^6)(1-t^9)},\) \(\sum_{m\ge0}a_m(A_5)t^m=\frac{1-t^{30}}{(1-t^6)(1-t^{10})(1-t^{15})}.\)

\begin{cor}\label{cor:exceptional-first-simple}
Assume \(\ch V=\ch V_{A_1}^{H},H\in\{A_4,S_4,A_5\}.\) Then the first positive primary weight is \(9,16,36\), according as \(H=A_4,S_4,A_5\). Moreover, if $h$ denotes this weight, then \(P_h(V)\) is one-dimensional, and any nonzero vector in it generates \(X_3,X_4,X_6\), respectively.
\end{cor}

\begin{proof}
Write \(V_{A_1}^H\cong\bigoplus_{j\ge0}a_j(H)X_j\) and let $m=3,4,6$ for $H=A_4,S_4,A_5$, respectively. The table gives \(a_j(H)=0(0<j<m),a_m(H)=1.\) Since the lowest conformal weight of $X_j$ is $j^2$, the character identity and the inclusion $L(1,0)\subset V$ imply \(V_n=L(1,0)_n(0\le n<m^2),\dim V_{m^2}=\dim L(1,0)_{m^2}+1.\) The simple vacuum module $L(1,0)$ has no positive-weight primary vectors: such a vector would generate a nonzero proper highest-weight submodule. Hence there is no positive primary vector in $V$ below weight $m^2$. Lemma~\ref{lem:first-primary} therefore shows that \(P_{m^2}(V)\) is one-dimensional and that every nonzero \(v\in P_{m^2}(V)\) is a nonisotropic primary vector.

Set $H_m=(m+1)^2$. At conformal weight $H_m$, no summand $X_j$ with $j>m+1$ contributes, while the positive summands with $j\le m+1$ are the single copy of $X_m$ and $a_{m+1}(H)$ copies of $X_{m+1}$. The top level of $X_{m+1}$ is one-dimensional, and hence \(\dim V_{H_m}=\dim L(1,0)_{H_m}+\dim(X_m)_{H_m}+a_{m+1}(H).\) Now \(a_4(A_4)=1,\ a_5(S_4)=0,\ a_7(A_5)=0\). Lemma~\ref{lem:first-singular} applies with $\varepsilon=1,0,0$ in the three cases, proving that the cyclic module generated by any such nonzero $v$ is respectively $X_3,X_4,X_6$.
\end{proof}

\begin{thm}\label{thm:exceptional-rigidity}
Let $H\in\{A_4,S_4,A_5\}$. Let $V$ be a simple self-contragredient VOA of CFT type and central charge one. If \(\ch V=\ch V_{A_1}^{H},\) then \(V\cong V_{A_1}^{H}.\)
\end{thm}

\begin{proof}
Let $m_0=3,4,6$ for $H=A_4,S_4,A_5$, respectively. By Corollary~\ref{cor:exceptional-first-simple}, $V$ contains a nondegenerate embedded $X_{m_0}$ and satisfies \(P_{j^2}(V)=0\) for $0<j<m_0$. Set \(U=\langle L(1,0),X_{m_0}\rangle_{\mathrm{VOA}}.\) Theorem~\ref{thm:closed-orbit} gives \(U\in\IndCMY(\Csq)\), a commutative rational $G$-algebra \(B=\widehat\Phi^{-1}(U)\), and a surjective $G$-algebra homomorphism \(B\twoheadrightarrow\OO(G/H')\), where $H'$ is conjugate in $G$ to $H$. Choose $g\in G$ with $H'=gHg^{-1}$. Then \(G/H'\to G/H,\ xH'\mapsto xgH\) is a $G$-equivariant isomorphism, and hence induces a $G$-algebra isomorphism \(\OO(G/H')\cong\OO(G/H)\). Composing with this isomorphism, we may fix a surjective $G$-algebra homomorphism \(\rho:B\twoheadrightarrow\OO(G/H).\) For $j\ge0$, put \(b_j:=\dim\Hom_G(W_{2j},B).\) Lemma~\ref{lem:ind-equivalence} then gives \(U\cong\bigoplus_{j\ge0}X_j\otimes\Hom_G(W_{2j},B)\cong\bigoplus_{j\ge0}b_jX_j.\) On the other hand, taking $H$-fixed points in the standard $G\times L(1,0)$-module decomposition \(V_{A_1}\cong\bigoplus_{j\ge0}W_{2j}\otimes X_j\) gives \(V_{A_1}^{H}\cong\bigoplus_{j\ge0}a_j(H)X_j.\) Each $b_j$ is finite, since every copy of $X_j$ contributes an independent lowest-weight vector to the finite-dimensional space $U_{j^2}\subseteq V_{j^2}$. Moreover, at any fixed conformal weight $n$, only the summands with $j^2\le n$ can contribute. Thus all character comparisons below are legitimate coefficient by coefficient.

We claim that \(b_j\ge a_j(H)\) for every $j$. By Lemma~\ref{lem:ind-equivalence}, \(\Rep^{\mathrm{rat}}G\) is a semisimple abelian category. Hence the surjection \(\rho:B\twoheadrightarrow\OO(G/H)\) splits $G$-equivariantly, so \(\OO(G/H)\) embeds as a $G$-submodule of $B$. Let $T_j\subset\OO(G/H)$ be the $W_{2j}$-isotypic component. By Peter--Weyl, \(T_j\cong W_{2j}^{\oplus a_j(H)}\), and therefore \(T_j\hookrightarrow B\). Consequently, \(\dim\Hom_G(W_{2j},B)\ge\dim\Hom_G(W_{2j},T_j)=a_j(H),\) that is, \(b_j\ge a_j(H)\), as claimed.

All coefficients in the characters of the $X_j$ are nonnegative, so \(b_j\ge a_j(H)\) for all $j$ implies \(\ch U\ge\ch V_{A_1}^{H}\) coefficientwise. On the other hand, $U\subseteq V$ gives \(\ch U\le\ch V\), while by hypothesis \(\ch V=\ch V_{A_1}^{H}\). Hence \(\ch U=\ch V=\ch V_{A_1}^{H}.\) Since $U_n\subseteq V_n$ and the two spaces have the same dimension for every $n$, we have $U_n=V_n$ for all $n$, and therefore $U=V$.

Put \(K=\ker\rho\). Exactness of $\widehat\Phi$ gives an exact sequence in $\IndCMY(\Csq)$, \(0\longrightarrow\widehat\Phi(K)\longrightarrow U \longrightarrow\widehat\Phi(\OO(G/H))\longrightarrow0\). By Corollary~\ref{cor:fixed-point-dictionary}, the last term is $V_{A_1}^{H}$ as an $L(1,0)$-module. All three objects lie in $\IndCMY(\Csq)$, so $L(0)$ acts semisimply on them, and the maps commute with $L(0)$. Consequently the preceding exact sequence restricts, for each $n\ge0$, to \(0\longrightarrow\widehat\Phi(K)_n\longrightarrow U_n \longrightarrow (V_{A_1}^{H})_n\longrightarrow0\). Since \(\ch U=\ch V_{A_1}^{H}\), it follows that \(\dim\widehat\Phi(K)_n=0\) for every $n\ge0$. Every object of $\IndCMY(\Csq)$ is an algebraic direct sum of the $X_j$, so its $L(0)$-spectrum is contained in $\Z_{\ge0}$. Hence \(\widehat\Phi(K)=0\). Because $\widehat\Phi$ is an equivalence, $K=0$. Thus $\rho$ is an isomorphism \(B\cong\OO(G/H)\) of commutative rational $G$-algebras. Applying the symmetric monoidal equivalence $\widehat\Phi$, the commutative algebra object corresponding to $U=V$ is therefore isomorphic to \(\widehat\Phi(\OO(G/H))\). Corollary~\ref{cor:fixed-point-dictionary} now yields \(V\cong V_{A_1}^{H}\) as VOA extensions of $L(1,0)$, and hence as vertex operator algebras.
\end{proof}

\begin{cor}\label{cor:A5-gap}
Under the hypotheses of Theorem~\ref{thm:main}, suppose that $ P_h(V)=0 $ for every integer $1\le h<36$. Then \(V\cong V_{A_1}^{A_5}\).
\end{cor}

\begin{proof}
By Theorem~\ref{thm:character-reduction}, \(\ch V\in \left\{ \ch V_L,\ \ch V_L^+,\ \ch V_{A_1}^{A_4},\ \ch V_{A_1}^{S_4},\ \ch V_{A_1}^{A_5} \right\}\). Suppose first that \(\ch V=\ch V_L\). Since \(\ch V=\ch V_L\), we have $ \dim V_n=\dim (V_L)_n $ for all $n\ge0. $ Because \(L\) has rank one, the Heisenberg subVOA \(M(1)\subset V_L\) contains a nonzero weight-one vector, namely \(\alpha(-1)\mathbf 1\) for any nonzero \(\alpha\in L\otimes_{\mathbb Z}\mathbb C\). Hence $ \dim V_1=\dim (V_L)_1>0. $ On the other hand, \(V\) is of CFT type, so $ V_0=\mathbb C\mathbf1=L(1,0)_0, $ whereas $ L(1,0)_1=0 $ since \(L(-1)\mathbf1=0\). Thus the hypotheses of Lemma~\ref{lem:primary-complement} are satisfied with \(h=1\). It follows that $ V_1=L(1,0)_1\oplus P_1(V), $ and therefore $ \dim P_1(V) =\dim V_1-\dim L(1,0)_1 =\dim V_1>0. $ Hence \(P_1(V)\ne0\), contradicting the assumption that \(P_h(V)=0\) for \(1\le h<36\).

Next suppose that $ \ch V=\ch V_L^+,$ $L=\mathbb Z\alpha, $ $\langle\alpha,\alpha\rangle=2k. $ Since \(L\) is a positive-definite even rank-one lattice, \(k\) is a positive integer. By Proposition~\ref{prop:rank-one-lattice-package}(i)--(ii), the fixed-point VOA \(V_L^+\) contains the vacuum Virasoro subVOA \(L(1,0)\), and its contributions outside this vacuum Virasoro module arise from two sources. First, the Heisenberg fixed-point subVOA \(M(1)^+\) has no contribution outside \(L(1,0)\) in weights \(<4\), whereas at weight \(4\) it contains a nonzero Virasoro-primary contribution. Second, the nonzero lattice sectors of \(V_L^+\) have lowest conformal weight \(k\): indeed, the vectors $ e^\alpha+e^{-\alpha}\in V_L^+ $ have conformal weight $ \frac{\langle\alpha,\alpha\rangle}{2}=k, $ and no nonzero lattice vector has smaller conformal weight. Set $ h=\min\{k,4\}. $ It follows that, for every \(0\le n<h\), no contribution to \(V_L^+\) outside the vacuum Virasoro module can occur, and hence $ (V_L^+)_n=L(1,0)_n. $ At weight \(h\), however, there is a nonzero contribution outside \(L(1,0)\): if \(k<4\), it is supplied by the lattice sector generated by \(e^\alpha+e^{-\alpha}\); if \(k>4\), it is supplied by the first nonvacuum Virasoro summand of \(M(1)^+\); and if \(k=4\), both sources occur in weight \(4\). Consequently, $ \dim (V_L^+)_h>\dim L(1,0)_h. $ Since \(\ch V=\ch V_L^+\), we have $ \dim V_n=\dim (V_L^+)_n (n\ge0). $ Therefore $ \dim V_n=\dim L(1,0)_n (0\le n<h), $ while $ \dim V_h>\dim L(1,0)_h. $ Lemma~\ref{lem:primary-complement} is thus applicable at the first such weight \(h\), and gives $ V_h=L(1,0)_h\oplus P_h(V) $ (and, with the notation of that lemma, this is an orthogonal direct sum). Hence $ \dim P_h(V) = \dim V_h-\dim L(1,0)_h >0. $ Thus \(P_h(V)\ne0\). Since \(k\ge1\), $ 1\le h=\min\{k,4\}\le4<36, $ contradicting the hypothesis that \(P_j(V)=0\) for every \(1\le j<36\).

If \(\ch V=\ch V_{A_1}^{A_4}\) or \(\ch V=\ch V_{A_1}^{S_4}\), Corollary~\ref{cor:exceptional-first-simple} gives \(P_9(V)\ne0\) or \(P_{16}(V)\ne0\), respectively, again a contradiction. Thus necessarily \(\ch V=\ch V_{A_1}^{A_5}.\) Theorem~\ref{thm:exceptional-rigidity} now yields \(V\cong V_{A_1}^{A_5}\).
\end{proof}

\section{The weight-four core and reconstruction of \texorpdfstring{$M(1)^+$}{M(1)+}}\label{sec:weight4}

\subsection{The spin-two matrix model}

Retain \(G=\PSLtwo\cong\SO_3(\C)\), and let $ E:=W_4\cong\Sym^2_0(\C^3) $ be its five-dimensional irreducible (spin-two) representation, realized as the space of traceless symmetric \(3\times3\) matrices with $ g\cdot A=gAg^{-1}=gAg^{\mathsf T}, g\in \SO_3(\C). $ The trace pairing $ \tau(A,B):=\operatorname{tr}(AB) $ is a nondegenerate \(G\)-invariant symmetric bilinear form on \(E\). Let $\tau^\flat:E\overset{\sim}{\longrightarrow}E^*,\tau^\flat(A)=\tau(A,-),$ and set $ \iota_E:=(\tau^\flat)^{-1}:E^*\overset{\sim}{\longrightarrow}E. $ Define $ q(A):=\operatorname{tr}(A^2) $ and the quadratic \(G\)-equivariant polynomial map $ Q(A):=A^2-\frac13\operatorname{tr}(A^2)I_3. $ Its polarization is the symmetric bilinear \(G\)-equivariant map \(\widetilde Q(A,B):= \frac12\bigl(Q(A+B)-Q(A)-Q(B)\bigr) = \frac12(AB+BA)-\frac13\operatorname{tr}(AB)I_3\), so that \(\widetilde Q(A,A)=Q(A)\). By the Clebsch--Gordan decomposition, $ \Sym^2E\cong W_0\oplus W_4\oplus W_8. $ Consequently, $ \dim(\Sym^2E^*)^G=1, \dim\Hom_G(\Sym^2E,E)=1. $ Since \(q\) and \(\widetilde Q\) are nonzero, they span these one-dimensional spaces: $ (\Sym^2E^*)^G=\C q, \Hom_G(\Sym^2E,E)=\C\widetilde Q. $

Fix the orthogonal decomposition $\C^3=\C e_1\oplus (\C e_1)^\perp, (\C e_1)^\perp=\operatorname{span}_{\C}\{e_2,e_3\},$ and set \(SO_2(\C):=\ker\!\left(\det:O_2(\C)\to\{\pm1\}\right).\) We use the maximal torus \(T:=\{\operatorname{diag}(1,R):R\in SO_2(\C)\}\leq G\cong SO_3(\C), N(T):=N_G(T)\).

\begin{lem}\label{lem:quadratic-orbit}
Let $c\in\C^\times$ and $\lambda\in\C$. Define the ideal \(I_{c,\lambda}:=\bigl(q-c,\ \ell\circ(Q-\lambda\id_E):\ell\in E^*\bigr)\subseteq\Sym(E^*)\) and the corresponding closed affine subscheme \(Z_{c,\lambda}:=\Spec\bigl(\Sym(E^*)/I_{c,\lambda}\bigr)\subseteq E.\) If $Z_{c,\lambda}$ is nonempty, then \(\lambda\ne0\) and \(c=6\lambda^2\), and, scheme-theoretically, \(Z_{c,\lambda}=G\cdot\operatorname{diag}(2\lambda,-\lambda,-\lambda)\cong G/N(T)\). In particular, $Z_{c,\lambda}$ is smooth and reduced.
\end{lem}

\begin{proof}
Assume that $Z_{c,\lambda}$ is nonempty and choose a closed point $A\in Z_{c,\lambda}(\C)$, using the conventions of Subsection~\ref{subsec:AG-conventions}. The defining equation $Q(A)=\lambda A$, together with $q(A)=c$, is equivalent to \(A^2-\lambda A-\frac c3I_3=0\). If $\lambda=0$, then $A^2=(c/3)I_3$. Since $c\ne0$, every eigenvalue of $A$ is one of the two nonzero numbers $\pm\sqrt{c/3}$, and three such eigenvalues cannot sum to zero. This contradicts $\operatorname{tr}(A)=0$, so $\lambda\ne0$.

Equation (4.1) shows that the minimal polynomial of $A$ divides $x^2-\lambda x-c/3$, so $A$ has at most two eigenvalues. It cannot have only one eigenvalue: tracelessness would then force that eigenvalue to be $0$, whereas substituting $0$ into the polynomial gives $c=0$. Thus $A$ has exactly two distinct eigenvalues, of algebraic multiplicities $1$ and $2$. Write them as $a$ and $b$, respectively. Then \(a+2b=0,a+b=\lambda,ab=-\frac c3,\) and hence \(b=-\lambda,a=2\lambda,c=6\lambda^2.\) The two roots are distinct, so $A$ is semisimple. Since $A$ is symmetric, eigenspaces for distinct eigenvalues are orthogonal. The one-dimensional $2\lambda$-eigenspace is nondegenerate: if a nonzero generator were isotropic, it would be orthogonal both to itself and to the $-\lambda$-eigenspace, hence to all of $\C^3$, contradicting the nondegeneracy of the standard symmetric bilinear form. Its orthogonal complement, which is the two-dimensional $-\lambda$-eigenspace, is therefore nondegenerate as well. Choosing orthonormal bases in the two eigenspaces and changing the sign of one basis vector if necessary, we obtain an orthogonal eigenbasis of determinant one. Thus $A$ is conjugate under $G\cong SO_3(\C)$ to \(A_0:=\operatorname{diag}(2\lambda,-\lambda,-\lambda).\) Since $q$ is $G$-invariant and $Q$ is $G$-equivariant, the ideal $I_{c,\lambda}$ is $G$-stable and hence so is $Z_{c,\lambda}$. The preceding argument gives $Z_{c,\lambda}(\C)\subseteq G\cdot A_0$, while $A_0\in Z_{c,\lambda}(\C)$ and $G$-stability give the reverse inclusion. Therefore \(Z_{c,\lambda}(\C)=G\cdot A_0.\)

It remains to identify the stabilizer. Put \(L:=\C e_1,W:=L^\perp=\langle e_2,e_3\rangle.\) Since $A_0$ acts by $2\lambda$ on $L$ and by $-\lambda$ on $W$, every $g\in G_{A_0}$ preserves the two eigenspaces $L$ and $W$. Relative to the orthogonal decomposition $\C^3=L\oplus W$, the restriction $g|_W$ lies in $O_2(\C)$. Since $SO_2(\C)$ is a normal subgroup of $O_2(\C)$, the element $g|_W$ normalizes $SO_2(\C)$; hence $g$ normalizes $T$, and therefore $G_{A_0}\subseteq N(T)$. Conversely, $L$ is precisely the common fixed-point space $(\C^3)^T$. Thus every element of $N(T)$ preserves $L$, and hence also $W=L^\perp$. Since $A_0$ acts as a scalar on each summand, every such element commutes with $A_0$. Therefore \(G_{A_0}=N(T).\) By the orbit and Weyl-group facts recalled in Subsection~\ref{subsec:AG-conventions}, \(G/N(T)\cong G\cdot A_0,\) and, since $N(T)/T$ is finite, \(\dim G\cdot A_0=\dim G-\dim N(T)=2.\)

We next determine the scheme structure. Let $B=(b_{ij})\in E$ be a tangent vector at $A_0$. Linearizing the equations $q(A)=c$ and $Q(A)=\lambda A$ gives \(2\operatorname{tr}(A_0B)=0,A_0B+BA_0-\frac23\operatorname{tr}(A_0B)I_3=\lambda B.\) Set $a_1=2\lambda$ and $a_2=a_3=-\lambda$, so that $A_0=\operatorname{diag}(a_1,a_2,a_3)$. The first linearized equation implies $\operatorname{tr}(A_0B)=0$, and therefore the $(i,j)$-entry of the second equation is \((a_i+a_j-\lambda)b_{ij}=0.\) Since $\lambda\ne0$, this gives \(b_{11}=b_{22}=b_{33}=b_{23}=0,\) whereas the coefficients of $b_{12}$ and $b_{13}$ vanish identically; thus $b_{12}$ and $b_{13}$ are free. Consequently
\[
T_{A_0}Z_{c,\lambda}
=
\left\{
\begin{pmatrix}
0&u&v\\
u&0&0\\
v&0&0
\end{pmatrix}
:u,v\in\C
\right\},
\qquad
\dim_\C T_{A_0}Z_{c,\lambda}=2.
\]
Because the $G$-action preserves $Z_{c,\lambda}$ and is transitive on its closed points, the same tangent-space dimension holds at every closed point $x\in Z_{c,\lambda}$. Moreover, every such $x$ lies in $G\cdot A_0$. Since $I_{c,\lambda}$ is $G$-stable and $A_0\in Z_{c,\lambda}(\C)$, the orbit morphism \(G\longrightarrow E,g\longmapsto g\cdot A_0,\) factors through $Z_{c,\lambda}$. Hence the two-dimensional orbit $G\cdot A_0$ is a locally closed subvariety of $Z_{c,\lambda}$ through $x$, whence \(2\leq\dim_x Z_{c,\lambda}\leq\dim_\C T_xZ_{c,\lambda}=2.\) Consequently, \(\dim\mathcal O_{Z_{c,\lambda},x}=\dim_x Z_{c,\lambda}=2=\dim_\C\mathfrak m_x/\mathfrak m_x^2.\) By the regularity and smoothness criteria of Subsection~\ref{subsec:AG-conventions}, every closed point of $Z_{c,\lambda}$ is smooth. Since the smooth locus is open and finite-type $\C$-schemes are Jacobson, $Z_{c,\lambda}$ is smooth, and hence reduced.

Finally, $ G\cdot A_0=Z_{c,\lambda}(\C) $ is Zariski closed in the affine space \(E\), so the orbit is a closed subvariety of \(E\). Give it the reduced induced scheme structure. Then \(Z_{c,\lambda}\) and \(G\cdot A_0\) are reduced closed subschemes of \(E\) with the same \(\C\)-points. By the Nullstellensatz consequence recalled in Subsection~\ref{subsec:AG-conventions}, they coincide scheme-theoretically. Together with the orbit--stabilizer identification above, $ Z_{c,\lambda}=G\cdot A_0\cong G/N(T), $ as claimed.
\end{proof}

\begin{thm}\label{thm:M1plus-core}
Let $V$ be a simple self-contragredient VOA of CFT type and central charge one. Suppose $J\in V_4$ is a nonisotropic primary vector, its cyclic Virasoro module is \(U(\Vir)J\cong X_2=L(1,4),\) and \(\dim P_4(V)=1.\) Then \(\langle L(1,0),J\rangle_{\mathrm{VOA}}\cong M(1)^+.\)
\end{thm}

\begin{proof}
Set \(U_4:=\langle L(1,0),J\rangle_{\mathrm{VOA}}.\) By Theorem~\ref{thm:square-core}, $U_4\in\IndCMY(\Csq)$. Via the symmetric monoidal equivalence \(\widehat\Phi^{-1}:\IndCMY(\Csq)\to\Rep^{\mathrm{rat}}G,\) let \(B:=\widehat\Phi^{-1}(U_4),\) regarded as a commutative rational $G$-algebra. By Lemma~\ref{lem:finite-generation}, $B$ is generated by the copy $E=W_4$ corresponding to the cyclic summand $U(\Vir)J\cong X_2$.

Under $\widehat\Phi$, the trivial $G$-module corresponds to $X_0=L(1,0)$. Every copy of $X_0$ in $U_4$ contributes a nonzero weight-zero vector, while $(U_4)_0=\C\one$. Since $\widehat\Phi(\C)=X_0$, full faithfulness of $\widehat\Phi$ gives \(\Hom_G(\C,B)\cong\Hom_{\IndCMY(\Csq)}(X_0,U_4).\) Thus the multiplicity of the trivial $G$-module in $B$ equals the multiplicity of $X_0$ in $U_4$. Every copy of $X_0$ contributes a nonzero weight-zero vector, while $(U_4)_0=\C\one$. Since the vacuum submodule $L(1,0)=X_0$ is already contained in $U_4$, the multiplicity is exactly one. Hence \(B^G\cong\Hom_G(\C,B)=\C.\) Similarly, since $\widehat\Phi(E)=X_2$, \(\Hom_G(E,B)\cong\Hom_{\IndCMY(\Csq)}(X_2,U_4).\) Hence the multiplicity of $E$ in $B$ equals the multiplicity of $X_2$ in $U_4$. The highest-weight line of every such copy lies in \(P_4(U_4)=U_4\cap P_4(V)\), which has dimension at most one, whereas $J$ supplies one such copy. Thus $E$ occurs in $B$ with multiplicity one. Let \(j:E\hookrightarrow B\) denote this generating copy.

Since $B$ is generated by $E$, the map \(j\circ\iota_E:E^*\longrightarrow B\) extends uniquely to a surjective $G$-algebra homomorphism \(\pi:\Sym(E^*)\twoheadrightarrow B.\) Since $q$ is $G$-invariant and $B^G=\C$, there is a scalar $c\in\C$ such that $\pi(q)=c\one$. We claim that $c\ne0$. Choose a $\tau$-orthonormal basis $e_1,\ldots,e_5$ of $E$ and put $x_i:=\tau(e_i,-)\in E^*$. Then \(q=\sum_{i=1}^5 x_i^2,\iota_E(x_i)=e_i,\) and hence \(\pi(q)=\sum_{i=1}^5 j(e_i)j(e_i).\) The tensor \(\sum_{i=1}^5 e_i\otimes e_i\) spans the unique trivial $G$-submodule of $\Sym^2E$. Under the symmetric tensor equivalence, the restriction of the multiplication of $B$ to this trivial summand corresponds to the vacuum-channel component of the product of the $X_2$-summand of $U_4$ with itself. This component is nonzero by Lemma~\ref{lem:vacuum-channel}, because $J$ is nonisotropic. Therefore \(\pi(q)\ne0\), and hence $c\ne0$.

To obtain the quadratic covariant relation without making any normalization choice for the $W_4$ component of the multiplication, define the $G$-equivariant linear map \(Q^*:E^*\longrightarrow\Sym^2(E^*),\ Q^*(\ell)=\ell\circ Q.\) The composite $\pi\circ Q^*:E^*\to B$ is $G$-equivariant. Since $E^*\cong E$ and $E$ occurs with multiplicity one in $B$, there is a unique scalar $\lambda\in\C$ such that \(\pi\circ Q^*=\lambda(j\circ\iota_E).\) Because $\pi|_{E^*}=j\circ\iota_E$, this is equivalently \(\pi\bigl(Q^*(\ell)-\lambda\ell\bigr)=0(\ell\in E^*).\) Thus $q-c$ and every coordinate function of the polynomial map $Q-\lambda\id_E$ lie in $\ker\pi$. Therefore the ideal $I_{c,\lambda}$ of Lemma~\ref{lem:quadratic-orbit} is contained in $\ker\pi$, and $\pi$ factors through a surjective $G$-algebra homomorphism \(R_{c,\lambda}:=\frac{\Sym(E^*)}{I_{c,\lambda}}\twoheadrightarrow B.\) Since this map is unital and $B\ne0$, the algebra $R_{c,\lambda}$ is a nonzero finitely generated $\C$-algebra. By the Nullstellensatz its spectrum has a closed point, so Lemma~\ref{lem:quadratic-orbit} applies and gives \(R_{c,\lambda}\cong\OO(G/N(T)).\) By Lemma~\ref{lem:G-simple-homogeneous}, this coordinate algebra has no nonzero proper $G$-stable ideal. Hence the kernel of the displayed surjection is zero and \(B\cong\OO(G/N(T)).\)

Lemma~\ref{lem:quadratic-orbit} also realizes $G/N(T)$ as a closed subvariety of the affine space $E$, so $G/N(T)$ is affine. Since $N(T)$ is closed, Matsushima's criterion recalled in Subsection~\ref{subsec:AG-conventions} shows that $N(T)$ is reductive. Thus Corollary~\ref{cor:fixed-point-dictionary} applies and yields \(U_4\cong V_{A_1}^{N(T)}.\) All maximal tori of $G$ are conjugate, so, after conjugating $T$ if necessary, we may take its maximal compact torus to be the standard lattice torus $T_{A_1}$ of Proposition~\ref{prop:rank-one-lattice-package}(v). The lattice involution $\theta$ conjugates every element of $T_{A_1}$ to its inverse, and hence represents the nontrivial Weyl-group element. Consequently \(K_{A_1}:=T_{A_1}\rtimes\langle\theta\rangle\cong O(2)\) is the standard maximal compact subgroup of the algebraic normalizer $N(T)\cong O_2(\C)$. By the compact--complexification convention of Remark~\ref{rmk:compact-complexification}, $N(T)$ and $K_{A_1}$ have the same fixed vectors on each finite-dimensional homogeneous subspace of $V_{A_1}$. Therefore \(V_{A_1}^{N(T)}=V_{A_1}^{K_{A_1}}=M(1)^+\) by Proposition~\ref{prop:rank-one-lattice-package}(v). Hence \(U_4\cong M(1)^+\), as required.
\end{proof}

\subsection{Recovering the \texorpdfstring{$L(1,4)$}{L(1,4)} summand from the lattice-orbifold character}

Let \(L=\Z\alpha,\ \langle\alpha,\alpha\rangle=2k,\ k\in\Z_{>0}\). By Proposition~\ref{prop:rank-one-lattice-package}(i)--(ii), \(V_L^+\cong M(1)^+\oplus\bigoplus_{r\ge1}M(1,r\alpha)\) and \(M(1)^+\cong\bigoplus_{j\ge0}X_{2j}\), and the first lattice sector has one-dimensional top in conformal weight $k$.

\begin{lem}\label{lem:k3-weight3}
Let $\Lambda_3=\Z\alpha$ with $\langle\alpha,\alpha\rangle=6$ and assume $\ch V=\ch V_{\Lambda_3}^+$. Then $P_3(V)$ is one-dimensional. Every nonzero $F\in P_3(V)$ satisfies \((F,F)\ne0\) and \(U(\Vir)F\cong L(1,3).\)
\end{lem}

\begin{proof}
Both VOAs have central charge one, so equality of vacuum characters is equivalent to equality of graded characters. By Proposition~\ref{prop:rank-one-lattice-package}(i)--(ii), below weight four the character of $V_{\Lambda_3}^+$ consists of the vacuum module $X_0=L(1,0)$ together with the one-dimensional top of the charge-one lattice sector in weight three. Hence $V_n=L(1,0)_n$ for $0\le n<3$ and $\dim V_3=\dim L(1,0)_3+1$. Lemma~\ref{lem:first-primary} gives $\dim P_3(V)=1$, and every nonzero $F\in P_3(V)$ satisfies $(F,F)\ne0$.

For such an $F$, the definition of $P_3(V)$ gives \(L(0)F=3F\) and \(L(n)F=0\) for all $n>0$. Thus $F$ is a Virasoro highest-weight vector of central charge one and highest weight three. By the universal property of the Verma module, there is a nonzero Virasoro homomorphism \(\phi_F:V(1,3)\longrightarrow V\) sending the highest-weight vector to $F$, and \(\operatorname{Im}\phi_F=U(\Vir)F.\) By Proposition~\ref{prop:c1-verma}, $V(1,h)$ is reducible precisely when \(h=n^2/4\) for some $n\in\Z_{\ge0}$. Since \(3\ne n^2/4\) for every $n\in\Z_{\ge0}$, the Verma module $V(1,3)$ is irreducible. Hence the canonical quotient \(V(1,3)\twoheadrightarrow L(1,3)\) is an isomorphism. Since $\phi_F$ is nonzero and its domain is simple, $\phi_F$ is injective. Therefore \(U(\Vir)F=\operatorname{Im}\phi_F\cong L(1,3).\)
\end{proof}

\begin{lem}\label{lem:intervening}
Assume $k=3$ or $k>4$ and \(\ch V=\ch V_L^+\) for a simple self-contragredient CFT-type VOA $V$ of central charge one. Then \(\dim P_4(V)=1\), the restriction of the invariant form to \(P_4(V)\) is nondegenerate, and every \(0\ne J\in P_4(V)\) satisfies \(U(\Vir)J\cong X_2=L(1,4).\)
\end{lem}

\begin{proof}
We first construct the weight-four primary line. Suppose $k=3$, so $L=\Lambda_3$. Choose $0\ne F\in P_3(V)$ as in Lemma~\ref{lem:k3-weight3} and put $D=U(\Vir)F\cong L(1,3)$. The modules $L(1,0)$ and $D$ are orthogonal: by contravariance, pairing a Virasoro descendant of $F$ with a vacuum descendant reduces to pairing $F$ with a weight-three vector of $L(1,0)$, which vanishes because $F$ is primary. Both restrictions of the invariant form are nondegenerate, the latter because $D$ is irreducible and $(F,F)\ne0$. Hence $B:=L(1,0)\oplus D$ is an orthogonal nondegenerate sum. By Proposition~\ref{prop:rank-one-lattice-package}(i)--(ii), through weight four the only further contribution to the character is the one-dimensional top of $X_2=L(1,4)$ in weight four. Thus $V_n=B_n$ for $n<4$ and $\dim V_4=\dim B_4+1$. Therefore $Q:=B_4^\perp\cap V_4$ is one-dimensional and nondegenerate. Choose $0\ne J\in Q$. For $1\le r\le4$ and $a\in B_{4-r}=V_{4-r}$, invariance gives \((L(r)J,a)=(J,L(-r)a)=0\), so $L(r)J=0$; for $r>4$ this follows from CFT type. Hence $J\in P_4(V)$. Conversely every $u\in P_4(V)$ is orthogonal to $L(1,0)_4$, and \((u,L(-1)F)=(L(1)u,F)=0\); since $D_4=\C L(-1)F$, we have $u\in Q$. Thus $P_4(V)=Q=\C J$ and $(J,J)\ne0$.

Now suppose $k>4$. Every nonzero lattice sector starts above weight four, while \(M(1)^+=X_0\oplus X_2\oplus X_4\oplus\cdots\), with $X_0=L(1,0)$, $X_2=L(1,4)$, and $X_4$ beginning in weight $16$. Hence $V_n=L(1,0)_n$ for $0\le n<4$ and $\dim V_4=\dim L(1,0)_4+1$. Lemma~\ref{lem:first-primary} gives $\dim P_4(V)=1$ and nondegeneracy of the form on this line. Fix $0\ne J\in P_4(V)$, so again $(J,J)\ne0$.

Set $C=U(\Vir)J$. By Proposition~\ref{prop:c1-verma} with $m=2$, the first nontrivial singular vector of $V(1,4)$ occurs at level five, hence at absolute weight nine. Let \(\phi:V(1,4)\twoheadrightarrow C\) be the canonical map. The minimal-kernel argument of Lemma~\ref{lem:first-singular} shows that $\phi$ is injective below weight nine.

If $k\ge9$, Lemma~\ref{lem:first-singular} applies with $m=2$, $h=4$, and $H=9$: for $k>9$ there is no lattice contribution in weight nine, while for $k=9$ the lattice top contributes exactly one dimension. Hence $C\cong X_2$.

It remains to consider $k\in\{3,5,6,7,8\}$. For $k=3$, retain the module $D\cong L(1,3)$ constructed above. Since $J\perp B_4$, contravariance shows that $C\perp B=L(1,0)\oplus D$, so $L(1,0)$, $C$, and $D$ are pairwise orthogonal.

Suppose instead that $5\le k\le8$. First $C\perp L(1,0)$ by contravariance, since $J\perp L(1,0)_4$; thus $A:=L(1,0)\oplus C$ is an orthogonal direct sum. By Corollary~\ref{cor:c1-verma-low-weight}, $C$ and $X_2$ have the same graded dimensions below weight nine. Since the first lattice sector starts in weight $k$ and $X_4$ starts in weight $16$, character comparison gives $V_n=A_n$ for $n<k$ and $\dim V_k=\dim A_k+1$. The form on $A_k$ is nondegenerate: on $C_k$ this follows by pulling back the Shapovalov form, since $k<9$ and $(J,J)\ne0$. Choose a nonisotropic vector $0\ne F\in A_k^\perp\cap V_k$. For $1\le r\le k$ and $a\in A_{k-r}$, \((L(r)F,a)=(F,L(-r)a)=0\); since $V_{k-r}=A_{k-r}$ is nondegenerate, $L(r)F=0$, and for $r>k$ the target is zero. Thus $F$ is primary. Because none of $5,6,7,8$ is of the form $n^2/4$, Proposition~\ref{prop:c1-verma} gives \(D:=U(\Vir)F\cong L(1,k)\). Again contravariance gives $D\perp A$. Thus in every remaining case $k\in\{3,5,6,7,8\}$ we have pairwise orthogonal Virasoro submodules \(L(1,0), C=U(\Vir)J, D\cong L(1,k).\)

No second lattice sector occurs by weight nine because $4k>9$, and no further $M(1)^+$-summand occurs because $X_4$ starts in weight $16$. Moreover the charge-one Fock sector and $L(1,k)$ have the same graded dimension $q^kP(q)$ for these values of $k$. Hence \(\dim V_9=\dim L(1,0)_9+\dim(X_2)_9+\dim L(1,k)_9.\) Let $\widetilde s\in V(1,4)_9$ be a nonzero first singular vector. If $s=\phi(\widetilde s)\ne0$, the same minimal-kernel argument makes $\phi$ injective through weight nine, and therefore \(\dim C_9=\dim V(1,4)_9=\dim(X_2)_9+1\). Since $L(1,0)_9\oplus C_9\oplus D_9\subseteq V_9$, this contradicts the preceding character equality. Thus $s=0$, so $\phi$ factors through the simple quotient $X_2=L(1,4)$. The induced nonzero surjection $X_2\to C$ is an isomorphism. Hence \(U(\Vir)J=C\cong X_2\) in all cases.
\end{proof}

\begin{cor}\label{cor:M1plus-kgt4}
Under the hypotheses of Lemma~\ref{lem:intervening}, for every \(0\ne J\in P_4(V)\), the VOA \(\langle L(1,0),J\rangle_{\mathrm{VOA}}\) is isomorphic to \(M(1)^+\).
\end{cor}

\begin{proof}
Lemma~\ref{lem:intervening} gives \((J,J)\ne0\), \(U(\Vir)J\cong X_2\), and \(\dim P_4(V)=1\). Theorem~\ref{thm:M1plus-core} now applies.
\end{proof}

\section{The \texorpdfstring{$O_2$}{O2} charge category and the rank-one orbifold branch}\label{sec:o2}

\subsection{Compact-lattice normalization of the charge category}

Fix \(k\in\Z_{>0}\) and choose \(\lambda\in\mathfrak h\) such that \(\langle\lambda,\lambda\rangle=2k\). Then \(L_k=\Z\lambda\) is a positive-definite even lattice of rank one, and \(T_k:=\Hom(L_k,U(1))\cong U(1)\). Write \(\mathbb G_m=\C^\times\) for the multiplicative algebraic group and, for \(r\in\Z\), let \(\C_r\) denote its one-dimensional rational module on which \(z\in\mathbb G_m\) acts by \(z^r\). We use throughout the concrete $T_k$- and $O(2)$-actions recorded in Proposition~\ref{prop:rank-one-lattice-package}(v).

\begin{thm}\label{thm:Fock-Gm}
Let \(\mathcal P_k=\left\langle M(1),\ M(1,r\lambda)\mid r\in\Z\setminus\{0\}\right\rangle_{\oplus}\) be the full subcategory of finite direct sums of the indicated Heisenberg modules. Then \(V_{L_k}^{T_k}=M(1)\), and there is a symmetric vertex tensor equivalence
\[
\Psi_k^{\mathrm{Fock}}:\Rep\mathbb G_m\overset{\sim}{\longrightarrow}\mathcal P_k,
\qquad
\C_r\longmapsto
\begin{cases}
M(1),&r=0,\\
M(1,r\lambda),&r\ne0.
\end{cases}
\]
\end{thm}

\begin{proof}
By Proposition~\ref{prop:rank-one-lattice-package}(v), \(V_{L_k}=M(1)\oplus\bigoplus_{r\in\Z\setminus\{0\}}M(1,r\lambda), V_{L_k}^{T_k}=M(1)\). The lattice VOA is simple and the compact torus action is continuous on every finite-dimensional homogeneous subspace. Moreover, the Fock modules above belong to the vertex tensor category generated by real-charge Heisenberg Fock modules; see \cite[Theorem~2.3]{CKLR}. Therefore Theorem~\ref{thm:McRae-VOA-orbifold}(ii) gives a symmetric vertex tensor equivalence \(\Phi_{T_k}:\Rep_{\mathrm{fd}}T_k\overset{\sim}{\longrightarrow}\mathcal P_k\), with the $P(z)$-tensor products independent of the chosen admissible ambient category.

We now fix the charge normalization. For \(r\in\Z\), let \(\chi_r\in\Rep_{\mathrm{fd}}T_k\) be the character \(\chi_r(\chi)=\chi(r\lambda)\). With the convention \(\Phi_{T_k}(M)=(M\otimes V_{L_k})^{T_k}\) fixed in Subsection~\ref{subsec:McRae-VOA}, the charge decomposition above gives \(\Phi_{T_k}(\chi_r)\cong M(1,-r\lambda)\). Indeed, the $T_k$-invariants in \(\chi_r\otimes M(1,s\lambda)\) occur precisely when \(r+s=0\). Let \(\iota:T_k\to T_k\) be inversion, \(\iota(\chi)=\chi^{-1}\). The pullback \(\iota^*\) is a symmetric vertex tensor autoequivalence of \(\Rep_{\mathrm{fd}}T_k\), and \(\iota^*(\chi_r)=\chi_{-r}\). Consequently \((\Phi_{T_k}\circ\iota^*)(\chi_r)\cong M(1,r\lambda)\). Finally, restriction from \(\mathbb G_m=\C^\times\) to its maximal compact subgroup $U(1)$, followed by the identification $T_k\cong U(1)$, gives the symmetric tensor equivalence of Remark~\ref{rmk:compact-complexification}; under it \(\C_r\) restricts to \(\chi_r\). Transporting the preceding normalized compact equivalence therefore yields \(\Psi_k^{\mathrm{Fock}}:\Rep\mathbb G_m\overset\sim\longrightarrow\mathcal P_k\) with \(\C_r\mapsto M(1,r\lambda)\), as claimed.
\end{proof}

Following Proposition~\ref{prop:rank-one-lattice-package}(v), set \(K_k:=\langle T_k,\theta\rangle=T_k\rtimes\langle\theta\rangle\cong O(2)\).
\begin{thm}\label{thm:O2-category}
In the $C_1$-cofinite $M(1)^+$ tensor category, let $\Dk$ be the full subcategory of finite direct sums of \(M(1)^+,\ M(1)^-,\ M(1,r\lambda)\ (r\ge1).\) Then there is a symmetric vertex tensor equivalence \(\Psi_k:\Rep\Otwo\overset\sim\longrightarrow\Dk\). Under the equivalence, \(\C\longmapsto M(1)^+,\det\longmapsto M(1)^-,\rho_r\longmapsto M(1,r\lambda)\), where $\rho_r$ is the two-dimensional irreducible $\Otwo$-representation with torus weights $\{r,-r\}$.
\end{thm}

\begin{proof}
By Proposition~\ref{prop:ALY-package}(i)--(ii), the category $\mathcal C_1(M(1)^+)$ is a vertex and braided tensor category and contains all the multiplicity modules occurring below. The compact group \(K_k\cong O(2)\) acts on the simple vertex operator algebra $V_{L_k}$. Proposition~\ref{prop:rank-one-lattice-package}(v), specifically \eqref{eq:rank-one-O2-Schur-Weyl}, gives \(V_{L_k}^{K_k}=M(1)^+\) and identifies its compact-orbifold multiplicity category with precisely $\Dk$. Applying Theorem~\ref{thm:McRae-VOA-orbifold}(ii), with ambient category $\mathcal C_1(M(1)^+)$, gives a symmetric vertex tensor equivalence \(\Rep_{\mathrm{fd}}K_k\overset{\sim}{\longrightarrow}\Dk\). The ambient-independence assertion in the same theorem applies to all its $P(z)$-tensor products. Finally, complexification $K_k\cong O(2)\subset \Otwo$ and Remark~\ref{rmk:compact-complexification} identify the three irreducible types with $\C$, $\det$, and $\rho_r$, respectively. This gives the stated equivalence \(\Psi_k:\Rep\Otwo\overset{\sim}{\longrightarrow}\Dk\).
\end{proof}

The fusion formulas in this subcategory are therefore \(M(1)^-\boxtimes M(1)^-\cong M(1)^+, M(1)^-\boxtimes M(1,r\lambda)\cong M(1,r\lambda)\), and
\[
M(1,r\lambda)\boxtimes M(1,s\lambda)\cong
\begin{cases}
M(1,(r+s)\lambda)\oplus M(1,|r-s|\lambda),&r\ne s,\\
M(1,2r\lambda)\oplus M(1)^+\oplus M(1)^-,&r=s.
\end{cases}
\]
These are also exactly the specializations of Proposition~\ref{prop:ALY-package}(iv).

\begin{lem}\label{lem:Dk-ambient-Serre}
The full subcategory $\Dk$ is closed under ambient $M(1)^+$-submodules and quotients, and its inclusion in the grading-restricted $M(1)^+$-module category is exact.
\end{lem}

\begin{proof}
Every object of $\Dk$ is, as an actual $M(1)^+$-module, a finite direct sum of the simple modules $M(1)^+$, $M(1)^-$, and $M(1,r\lambda)$. It is therefore semisimple in the ordinary module-theoretic sense. Every ambient submodule is a direct summand and hence a finite direct sum of the same simple modules; the same holds for quotients. Kernels and cokernels in the full subcategory consequently agree with the ambient ones.
\end{proof}

The following proposition is the $M(1)^+$ charge-category analogue of Theorem~\ref{thm:CMY-square}; its final assertion is the corresponding analogue of Theorem~\ref{thm:square-core}. The only additional step is that coefficient images are first placed in the ambient category $\mathcal C_1(M(1)^+)$ and then returned to $\Dk$ by the $O_2$ tensor structure and Lemma~\ref{lem:Dk-ambient-Serre}.

\begin{prop}\label{prop:CMY-O2}
The category $\Dk$ satisfies the hypotheses of Theorem~\ref{thm:CMY-direct-limit}. Consequently, $\IndCMY(\Dk)$ carries the $P(z)$-vertex tensor category and braided tensor category structures constructed by Creutzig--McRae--Yang, extending those on $\Dk$. Moreover, by Theorem~\ref{thm:CMY-extension}, the following two categories are isomorphic:
\begin{enumerate}[label=\textnormal{(\arabic*)}]
\item Vertex operator algebras $(A,Y_A,\one_A,\omega_A)$ such that
\begin{itemize}
\item $A$ is an $M(1)^+$-module in $\IndCMY(\Dk)$;
\item writing $Y_A^{M(1)^+}$ for the given $M(1)^+$-module vertex
operator on $A$, \(Y_A^{M(1)^+}(v,x)=Y_A(v_{-1}\one_A,x)(v\in M(1)^+)\);
\item $\omega_A=L(-2)\one_A$.
\end{itemize}

\item Commutative associative algebra objects
$(A,\mu_A,\iota_A)\in\CAlg(\IndCMY(\Dk))$ such that \(A=\bigoplus_{n\in\Z}A_n,A_n=\{a\in A\mid L(0)a=na\}\), with $A_n=0$ for all sufficiently negative $n$ and $\dim_\C A_n<\infty$ for every $n\in\Z$.
\end{enumerate}
In particular, this correspondence applies to the extension-admissible objects of Definition~\ref{def:extension-admissible}. Finally, if a grading-restricted VOA extension $A\supset M(1)^+$ contains a finite direct sum $S$ of simple objects of $\Dk$, then \(\langle M(1)^+,S\rangle_{\mathrm{VOA}}\in\IndCMY(\Dk)\).
\end{prop}

\begin{proof}
We verify the five hypotheses of Theorem~\ref{thm:CMY-direct-limit}, in parallel with the proof of Theorem~\ref{thm:CMY-square}. For condition~\textnormal{(1)}, the tensor unit $M(1)^+$ is an object of $\Dk$ by definition. For condition~\textnormal{(2)}, Lemma~\ref{lem:Dk-ambient-Serre} shows that $\Dk$ is closed under $M(1)^+$-submodules and quotients in the ambient grading-restricted module category. Closure under finite direct sums is part of the definition of $\Dk$. For condition~\textnormal{(3)}, every object of $\Dk$ is a finite direct sum of the simple modules \(M(1)^+, M(1)^-, M(1,r\lambda)(r\ge1).\) Each simple module is generated by any nonzero vector, since the submodule generated by such a vector is nonzero and hence equals the whole simple module. Thus every object of $\Dk$ is finitely generated. For condition~\textnormal{(4)}, Proposition~\ref{prop:ALY-package}(ii) provides the ambient vertex tensor category $\mathcal C_1(M(1)^+)$. Theorem~\ref{thm:O2-category}, together with the ambient-independence statement in Theorem~\ref{thm:McRae-VOA-orbifold}(ii), shows that $\Dk$ is a full vertex tensor subcategory of $\mathcal C_1(M(1)^+)$ and that its $P(z)$-tensor products and structural isomorphisms are the restrictions of the ambient ones. Hence condition~\textnormal{(4)} holds. It remains to verify condition~\textnormal{(5)}. As in the square case, we prove the stronger statement with an arbitrary generalized target. Let $M,N\in\Dk$, let $X$ be an arbitrary generalized $M(1)^+$-module, and let $\mathcal Y$ be a logarithmic intertwining operator of type \(\binom{X}{M N}\). Set $I:=\operatorname{Im}\mathcal Y$. Since $M$ and $N$ are $C_1$-cofinite, Proposition~\ref{prop:CMY-coefficient-image} shows that $I$ is a $C_1$-cofinite grading-restricted generalized $M(1)^+$-module. Thus \(I\in\mathcal C_1(M(1)^+)\) by Proposition~\ref{prop:ALY-package}(ii). The operator $\mathcal Y$ corestricts to \(\mathcal Y^I:M\otimes N\longrightarrow I\{x\}[\log x]\), and its coefficients span $I$ by the definition of the coefficient image. Fix $z\in\C^\times$ and a branch of $\log z$, and let \(I_{\mathcal Y^I}:M\otimes N\longrightarrow\overline I\) be the corresponding $P(z)$-intertwining map. Since $I\in\mathcal C_1(M(1)^+)$, the universal property of the ambient $P(z)$-tensor product gives a unique $M(1)^+$-module morphism \(f:M\boxtimes^{\mathcal C_1(M(1)^+)}_{P(z)}N\longrightarrow I\) such that \(I_{\mathcal Y^I}=\overline f\circ\boxtimes_{P(z)}\). Applying the inverse correspondence \eqref{eq:Pz-inverse} gives $\mathcal Y^I=f\circ\mathcal Y_{\boxtimes}$. Hence every coefficient of $\mathcal Y^I(m,x)n$ lies in $\operatorname{Im}f$. Since these coefficients span $I$, the morphism $f$ is surjective. By Theorem~\ref{thm:O2-category} and the ambient-independence statement in Theorem~\ref{thm:McRae-VOA-orbifold}(ii), \(M\boxtimes^{\mathcal C_1(M(1)^+)}_{P(z)}N\) is an object of $\Dk$; equivalently, it is a finite direct sum of the simple modules displayed above, with multiplicities determined by the $O_2$ tensor-product rules. Since $I$ is an ambient quotient of this object, Lemma~\ref{lem:Dk-ambient-Serre} gives $I\in\Dk$. In particular, when $X\in\IndCMY(\Dk)$, this is exactly condition~\textnormal{(5)} of Theorem~\ref{thm:CMY-direct-limit}. All five hypotheses are therefore satisfied. Theorem~\ref{thm:CMY-direct-limit} gives the asserted $P(z)$-vertex tensor and braided tensor category structures on $\IndCMY(\Dk)$, extending those on $\Dk$. Theorem~\ref{thm:CMY-extension} then identifies the two categories displayed in the statement. The assertion concerning extension-admissible objects follows immediately from Definition~\ref{def:extension-admissible}.

It remains to prove the final assertion. Put \(U:=\langle M(1)^+,S\rangle_{\mathrm{VOA}}\text{ and }E_0:=M(1)^++S\subseteq A\). The natural addition map \(M(1)^+\oplus S\twoheadrightarrow E_0\) is an ambient $M(1)^+$-module epimorphism. Its source belongs to $\Dk$, so Lemma~\ref{lem:Dk-ambient-Serre} gives $E_0\in\Dk$. Suppose that $E_n\in\Dk$ has been constructed and choose an internal decomposition \(E_n=\bigoplus_{\alpha\in\Lambda_n}T_\alpha\) into its actual embedded simple summands; the set $\Lambda_n$ is finite. For every ordered pair $\alpha,\beta\in\Lambda_n$, the restriction \(Y_A|_{T_\alpha\otimes T_\beta}\) is an $M(1)^+$-intertwining operator of type \(\binom{A}{T_\alpha T_\beta}\), because $A$ is a VOA extension of $M(1)^+$. Let \(I_{\alpha,\beta}:=\operatorname{Im}\bigl(Y_A|_{T_\alpha\otimes T_\beta}\bigr)\) be its coefficient image. The stronger coefficient-image statement proved above applies with the generalized target $A$ and gives $I_{\alpha,\beta}\in\Dk$. Define \(E_{n+1}:=E_n+\sum_{\alpha,\beta\in\Lambda_n}I_{\alpha,\beta}\). The natural addition map \(E_n\oplus\bigoplus_{\alpha,\beta\in\Lambda_n}I_{\alpha,\beta}\twoheadrightarrow E_{n+1}\) is an ambient $M(1)^+$-module epimorphism. Its source belongs to $\Dk$, since $\Lambda_n$ is finite, so Lemma~\ref{lem:Dk-ambient-Serre} gives $E_{n+1}\in\Dk$. Thus \(E_0\subseteq E_1\subseteq E_2\subseteq\cdots\) is an increasing sequence of $\Dk$-submodules of $A$. Set $E_\infty:=\bigcup_{n\ge0}E_n$. If $a,b\in E_\infty$, choose $n$ with $a,b\in E_n$ and decompose them as finite sums of vectors from the embedded simple summands $T_\alpha$. By bilinearity and the definition of the coefficient images $I_{\alpha,\beta}$, every coefficient of $Y_A(a,x)b$ lies in $E_{n+1}\subseteq E_\infty$. Hence $E_\infty$ is closed under all mode products. Since it contains $M(1)^+$, it contains the vacuum and the conformal vector, and is therefore a vertex operator subalgebra of $A$ containing $S$. Thus $U\subseteq E_\infty$. Conversely, $E_0\subseteq U$. If $E_n\subseteq U$, then every $T_\alpha\subseteq E_n$ lies in $U$; since $U$ is a vertex operator subalgebra, all coefficients of $Y_A(T_\alpha,x)T_\beta$ belong to $U$. Hence $I_{\alpha,\beta}\subseteq U$ and therefore $E_{n+1}\subseteq U$. Induction gives $E_n\subseteq U$ for all $n$, so \(U=E_\infty=\bigcup_{n\ge0}E_n\). This is a directed union of $\Dk$-subobjects. By the definition of the Creutzig--McRae--Yang direct-limit completion, $U\in\IndCMY(\Dk)$, as required.
\end{proof}

Lemma~\ref{lem:ind-equivalence} now extends $\Psi_k$ to an exact symmetric monoidal equivalence
\[
\widehat\Psi_k:\Rep^{\mathrm{rat}}\Otwo\overset\sim\longrightarrow\IndCMY(\Dk).
\]
\subsection{The first lattice module}

\begin{lem}\label{lem:M1plus-lowest-line}
Let $M=\bigoplus_{\mu\in\C}M_{[\mu]}$ be a grading-restricted generalized $M(1)^+$-module. Suppose that, for some $h\in\C$ and some $0\ne v\in M$, \(M_{[h]}=\C v\) and \(M_{[h-r]}=0\) for \(r\in\Z_{>0}\). Then the cyclic submodule $M(1)^+\cdot v$ is a highest-weight $M(1)^+$-module in the sense fixed in Subsection~\ref{subsec:ALY-package}.
\end{lem}

\begin{proof}
Since $M_{[h]}=\C v$ is one-dimensional and is a generalized $L(0)$-eigenspace of eigenvalue $h$, necessarily $L(0)v=hv$. Let $a\in M(1)^+$ be homogeneous. The standard mode-weight relation \([L(0),a_n]=(\wt(a)-n-1)a_n\) gives \(a_nM_{[h]}\subseteq M_{[h+\wt(a)-n-1]}\). Hence, if $n\ge\wt(a)$, then \(a_nv\in M_{[h-(n+1-\wt(a))]}=0\), because $n+1-\wt(a)\in\Z_{>0}$. Thus $v\in\Omega(M)$.

For $n=\wt(a)-1$, the zero mode $o(a)=a_{\wt(a)-1}$ preserves $M_{[h]}=\C v$. Since $\Omega(M)$ is an $A(M(1)^+)$-module under zero modes, $\C v$ is therefore a one-dimensional $A(M(1)^+)$-submodule of $\Omega(M)$, and hence is irreducible.

Set $N:=M(1)^+\cdot v$. As an $M(1)^+$-submodule of the grading-restricted generalized module $M$, the module $N$ is itself grading restricted. Moreover, $v\in\Omega(N)$ and $\C v$ is a one-dimensional irreducible $A(M(1)^+)$-module. Therefore, by the convention of Subsection~\ref{subsec:ALY-package}, $N$ is a highest-weight $M(1)^+$-module generated by $v$.
\end{proof}

\begin{lem}\label{lem:first-lattice}
Let $V$ be a simple self-contragredient VOA of CFT type and central charge one. Assume $k=3$ or $k>4$ and \(\ch V=\ch V_L^+, L=\Z\alpha,\langle\alpha,\alpha\rangle=2k\). Let \(U\cong M(1)^+\) be the subVOA obtained in Corollary~\ref{cor:M1plus-kgt4}, and put $K=U^\perp$ with respect to the invariant bilinear form. Then:
\begin{enumerate}[label=(\roman*),leftmargin=2.1em]
\item $K$ is a grading-restricted $U$-module, and the restriction of the invariant bilinear form to each homogeneous subspace $K_j$ is nondegenerate;
\item $K_j=0$ for $j<k$ and $\dim K_k=1$;
\item every $0\ne F\in K_k$ is nonisotropic and is a $U$-highest-weight vector;
\item the $U$-module $U\cdot F$ is irreducible, and the restriction of the
invariant bilinear form to $U\cdot F$ is nondegenerate; and \(U\cdot F\cong M(1,\lambda)\) for some $\lambda\in\mathfrak h$ with \(\langle\lambda,\lambda\rangle=2k\).
\end{enumerate}
\end{lem}

\begin{proof}
The restriction of the invariant form to the simple subVOA $U$ is nondegenerate: its radical is an ideal, and the vacuum has nonzero norm. Since distinct conformal-weight spaces are orthogonal and the form on each homogeneous space $V_j$ is nondegenerate, for every $j\ge0$ we have an orthogonal decomposition \(V_j=U_j\perp K_j\), where $K_j=K\cap V_j$. In particular, the restriction of the form to each $K_j$ is nondegenerate. Moreover, $K$ is graded, lower bounded, and has finite-dimensional homogeneous spaces.

To verify that $K$ is a $U$-module, let $a,u'\in U$ and $v\in K$. Invariance of the bilinear form gives \(\bigl(Y(a,z)v,u'\bigr)=\bigl(v,Y(e^{zL(1)}(-z^{-2})^{L(0)}a,z^{-1})u'\bigr).\) Because $U$ is a conformal subVOA, every coefficient on the right belongs to $U$, and the right-hand side vanishes. Comparing coefficients shows that $a_nv\in K$ for every $n\in\Z$. Hence $K$ is a grading-restricted $U$-module.

By Proposition~\ref{prop:rank-one-lattice-package}(i), the character of $V_L^+$ is \(\ch M(1)^++\sum_{r\ge1}\ch M(1,r\alpha)\), and the first summand outside $M(1)^+$ begins at weight $k$ with one-dimensional top. Since $U\cong M(1)^+$ and \(V_j=U_j\perp K_j\), one has \(\dim K_j=\dim V_j-\dim U_j\). Therefore the character equality \(\ch V=\ch V_L^+\) gives $K_j=0$ for $j<k$ and $\dim K_k=1$. The form on $K_k$ is nondegenerate, so every nonzero $F\in K_k$ is nonisotropic. If $a\in U$ is homogeneous and $n\ge\wt(a)$, then \(a_nF\in K_{k+\wt(a)-n-1}=0\). The zero mode $o(a)=a_{\wt(a)-1}$ preserves the one-dimensional space $K_k$ and hence acts by a scalar. Via $U\cong M(1)^+$, Lemma~\ref{lem:M1plus-lowest-line} shows that $U\cdot F$ is a grading-restricted highest-weight module in the sense fixed in Subsection~\ref{subsec:ALY-package}, of conformal weight $k$.

By Proposition~\ref{prop:ALY-package}(iii), $U\cdot F$ is irreducible because $k\ne0$. The classification recorded before Proposition~\ref{prop:ALY-package} shows that the non-Fock irreducibles have lowest conformal weights $0$, $1$, $1/16$, or $9/16$. Since the present lowest weight is the integer $k\ge3$, the generated module must be of Fock type: \(U\cdot F\cong M(1,\lambda)\) for some $\lambda\in\mathfrak h\setminus\{0\}$ with \(\langle\lambda,\lambda\rangle/2=k\), equivalently \(\langle\lambda,\lambda\rangle=2k\). The sign of $\lambda$ is immaterial because $M(1,\lambda)\cong M(1,-\lambda)$.

Finally, the radical of the restriction of the invariant form to $U\cdot F$ is a $U$-submodule. Since $U\cdot F$ is irreducible and $(F,F)\ne0$, this radical is zero. Thus the restriction of the form to $U\cdot F$ is nondegenerate.
\end{proof}

\subsection{The nonzero quadratic orbit}

Let $\rho_1$ be the standard two-dimensional representation of $\Otwo$. Fix a nondegenerate $\Otwo$-invariant symmetric bilinear form $ \beta_Q:\rho_1\otimes\rho_1\longrightarrow\C, $ and regard the same form as the invariant quadratic polynomial \(Q_\rho\in\Sym^2(\rho_1^*)^{\Otwo}\), normalized by \(Q_\rho(v)=\beta_Q(v,v)\). Let \(\iota_\rho:\rho_1\overset{\sim}{\longrightarrow}\rho_1^*, v\longmapsto\beta_Q(v,-)\), be the induced $\Otwo$-equivariant self-duality, and define the associated inverse-form tensor \(\check Q_\rho:= \bigl(\iota_\rho^{-1}\bigr)^{\otimes2}(Q_\rho) \in\Sym^2(\rho_1)^{\Otwo}\). Contracting $\check Q_\rho$ with $\beta_Q$ gives $ \langle\beta_Q,\check Q_\rho\rangle =\operatorname{tr}(\id_{\rho_1}) =\dim\rho_1=2. $ We retain throughout the distinction between the quadratic coordinate tensor $Q_\rho$ and the inverse-form tensor $\check Q_\rho$.

\begin{lem}\label{lem:O2-quadric}
For $c\in\C^\times$, set \(R_c=\frac{\Sym(\rho_1^*)}{(Q_\rho-c)}.\) Choose $x_0\in\rho_1$ with $Q_\rho(x_0)=c$ and put \(H_{\mathrm{refl}}:=(\Otwo)_{x_0}.\) Then $H_{\mathrm{refl}}\cong\Z_2$ is an order-two reflection subgroup and \(\Spec R_c\cong\Otwo/H_{\mathrm{refl}}.\) In particular, $R_c$ is reduced and has no nonzero proper $\Otwo$-stable ideals. As an $\Otwo$-module, \(R_c\cong\C\oplus\bigoplus_{r\ge1}\rho_r\).
\end{lem}

\begin{proof}
Put $X_c:=\Spec R_c$. Its closed points are exactly the vectors $x\in\rho_1$ satisfying $Q_\rho(x)=c$. Since $c\ne0$, every such $x$ is nonzero, and our normalization $Q_\rho(v)=\beta_Q(v,v)$ gives \((d(Q_\rho-c))_x=2\beta_Q(x,-)\ne0\), because $\beta_Q$ is nondegenerate. The Jacobian criterion recalled in Subsection~\ref{subsec:AG-conventions} therefore shows that $X_c$ is smooth, hence reduced.

The quadric is nonempty. Indeed, since $\beta_Q$ is nondegenerate, there is $v\in\rho_1$ with $Q_\rho(v)\ne0$; after rescaling $v$ by a square root in $\C$, we obtain a vector of $Q_\rho$-value $c$. Moreover, $\Otwo$ acts transitively on $X_c$. To see this directly, let $x,y\in X_c$. Since $Q_\rho(x)=Q_\rho(y)=c\ne0$, the orthogonal decompositions \(\rho_1=\C x\perp x^\perp=\C y\perp y^\perp\) are decompositions into nondegenerate one-dimensional spaces. The map $x\mapsto y$ is an isometry on the first summands, and the one-dimensional nondegenerate spaces $x^\perp$ and $y^\perp$ are isometric over $\C$. Taking the orthogonal direct sum of these two isometries gives an element of $\Otwo$ carrying $x$ to $y$.

Now fix $x_0\in X_c$. An element of $\Otwo$ fixing $x_0$ preserves the orthogonal line $x_0^\perp$ and acts on it by $\pm1$. Thus \((\Otwo)_{x_0}=\{1,s_{x_0}\}\cong\Z_2,\) where $s_{x_0}$ is the reflection acting trivially on $\C x_0$ and by $-1$ on $x_0^\perp$. Since $X_c$ is now known to be reduced, the orbit--stabilizer identification from Subsection~\ref{subsec:AG-conventions} applies to the transitive action and gives \(X_c\cong\Otwo/H_{\mathrm{refl}}.\) Lemma~\ref{lem:G-simple-homogeneous}, applied to this nonempty reduced affine homogeneous $\Otwo$-variety, gives the assertion about $\Otwo$-stable ideals.

It remains to determine the left $\Otwo$-module structure of $R_c$. Choose an isotropic basis $e_+,e_-$ of $\rho_1$ with \(\beta_Q(e_+,e_-)=1\). With \(T:=\{t_a:a\in\C^\times,\ t_a e_+=ae_+,\ t_a e_-=a^{-1}e_-\} \cong\mathbb G_m\) and the reflection $s(e_+)=e_-$, $s(e_-)=e_+$, one has \(\Otwo=T\rtimes\langle s\rangle.\) Indeed, in the basis $e_+,e_-$ the Gram matrix of $\beta_Q$ is \(J=\left(\begin{smallmatrix}0&1\\1&0\end{smallmatrix}\right)\). If \(g=\left(\begin{smallmatrix}a&b\\c&d\end{smallmatrix}\right)\in\Otwo\), then $g^{\mathsf T}Jg=J$ gives \(ac=bd=0\) and \(ad+bc=1\). Hence either \(g=\operatorname{diag}(a,a^{-1})=t_a\) or \(g=\left(\begin{smallmatrix}0&b\\b^{-1}&0\end{smallmatrix}\right)=t_b s\). Moreover $s^2=1$ and $st_as^{-1}=t_{a^{-1}}$, which proves the stated semidirect-product decomposition. Choosing $a_0\in\C^\times$ with $2a_0^2=c$ and replacing the base point $x_0$ within its $\Otwo$-orbit if necessary, we may take \(x_0=a_0(e_++e_-)\), so that \(H_{\mathrm{refl}}=\langle s\rangle.\) Hence the map $T\to\Otwo/H_{\mathrm{refl}}$, $t\mapsto tH_{\mathrm{refl}}$, is an isomorphism of varieties. If $z$ denotes the standard coordinate on $T\cong\mathbb G_m$, then \(R_c\cong\C[z,z^{-1}] =\C\cdot1\oplus\bigoplus_{r\ge1}\operatorname{span}\{z^r,z^{-r}\}\). Under the left $\Otwo$-action, the summand $\operatorname{span}\{z^r,z^{-r}\}$ has torus weights $\{r,-r\}$ and the reflection interchanges its two weight lines. By the convention of Theorem~\ref{thm:O2-category}, this summand is $\rho_r$. Therefore \(R_c\cong\C\oplus\bigoplus_{r\ge1}\rho_r\), as claimed.
\end{proof}

\begin{lem}\label{lem:O2-admissible}
For every $c\in\C^\times$, the commutative algebra object $\widehat\Psi_k(R_c)$ in $\IndCMY(\Dk)$ is extension-admissible. As an $M(1)^+$-module, it is \(M(1)^+\oplus\bigoplus_{r\ge1}M(1,r\lambda)\).
\end{lem}

\begin{proof}
Since $Q_\rho\in\Sym^2(\rho_1^*)^{\Otwo}$, the ideal $(Q_\rho-c)\subseteq\Sym(\rho_1^*)$ is $\Otwo$-stable. Hence $R_c$ is a commutative rational $\Otwo$-algebra, and the symmetric monoidal equivalence $\widehat\Psi_k$ transports it to a commutative algebra object of $\IndCMY(\Dk)$. By Lemma~\ref{lem:O2-quadric}, \(R_c\cong\C\oplus\bigoplus_{r\ge1}\rho_r\). Using Theorem~\ref{thm:O2-category} together with the explicit Ind-extension formula of Lemma~\ref{lem:ind-equivalence}, we obtain \(\widehat\Psi_k(R_c)\cong M(1)^+\oplus\bigoplus_{r\ge1}M(1,r\lambda)\) as an $M(1)^+$-module.

Let $A:=\widehat\Psi_k(R_c)$ and let $\iota_A:M(1)^+\to A$ be its algebra unit. The preceding decomposition contains exactly one copy of the tensor unit $M(1)^+$, so \(\Hom_{M(1)^+}(M(1)^+,A)=\C.\) The unit $\C\to R_c$, $1\mapsto1$, is a nonzero monomorphism, and an equivalence preserves monomorphisms; therefore $\iota_A$ is injective. Since $\iota_A\ne0$, the one-dimensional Hom-space is precisely \(\C\iota_A\). Thus condition~\textnormal{(i)} of Definition~\ref{def:extension-admissible} holds.

The vacuum module $M(1)^+$ is $L(0)$-semisimple, has spectrum in $\Z_{\ge0}$, and has finite-dimensional homogeneous subspaces. For $r\ge1$, the Fock-module construction recalled in Subsection~\ref{subsec:ALY-package} gives lowest conformal weight \(\frac{\langle r\lambda,r\lambda\rangle}{2}=kr^2\in\Z_{>0}\) and spectrum $kr^2+\Z_{\ge0}$; its homogeneous subspaces are finite dimensional. Hence $L(0)$ acts semisimply on $A$ with integral eigenvalues, and its spectrum is bounded below by zero. Finally, for a fixed weight $N$, only finitely many $r\ge1$ satisfy $kr^2\le N$, so the weight-$N$ subspace of $A$ is a finite direct sum of finite-dimensional spaces. Conditions~\textnormal{(ii)}--\textnormal{(iii)} therefore hold, and $A$ is extension-admissible.
\end{proof}

\begin{lem}
\label{lem:standard-O2-all-k}
Let $k\in\Z_{>0}$ and let $L=\Z\alpha$ with $\langle\alpha,\alpha\rangle=2k$. Identify $L$ with the lattice $L_k=\Z\lambda$ fixed above by the isometry $\alpha\mapsto\lambda$, and use the induced lattice-VOA isomorphism to identify $V_L^+$ with $V_{L_k}^+$. Then \(V_L^+\in\IndCMY(\Dk),V_L^+=\langle M(1)^+,V_L^+[1]\rangle_{\mathrm{VOA}}, V_L^+[1]\cong M(1,\lambda)\). Regard the extension $V_L^+\supset M(1)^+$ as a commutative algebra object in $\IndCMY(\Dk)$ through Theorem~\ref{thm:CMY-extension}, and put \(B_k^{\rm std}:=\widehat\Psi_k^{-1}(V_L^+)\). Then $B_k^{\rm std}$ is a commutative rational $\Otwo$-algebra and, as a rational $\Otwo$-module, \(B_k^{\rm std}\cong\C\oplus\bigoplus_{r\ge1}\rho_r\). Moreover, $B_k^{\rm std}$ is generated as a commutative $\C$-algebra by its unique $\rho_1$-summand, and there exists $c_k\in\C^\times$ such that \(B_k^{\rm std}\cong R_{c_k}=\frac{\Sym(\rho_1^*)}{(Q_\rho-c_k)}\) as commutative rational $\Otwo$-algebras.
\end{lem}

\begin{proof}
By Proposition~\ref{prop:rank-one-lattice-package}(i), after the preceding identification $\alpha=\lambda$ one has \(V_L^+\cong M(1)^+\oplus\bigoplus_{r\ge1}M(1,r\lambda)\) as an $M(1)^+$-module. This is the directed union of its finite partial sums, which belong to $\Dk$, and hence $V_L^+\in\IndCMY(\Dk)$. Proposition~\ref{prop:rank-one-lattice-package}(iii) gives
\[
V_L^+=\langle M(1)^+,V_L^+[1]\rangle_{\mathrm{VOA}},\qquad V_L^+[1]\cong M(1,\lambda).
\]

By Theorem~\ref{thm:O2-category} and the explicit Ind-extension formula of Lemma~\ref{lem:ind-equivalence}, applying a quasi-inverse of $\widehat\Psi_k$ to the displayed $M(1)^+$-module decomposition gives directly \(B_k^{\rm std}\cong\C\oplus\bigoplus_{r\ge1}\rho_r\) as a rational $\Otwo$-module. In particular, \((B_k^{\rm std})^{\Otwo}=\C1\), and the $\rho_1$-summand occurs with multiplicity one. Lemma~\ref{lem:finite-generation}, applied to $\Psi_k$ with $E=V_L^+[1]$, shows that this $\rho_1$-summand generates $B_k^{\rm std}$ as an ordinary commutative $\C$-algebra.

Let \(j_k:\rho_1\hookrightarrow B_k^{\rm std}\) be this generating summand. The equivariant self-duality $\iota_\rho:\rho_1\overset\sim\longrightarrow\rho_1^*$ gives an $\Otwo$-equivariant map $j_k\circ\iota_\rho^{-1}:\rho_1^*\to B_k^{\rm std}$, which extends uniquely to a unital $\Otwo$-algebra homomorphism \(\pi_k:\Sym(\rho_1^*)\longrightarrow B_k^{\rm std}\). It is surjective because $j_k(\rho_1)$ generates $B_k^{\rm std}$. Since $Q_\rho$ is $\Otwo$-invariant and $(B_k^{\rm std})^{\Otwo}=\C1$, there is a scalar $c_k\in\C$ such that \(\pi_k(Q_\rho)=c_k1\). It remains to prove $c_k\ne0$. Let \(\mu_0^{\rm std}:\rho_1\otimes\rho_1\longrightarrow\C\) be the component of the multiplication of $B_k^{\rm std}$ obtained by projection onto its unique trivial summand. Under the symmetric vertex tensor equivalence $\Psi_k$, this morphism corresponds to the $M(1)^+$-summand component of the multiplication \(V_L^+[1]\boxtimes V_L^+[1]\longrightarrow V_L^+\). This component is nonzero. Indeed, with $F_1=e^\lambda+e^{-\lambda}\in V_L^+[1]$, the lattice vertex-operator formula recalled in Subsection~\ref{subsec:rank-one-lattice} gives
\[
Y(e^\lambda,z)e^{-\lambda}
=\varepsilon(\lambda,-\lambda)z^{-2k}(\one+O(z)),
\qquad
Y(e^{-\lambda},z)e^\lambda
=\varepsilon(-\lambda,\lambda)z^{-2k}(\one+O(z)).
\]
Because the lattice is even, $\varepsilon(\lambda,-\lambda)=\varepsilon(-\lambda,\lambda)$, and this common value is nonzero. The other two products have nonzero lattice charge and cannot contribute to the vacuum. Hence \((F_1)_{2k-1}F_1=2\varepsilon(\lambda,-\lambda)\one\ne0\), so $\mu_0^{\rm std}\ne0$.

By irreducibility and self-duality of $\rho_1$, Schur's lemma gives \(\Hom_{\Otwo}(\rho_1\otimes\rho_1,\C)=\C\beta_Q.\) Thus \(\mu_0^{\rm std}=b_k\beta_Q\) for some $b_k\in\C^\times$. By the definition of $\pi_k$ and the inverse-form tensor $\check Q_\rho$, the element $\pi_k(Q_\rho)$ is the product of the two $\rho_1$-legs of $\check Q_\rho$. It is $\Otwo$-invariant, so it lies in the trivial summand and therefore \(\pi_k(Q_\rho) =\mu_0^{\rm std}(\check Q_\rho) =b_k\beta_Q(\check Q_\rho) =2b_k\ne0.\) Consequently $c_k\in\C^\times$, and $\pi_k$ factors through a surjective $\Otwo$-algebra homomorphism \(R_{c_k}=\frac{\Sym(\rho_1^*)}{(Q_\rho-c_k)} \twoheadrightarrow B_k^{\rm std}.\) Its kernel is a proper $\Otwo$-stable ideal of $R_{c_k}$, hence is zero by Lemma~\ref{lem:O2-quadric}. Therefore $B_k^{\rm std}\cong R_{c_k}$ as commutative rational $\Otwo$-algebras.
\end{proof}

\begin{thm}\label{thm:VLplus-kgt4}
Let $L=\Z\alpha$ be a positive-definite even rank-one lattice with $\langle\alpha,\alpha\rangle=2k$, where $k\in\Z_{>4}$. Let $V$ be a simple self-contragredient CFT-type VOA of central charge one and assume only \(\ch V=\ch V_L^+.\) Then \(V\cong V_L^+.\)
\end{thm}

\begin{proof}
By Corollary~\ref{cor:M1plus-kgt4}, $V$ contains a subVOA \(U\cong M(1)^+.\) As usual, we identify $U$ with $M(1)^+$. By Lemma~\ref{lem:first-lattice}, there is a vector $0\ne F\in V_k$ with $(F,F)\ne0$ such that $U\cdot F$ is irreducible and \(U\cdot F\cong M(1,\lambda)\), where \(\langle\lambda,\lambda\rangle=2k\). Set \(W=\langle U,U\cdot F\rangle_{\mathrm{VOA}}.\) Since $U$ is a conformal subVOA of $V$ and $U\subseteq W$, the subVOA $W$ contains the conformal vector of $V$ and is therefore a conformal subVOA of $V$. By Proposition~\ref{prop:CMY-O2}, $W\in\IndCMY(\Dk)$. Put \(B=\widehat\Psi_k^{-1}(W).\) Lemma~\ref{lem:finite-generation}, applied to $\Psi_k$, shows that the commutative rational $\Otwo$-algebra $B$ is generated by the standard representation $\rho_1$.

Full faithfulness of $\widehat\Psi_k$ identifies \(B^{\Otwo}=\Hom_{\Otwo}(\C,B)\cong\Hom_{M(1)^+}(M(1)^+,W)\). A module map $M(1)^+\to W$ is determined by the image of the vacuum, and this image lies in $W_0=\C\one$ because $W\subset V$ and $V$ is of CFT type. Thus this Hom-space has dimension at most one. Since the extension unit $M(1)^+\hookrightarrow W$ is a nonzero $M(1)^+$-module map, the Hom-space is one-dimensional, and hence $B^{\Otwo}=\C1_B$. Let $j:\rho_1\hookrightarrow B$ be the generating copy. The map \(j\circ\iota_\rho^{-1}:\rho_1^*\longrightarrow B\) extends uniquely to a surjective unital $\Otwo$-algebra morphism \(\pi_B:\Sym(\rho_1^*)\twoheadrightarrow B\). Since $Q_\rho$ is $\Otwo$-invariant, there is a scalar $c\in\C$ such that \(\pi_B(Q_\rho)=c1_B\). We claim that $c\ne0$. Let \(\mu_0:\rho_1\otimes\rho_1\longrightarrow\C\) be the component of the multiplication of $B$ obtained by projection onto its unique trivial summand. Under the symmetric monoidal equivalence $\widehat\Psi_k$, this morphism corresponds to the $M(1)^+$-summand component of the algebra-object multiplication \((U\cdot F)\boxtimes(U\cdot F)\longrightarrow W\). Every nonunit simple object of $\Dk$ has positive lowest conformal weight, so the weight-zero coefficient of the product of two lowest-weight vectors lies entirely in this channel. Since $U$ is a conformal subVOA and $F$ is a $U$-highest-weight vector of lowest conformal weight $k$, one has $L(n)F=0$ for $n>0$ and $L(0)F=kF$; hence $F\in P_k(V)$. Lemma~\ref{lem:vacuum-channel} therefore gives \(F_{2k-1}F=(-1)^k(F,F)\one\ne0\). Since no nonunit simple object of $\Dk$ contains conformal weight zero, this nonzero weight-zero coefficient is a nonzero coefficient of the $M(1)^+$-channel of the corresponding intertwining operator. Hence the associated tensor-category morphism, and therefore $\mu_0$, is nonzero. Since \(\Hom_{\Otwo}(\rho_1\otimes\rho_1,\C)=\C\beta_Q\), there is $b\in\C^\times$ such that $\mu_0=b\beta_Q$. By the normalization of the inverse-form tensor, \(\pi_B(Q_\rho) =\mu_0(\check Q_\rho)1_B =b\beta_Q(\check Q_\rho)1_B =2b1_B\ne0.\) Therefore $c\in\C^\times$. The map $\pi_B$ factors through a surjective $\Otwo$-algebra homomorphism \(R_c=\frac{\Sym(\rho_1^*)}{(Q_\rho-c)}\twoheadrightarrow B\). Its kernel is a proper $\Otwo$-stable ideal of $R_c$, so Lemma~\ref{lem:O2-quadric} implies that the kernel is zero. Thus \(B\cong R_c.\)

By Lemma~\ref{lem:standard-O2-all-k}, the standard extension $V_L^+\supset M(1)^+$ corresponds to a commutative rational $\Otwo$-algebra \(B_k^{\rm std}\cong R_{c_k}\) for some $c_k\in\C^\times$. Choose $a\in\C^\times$ with $a^2c=c_k$. Sending the generating copy of $\rho_1$ in $R_{c_k}$ to $a$ times the generating copy in $R_c$ defines an $\Otwo$-algebra isomorphism \(R_{c_k}\overset\sim\longrightarrow R_c\). Applying $\widehat\Psi_k$ gives an isomorphism of commutative algebra objects in $\IndCMY(\Dk)$. By Lemma~\ref{lem:O2-admissible}, the commutative algebra objects $\widehat\Psi_k(R_{c_k})$ and $\widehat\Psi_k(R_c)$ are extension-admissible. Hence Theorem~\ref{thm:CMY-extension} transports the preceding algebra-object isomorphism to an isomorphism of VOA extensions of the fixed base $M(1)^+$. Consequently \(W\cong V_L^+\) and \(\ch W=\ch V_L^+=\ch V\). Since $W$ is a conformal subVOA of $V$, one has $W_n\subseteq V_n$ for every $n\ge0$. Equality of the characters therefore gives $\dim W_n=\dim V_n$ and hence $W_n=V_n$ for every $n$. Thus $W=V$, and so \(V\cong V_L^+.\)
\end{proof}

\section{The four low-norm lattice-orbifold cases}\label{sec:low}

Throughout this section, $V$ is assumed to satisfy the hypotheses of Theorem~\ref{thm:main}.

We retain the notation $P(q)$, $T(q)$ and $\grch$ from Proposition~\ref{prop:rank-one-lattice-package}(iv). In particular, for $L=\Z\alpha$ with $\langle\alpha,\alpha\rangle=2k$, \(\grch V_L^+=\frac{P(q)+T(q)}2+P(q)\sum_{r\ge1}q^{kr^2}.\)

\begin{prop}\label{prop:low-table}
Let $\Lambda_k=\Z\alpha$ with $\langle\alpha,\alpha\rangle=2k$ and write \(\grch V_{\Lambda_k}^+=\sum_{n\ge0}d_n^{(k)}q^n\). For $1\le k\le4$ and $0\le n\le9$,
\[
\begin{array}{c|rrrrrrrrrr}
 k\backslash n&0&1&2&3&4&5&6&7&8&9\\ \hline
1&1&1&2&3&7&9&15&21&32&44\\
2&1&0&2&2&5&6&11&14&24&30\\
3&1&0&1&2&4&5&9&12&19&25\\
4&1&0&1&1&4&4&8&10&17&21.
\end{array}
\]
\end{prop}

\begin{proof}
The expansions needed are \(P(q)=1+q+2q^2+3q^3+5q^4+7q^5+11q^6+15q^7+22q^8+30q^9+O(q^{10})\) and \(T(q)=1-q-q^3+q^4-q^5+q^6-q^7+2q^8-2q^9+O(q^{10}).\) Hence \(\frac{P(q)+T(q)}2 =1+q^2+q^3+3q^4+3q^5+6q^6+7q^7+12q^8+14q^9+O(q^{10}).\) By \eqref{eq:rank-one-character-interface}, only the terms with $kr^2\le9$ can contribute through weight nine. Thus the lattice contribution is \((q+q^4+q^9)P(q), (q^2+q^8)P(q), q^3P(q), q^4P(q)\) for $k=1,2,3,4$, respectively. Adding these four contributions to the preceding expansion of $(P(q)+T(q))/2$ and comparing coefficients through $q^9$ gives the stated table.
\end{proof}

\subsection{The cases \texorpdfstring{$k=1,2,3$}{k=1,2,3}}

\begin{thm}\label{thm:k1}
Let $\Lambda_1=\Z\alpha$ with $\langle\alpha,\alpha\rangle=2$. If \(\ch V=\ch V_{\Lambda_1}^+,\) then \(V\cong V_{\Lambda_1}^+.\)
\end{thm}

\begin{proof}
Proposition~\ref{prop:low-table} gives $\dim V_1=1$. By Proposition~\ref{prop:Li-invariant-package}(ii), the standing assumptions imply $L(1)V_1=0$, so $V$ is of strong CFT type in the terminology of Dong--Mason. Theorem~1 of \cite{DM} therefore implies that the Lie algebra $V_1$ is reductive. Since $\dim V_1=1$, its Lie rank is one. As $c(V)=\widetilde c(V)=1$, the equality case in \cite[Theorem~3]{DM} yields an isomorphism \(V\cong V_K\) for a positive-definite even rank-one lattice $K$.

Since $\Lambda_1\cong A_1$, we identify these lattices. Dong--Griess show in \cite[Section~3]{DG98} that \(V_{A_1}^+\cong V_{2A_1},\) where $2A_1=\Z(2\alpha)$ has generator of squared length $8$. Hence \(\ch V_K=\ch V_{2A_1}.\) Write $K=\Z\beta$ with $\langle\beta,\beta\rangle=2m$, $m\in\Z_{>0}$. By the rank-one lattice realization in Subsection~\ref{subsec:rank-one-lattice}, \(\grch V_K=P(q)\sum_{n\in\Z}q^{mn^2}, \grch V_{2A_1}=P(q)\sum_{n\in\Z}q^{4n^2}.\) Both vertex operator algebras have central charge one, so their equal vacuum characters have equal graded characters. Since $P(q)$ has constant term $1$ and is therefore invertible in $\C[[q]]$, it follows that \(\sum_{n\in\Z}q^{mn^2}=\sum_{n\in\Z}q^{4n^2}.\) The least positive exponent on the left is $m$, whereas on the right it is $4$; hence $m=4$ and $K\cong2A_1$. Therefore \(V\cong V_{2A_1}\cong V_{A_1}^+\cong V_{\Lambda_1}^+.\)
\end{proof}

\begin{thm}\label{thm:k2}
Let $\Lambda_2=\Z\alpha$ with $\langle\alpha,\alpha\rangle=4$. If \(\ch V=\ch V_{\Lambda_2}^+,\) then \(V\cong V_{\Lambda_2}^+.\)
\end{thm}

\begin{proof}
By Proposition~\ref{prop:low-table}, \(V_1=0\) and $\dim V_2=2$. Thus $V$ is of moonshine type, namely $V_0=\C\one$ and $V_1=0$. The standing assumptions also give rationality, $C_2$-cofiniteness, and $c(V)=\widetilde c(V)=1$. Hence \cite[Theorem~4.7]{ZhangDong} applies and gives \(V\cong L\!\left(\frac12,0\right)\otimes L\!\left(\frac12,0\right).\) For the rank-one lattice $\Lambda_2$ of squared norm four, \cite[Remark~3.2]{DG98} gives the VOA isomorphism \(V_{\Lambda_2}^+\cong L\!\left(\frac12,0\right)\otimes L\!\left(\frac12,0\right).\) Consequently $V\cong V_{\Lambda_2}^+$.
\end{proof}

\begin{thm}\label{thm:k3}
Let $\Lambda_3=\Z\alpha$ with $\langle\alpha,\alpha\rangle=6$. If \(\ch V=\ch V_{\Lambda_3}^+,\) then \(V\cong V_{\Lambda_3}^+.\)
\end{thm}

\begin{proof}
Corollary~\ref{cor:M1plus-kgt4} gives a conformal subVOA \(U\cong M(1)^+.\) Fix such an isomorphism and henceforth identify $U$ with $M(1)^+$. Lemma~\ref{lem:first-lattice}, now with $k=3$ and $L=\Lambda_3$, gives \(K:=U^\perp, K_j=0\ (j<3), \dim K_3=1,\) and, for every $0\ne F\in K_3$, the vector $F$ is nonisotropic and \(E:=U\cdot F\cong M(1,\mu), \langle\mu,\mu\rangle=6.\) Choose the charge $\lambda\in\mathfrak h$ of norm six used in the specialization $\mathcal D_3$ of Theorem~\ref{thm:O2-category}. Since $\mathfrak h$ is one-dimensional, $\mu=\pm\lambda$, and $M(1,\lambda)\cong M(1,-\lambda)$ as $M(1)^+$-modules. Thus \(E\cong M(1,\lambda)=\Psi_3(\rho_1)\in\mathcal D_3.\)

Set \(W=\langle U,E\rangle_{\mathrm{VOA}}\). By Proposition~\ref{prop:CMY-O2}, $W\in\IndCMY(\mathcal D_3)$; through Theorem~\ref{thm:CMY-extension}, regard $W$ as a commutative algebra object and put \(B:=\widehat\Psi_3^{-1}(W).\) Lemma~\ref{lem:finite-generation} shows that $B$ is generated by its copy of $\rho_1$. Moreover, $W_0=\C\one$, so the tensor unit $M(1)^+$ occurs in $W$ with multiplicity one; equivalently, \(B^{\Otwo}=\C1.\) Let \(j:\rho_1\hookrightarrow B\) be the generating summand. The self-duality $\iota_\rho$ extends $j\circ\iota_\rho^{-1}$ uniquely to a surjective $\Otwo$-algebra homomorphism \(\pi_B:\Sym(\rho_1^*)\twoheadrightarrow B.\) Since $Q_\rho$ is invariant, \(\pi_B(Q_\rho)=c1\) for some $c\in\C$.

We claim that $c\ne0$. Let \(\mu_0:\rho_1\otimes\rho_1\to\C\) be the trivial-summand component of the multiplication of $B$. Under $\Psi_3$, it corresponds to the $U$-summand component of \(E\boxtimes E\to W\). Since the conformal vector belongs to $U$ and $K_j=0$ for $j<3$, one has $F\in P_3(V)$. Hence Lemma~\ref{lem:vacuum-channel} gives \(F_5F=-(F,F)\one\ne0\), so $\mu_0\ne0$. By Schur's lemma, \(\Hom_{\Otwo}(\rho_1\otimes\rho_1,\C)=\C\beta_Q\), and therefore \(\mu_0=b\beta_Q\) for some $b\in\C^\times$. By the definition of $\pi_B$ and $\check Q_\rho$, \(c=\mu_0(\check Q_\rho)=b \beta_Q(\check Q_\rho)=2b\ne0.\) Thus $\pi_B$ factors through a surjection \(R_c\twoheadrightarrow B\). Its kernel is a proper $\Otwo$-stable ideal of $R_c$, so Lemma~\ref{lem:O2-quadric} gives \(B\cong R_c\).

For the standard extension $V_{\Lambda_3}^+\supset M(1)^+$, Lemma~\ref{lem:standard-O2-all-k} gives \(B_3^{\rm std}\cong R_{c_3}\) with $c_3\ne0$. Choose $a\in\C^\times$ with $a^2c=c_3$. Scaling the generating copy of $\rho_1$ by $a$ gives an $\Otwo$-algebra isomorphism \(R_{c_3}\overset\sim\longrightarrow R_c\), hence \(B_3^{\rm std}\cong B\). Applying $\widehat\Psi_3$ and then the extension--algebra correspondence of Theorem~\ref{thm:CMY-extension} therefore gives an isomorphism of VOA extensions \(V_{\Lambda_3}^+\cong W\). Consequently \(\ch W=\ch V\). Since $W$ is a conformal subVOA of $V$, one has $W_n\subseteq V_n$ for every $n\ge0$; equality of characters gives equality of dimensions in each degree, hence $W_n=V_n$ for all $n$ and therefore $W=V$. Thus \(V\cong V_{\Lambda_3}^+\).
\end{proof}

\subsection{The case \texorpdfstring{$k=4$}{k=4}: the first-singular map}

Let $\Lambda_4=\Z\alpha$ with $\langle\alpha,\alpha\rangle=8$, and assume \(\ch V=\ch V_{\Lambda_4}^+.\) Since both VOAs have central charge one, this is equivalent to equality of their graded characters, so Proposition~\ref{prop:low-table} applies. Let $V(1,4)$ be the Verma module with highest-weight vector $v_4$, and write $V(1,4)[r]$ for its level-$r$ subspace. By Proposition~\ref{prop:c1-verma}, fix a nonzero first singular vector $s_5\in V(1,4)[5]$, unique up to scalar. For $u\in P_4(V)$, let \(\phi_u:V(1,4)\to V\) be the canonical Virasoro homomorphism with $\phi_u(v_4)=u$, and set \(\Theta:V(1,4)\otimes P_4(V)\longrightarrow V, \Theta(x\otimes u)=\phi_u(x), S(u)=\phi_u(s_5)\). Thus $S:P_4(V)\to V_9$ is a linear map, which we call the \emph{first-singular map}. Its vanishing is independent of the scalar chosen for $s_5$.

\begin{lem}\label{lem:k4-saturation8}
The space $P_4(V)$ is two-dimensional and nondegenerate. For every $0\le r\le4$, the restriction \(\Theta_r:V(1,4)[r]\otimes P_4(V)\longrightarrow V_{4+r}\) of $\Theta$ is injective. Moreover, for $4\le n\le8$, \(V_n=L(1,0)_n\oplus \Theta\bigl(V(1,4)[n-4]\otimes P_4(V)\bigr).\)
\end{lem}

\begin{proof}
Proposition~\ref{prop:c1-verma} with $m=0$ and PBW give \(\grch L(1,0)=(1-q)P(q).\) Together with Proposition~\ref{prop:low-table}, this gives $V_n=L(1,0)_n$ for $0\le n<4$ and \(\dim V_4=4=\dim L(1,0)_4+2.\) Lemma~\ref{lem:primary-complement} therefore yields \(V_4=L(1,0)_4\mathbin{\perp}P_4(V), \dim P_4(V)=2,\) and the restriction of the invariant form to $P_4(V)$ is nondegenerate.

The level-zero restriction of $\Theta$ is injective, since \(\Theta(v_4\otimes u)=u\) for $u\in P_4(V)$. Suppose that $\Theta_r$ is not injective for some $1\le r\le4$, and choose the least such $r$ and $0\ne x\in\ker\Theta_r$. Since $\Theta$ is a Virasoro homomorphism, for every $n>0$ one has $L(n)x\in\ker\Theta$. If $0<n<r$, then $L(n)x$ lies at level $r-n$ and vanishes by the minimality of $r$; if $n=r$, it lies in the level-zero kernel, which is zero; and if $n>r$, it has negative level and is zero. Thus $x$ is a singular vector of positive level $r<5$ in $V(1,4)\otimes P_4(V)$. After choosing a basis of the two-dimensional space $P_4(V)$, this tensor product is a direct sum of two copies of $V(1,4)$, so at least one nonzero component of $x$ is a singular vector of $V(1,4)$ at level $r$. This contradicts Proposition~\ref{prop:c1-verma}, according to which the first nontrivial singular vector of $V(1,4)$ occurs at level five. Hence $\Theta_r$ is injective for every $0\le r\le4$.

The image of $\Theta$ is a Virasoro submodule supported in conformal weights at least four. Hence \(L(1,0)\cap\operatorname{Im}\Theta=0\): indeed, a nonzero intersection would be a nonzero Virasoro submodule of the simple module $L(1,0)$, therefore all of $L(1,0)$, which is impossible because $\operatorname{Im}\Theta$ contains no weight-zero vector. By PBW, the level-$r$ dimensions of $V(1,4)$ for $0\le r\le4$ are the first five coefficients $1,1,2,3,5$ of $P(q)$. Proposition~\ref{prop:low-table} and \(\grch L(1,0)=(1-q)P(q)\) give \(\dim V_n-\dim L(1,0)_n=2,2,4,6,10 (n=4,5,6,7,8),\) which are exactly twice those level dimensions. Since $\dim P_4(V)=2$ and the relevant restrictions of $\Theta$ are injective, the displayed direct sum has dimension $\dim V_n$ for every $4\le n\le8$, and hence equals $V_n$.
\end{proof}

\begin{lem}\label{lem:k4-kernel}
Let \(\Theta_5:=\Theta|_{V(1,4)[5]\otimes P_4(V)}: V(1,4)[5]\otimes P_4(V)\longrightarrow V_9\). Then \(\ker\Theta_5=\C s_5\otimes\ker S\), and hence \(\dim\operatorname{Im}\Theta_5=12+\operatorname{rank}S.\)
\end{lem}

\begin{proof}
Let $x\in\ker\Theta_5$. The Virasoro algebra acts on $V(1,4)\otimes P_4(V)$ through the first tensor factor, and $\Theta$ is a Virasoro homomorphism. Hence, for every $n>0$, \(\Theta(L(n)x)=L(n)\Theta(x)=0\). If $1\le n\le5$, then $L(n)x$ lies in level $5-n$, where the corresponding restriction of $\Theta$ is injective by Lemma~\ref{lem:k4-saturation8}; thus $L(n)x=0$. If $n>5$, then $L(n)x=0$ because $V(1,4)$ has no negative levels. Therefore $x$ is a level-five singular vector.

By Proposition~\ref{prop:c1-verma}, the level-five singular space of $V(1,4)$ is $\C s_5$. Consequently the level-five singular space of $V(1,4)\otimes P_4(V)$ is \(\C s_5\otimes P_4(V)\). On this subspace, \(\Theta_5(s_5\otimes u)=S(u)\), so \(\ker\Theta_5=\C s_5\otimes\ker S.\) By PBW, $\dim V(1,4)[5]$ is the coefficient of $q^5$ in \(P(q)=\prod_{n\ge1}(1-q^n)^{-1}\), namely $7$. Since $\dim P_4(V)=2$ by Lemma~\ref{lem:k4-saturation8}, rank--nullity gives \(\dim\operatorname{Im}\Theta_5 =14-\dim\ker\Theta_5 =14-\dim\ker S =12+\operatorname{rank}S\).
\end{proof}

\begin{thm}\label{thm:k4-Szero}
In the case $k=4$, the first-singular map satisfies $S=0$. Consequently, for every $0\ne u\in P_4(V)$, \(U(\Vir)u\cong X_2=L(1,4).\)
\end{thm}

\begin{proof}
Proposition~\ref{prop:low-table} gives $\dim V_9=21$, while the Virasoro character computation used in Lemma~\ref{lem:k4-saturation8} gives $\dim L(1,0)_9=8$. By Lemma~\ref{lem:k4-kernel}, \(\dim\operatorname{Im}\Theta_5=12+\operatorname{rank}S\). Also \(L(1,0)\cap\operatorname{Im}\Theta=0\): the intersection is a Virasoro submodule of the simple module $L(1,0)$, hence if nonzero it would equal all of $L(1,0)$, which is impossible because $\operatorname{Im}\Theta$ has no weight-zero vectors. Therefore \(8+12+\operatorname{rank}S =\dim\bigl(L(1,0)_9\oplus\operatorname{Im}\Theta_5\bigr) \le21\), and hence $\operatorname{rank}S\le1$.

Suppose that $\operatorname{rank}S=1$, and choose $u\in P_4(V)$ with $s:=S(u)\ne0$. Since $s_5$ is a level-five singular vector and $\phi_u$ is a Virasoro homomorphism, $s$ is a primary vector of conformal weight nine. The preceding dimension inequality is then an equality, so \(V_9=L(1,0)_9\oplus\operatorname{Im}\Theta_5.\) We show that $s$ is orthogonal to both summands. By PBW, $L(1,0)_9$ is spanned by vacuum descendants \(L(-n_1)\cdots L(-n_r)\one\) of total degree nine. Invariance of the bilinear form moves the negative Virasoro modes to positive modes acting on $s$, and these annihilate $s$ because $s$ is primary. Hence \(s\perp L(1,0)_9\). Likewise, $\operatorname{Im}\Theta_5$ is spanned by vectors of the form \(L(-n_1)\cdots L(-n_r)v\), where $v\in P_4(V)$ and $n_1+\cdots+n_r=5$. Again invariance and primarity give \(\bigl(s,L(-n_1)\cdots L(-n_r)v\bigr)=0.\) Thus $s\perp V_9$. By Proposition~\ref{prop:Li-invariant-package}, the invariant form restricts nondegenerately to $V_9$, so $s=0$, a contradiction. Therefore $\operatorname{rank}S=0$, and hence $S=0$.

Now let $0\ne u\in P_4(V)$. The canonical homomorphism \(\phi_u:V(1,4)\to V\) is nonzero because it sends the highest-weight vector $v_4$ to $u$. Since $S=0$, one has $\phi_u(s_5)=0$. By Proposition~\ref{prop:c1-verma}, the submodule generated by $s_5$ is the unique maximal proper submodule of $V(1,4)$, and the corresponding simple quotient is $X_2=L(1,4)$. Hence $\phi_u$ factors through a nonzero Virasoro homomorphism \(\overline\phi_u:X_2=L(1,4)\longrightarrow V.\) Since $X_2$ is simple, $\overline\phi_u$ is injective, and its image is exactly $U(\Vir)u$. Therefore \(U(\Vir)u\cong X_2=L(1,4)\).
\end{proof}

\subsection{The regular spin-two orbit}

\begin{lem}\label{lem:complex-symmetric}
Let $A$ be a symmetric $3\times3$ complex matrix with three distinct eigenvalues. Then there exists $g\in SO_3(\C)$ such that $gAg^{-1}$ is diagonal.
\end{lem}

\begin{proof}
Let $\beta(v,w)=v^{\mathsf T}w$ be the standard nondegenerate symmetric bilinear form on $\C^3$. Since $A^{\mathsf T}=A$, one has $\beta(Av,w)=\beta(v,Aw)$. Choose eigenvectors $v_i\ne0$ with $Av_i=\lambda_i v_i$, where the $\lambda_i$ are pairwise distinct. Then \((\lambda_i-\lambda_j)\beta(v_i,v_j)=0,\) so the three eigenlines are pairwise orthogonal. As the eigenvalues are distinct, $v_1,v_2,v_3$ form a basis of $\C^3$. Moreover, $\beta(v_i,v_i)\ne0$ for every $i$: otherwise $v_i$ would be orthogonal to each basis vector and hence to all of $\C^3$, contradicting the nondegeneracy of $\beta$.

Because $\C$ is algebraically closed, each $v_i$ may be rescaled to a vector $u_i$ with $\beta(u_i,u_i)=1$. Thus $u_1,u_2,u_3$ is an orthonormal eigenbasis. Let $P$ be the matrix with columns $u_1,u_2,u_3$. Then $P^{\mathsf T}P=I_3$, so $P\in O_3(\C)$ and $\det P=\pm1$. Replacing one $u_i$ by $-u_i$ if necessary, we may assume $\det P=1$, hence $P\in SO_3(\C)$. Finally, \(P^{-1}AP=\operatorname{diag}(\lambda_1,\lambda_2,\lambda_3).\) Taking $g=P^{-1}\in SO_3(\C)$ proves the claim.
\end{proof}

Retain $q(A)=\operatorname{tr}(A^2)$ and define the cubic invariant \(d(A):=\operatorname{tr}(A^3)\) for $A\in E$. Newton's identities give \(\det(tI_3-A)=t^3-\frac12q(A)t-\frac13d(A)\), so the discriminant of the characteristic polynomial is \(\Delta(A):=\frac12q(A)^3-3d(A)^2.\) Thus the regular locus \(E_{\mathrm{reg}}:=\{A\in E:\Delta(A)\ne0\}\) consists precisely of the matrices with three distinct eigenvalues. It is a nonempty Zariski-open subset of the affine space $E$, since it contains $\operatorname{diag}(1,0,-1)$, and hence is dense in $E$.

Let \(\mathfrak a :=\{\operatorname{diag}(x_1,x_2,x_3):x_1+x_2+x_3=0\} \subset E\) be the traceless diagonal subspace, and consider the restriction map \(\operatorname{res}:\C[E]^G\longrightarrow\C[\mathfrak a].\) By Lemma~\ref{lem:complex-symmetric}, every element of $E_{\mathrm{reg}}$ is $G$-conjugate to an element of $\mathfrak a$ with three distinct diagonal entries. Hence, if $f\in\C[E]^G$ satisfies $\operatorname{res}(f)=0$, then $f$ vanishes on $E_{\mathrm{reg}}$ and therefore, by density, on all of $E$. Thus $\operatorname{res}$ is injective.

Every permutation of the diagonal entries is induced by an element of $SO_3(\C)$. Indeed, for a permutation matrix $P_\sigma$, choose a diagonal sign matrix $D_\sigma$ with $\det D_\sigma=\det P_\sigma$ and set $g_\sigma=D_\sigma P_\sigma\in SO_3(\C)$. Since $D_\sigma$ centralizes $\mathfrak a$, conjugation by $g_\sigma$ permutes the diagonal entries according to $\sigma$. Therefore \(\operatorname{res}(f)\in\C[\mathfrak a]^{S_3}\) for every $f\in\C[E]^G$.

By the fundamental theorem of symmetric polynomials, \(\C[\mathfrak a]^{S_3}=\C[e_2,e_3]\), where $e_i$ are the elementary symmetric functions in $x_1,x_2,x_3$ and $e_1=x_1+x_2+x_3=0$. Newton's identities give \(e_2=-\frac12q|_{\mathfrak a}, e_3=\frac13d|_{\mathfrak a}\). Since $q,d\in\C[E]^G$, every restricted invariant is the restriction of a polynomial in $q$ and $d$. The injectivity of $\operatorname{res}$ therefore yields \(\C[E]^G=\C[q,d].\)

\begin{lem}\label{lem:regular-fiber}
Let $c_2,c_3\in\C$ satisfy \(\frac12c_2^3-3c_3^2\ne0.\) Then the affine scheme \(\mathcal F_{c_2,c_3} =\Spec\frac{\Sym(E^*)}{(q-c_2,d-c_3)}\) is smooth and reduced, and \(\mathcal F_{c_2,c_3}\cong G/D_2,\) where $D_2\cong(\Z/2)^2$ is the Klein four subgroup of coordinate-axis half-turns.
\end{lem}

\begin{proof}
For a traceless $3\times3$ matrix $A$, Newton's identities give \(\det(tI_3-A) =t^3-\frac12q(A)t-\frac13d(A)\). Set \(f(t)=t^3-\frac12c_2t-\frac13c_3.\) Its discriminant is \(\frac12c_2^3-3c_3^2\), so by hypothesis it has three distinct roots $x_1,x_2,x_3$. Their sum is zero, and Newton's identities give \(\sum_i x_i^2=c_2\) and \(\sum_i x_i^3=c_3\). Hence \(A_0:=\operatorname{diag}(x_1,x_2,x_3)\in \mathcal F_{c_2,c_3}(\C).\)

Let $A\in\mathcal F_{c_2,c_3}(\C)$. Its characteristic polynomial is $f$, so its eigenvalues are precisely $x_1,x_2,x_3$. By Lemma~\ref{lem:complex-symmetric}, $A$ is conjugate under $G\cong SO_3(\C)$ to a diagonal matrix with these entries in some order. As observed above, every permutation of the diagonal entries is induced by an element of $SO_3(\C)$. Thus \(\mathcal F_{c_2,c_3}(\C)=G\cdot A_0.\) Since the eigenvalues of $A_0$ are distinct, an element of $G$ stabilizes $A_0$ precisely when it preserves each coordinate eigenline. Hence
\[
G_{A_0}
 =\{\operatorname{diag}(\varepsilon_1,\varepsilon_2,\varepsilon_3):
   \varepsilon_i\in\{\pm1\},\
   \varepsilon_1\varepsilon_2\varepsilon_3=1\}
 \cong(\Z/2)^2=:D_2.
\]

It remains to verify that the scheme-theoretic fiber has no nilpotent thickening. For $A,H\in E$, \((dq)_A(H)=2\operatorname{tr}(AH), (dd)_A(H)=3\operatorname{tr}(A^2H).\) The trace pairing on $\Sym_3(\C)$ is nondegenerate, and $E=\Sym^2_0(\C^3)$ is the trace-orthogonal complement of $\C I_3$. Consequently \(E^\perp=\C I_3\) inside $\Sym_3(\C)$, and the trace pairing restricts nondegenerately to $E$. At every $A\in\mathcal F_{c_2,c_3}(\C)$ the eigenvalues are distinct, so $A\ne0$ and therefore $(dq)_A\ne0$ on $E$.

Suppose that $(dq)_A$ and $(dd)_A$ were linearly dependent. Then, since $(dq)_A\ne0$, there would be $\mu\in\C$ such that $(dd)_A=\mu(dq)_A$ on $E$. Thus \(\operatorname{tr}\bigl((3A^2-2\mu A)H\bigr)=0 (H\in E).\) Since $E^\perp=\C I_3$, there is $\nu\in\C$ such that \(3A^2-2\mu A=\nu I_3.\) Every eigenvalue of $A$ would then be a root of the quadratic polynomial \(3t^2-2\mu t-\nu\), contradicting the existence of three distinct eigenvalues. Hence the Jacobian of $(q,d)$ has rank two at every closed point of the fiber.

By the Jacobian criterion recalled in Subsection~\ref{subsec:AG-conventions}, every closed point of $\mathcal F_{c_2,c_3}$ is smooth of codimension two in the five-dimensional affine space $E$. Since the nonsmooth locus is closed and a finite-type $\C$-scheme is Jacobson, a nonempty nonsmooth locus would contain a closed point. Therefore $\mathcal F_{c_2,c_3}$ is smooth everywhere, of dimension three, and in particular is reduced.

Let $\mathcal O:=G\cdot A_0$. Since $q$ and $d$ are $G$-invariant, $\mathcal O\subseteq\mathcal F_{c_2,c_3}$. By the orbit facts recalled in Subsection~\ref{subsec:AG-conventions}, $\mathcal O$ is locally closed and \(\mathcal O\cong G/G_{A_0}=G/D_2.\) Every closed point of $\mathcal F_{c_2,c_3}$ lies in $\mathcal O$. Hence the constructible complement \(|\mathcal F_{c_2,c_3}|\setminus|\mathcal O|\) is empty: otherwise, by the Jacobson property recalled in the same subsection, it would contain a closed point. Thus \(|\mathcal F_{c_2,c_3}|=|\mathcal O|\), so the orbit is closed in $E$. Endow $\mathcal O$ with its reduced induced scheme structure. Then $\mathcal F_{c_2,c_3}$ and $\mathcal O$ are reduced closed subschemes of $E$ with the same $\C$-points. The Nullstellensatz consequence recalled in Subsection~\ref{subsec:AG-conventions} therefore gives \(\mathcal F_{c_2,c_3}=\mathcal O\cong G/D_2\) scheme-theoretically.
\end{proof}

\begin{lem}\label{lem:k4-parameters}
Put $P:=P_4(V)$, and let $G$ act trivially on the multiplicity space $P$. There is a commutative rational $G$-algebra $B_P$, generated by $E\otimes P$, with the following properties.
\begin{enumerate}[label=(\roman*),leftmargin=2.1em]
\item For every $u\in P$, if $B_u\subseteq B_P$ denotes the unital
$G$-subalgebra generated by $E\otimes u$, then under $\widehat\Phi$ it corresponds to the vertex operator subalgebra \(W_u:=\langle L(1,0),u\rangle_{\mathrm{VOA}}.\)
\item Let
\(\pi_u:\Sym(E^*)\longrightarrow B_u, \xi\longmapsto \iota_E(\xi)\otimes u,\) be the induced $G$-algebra epimorphism. Then \(\pi_u(q)=q_V(u)\one, \pi_u(d)=d_V(u)\one,\) for homogeneous forms \(q_V\in\Sym^2(P^*)\) and \(d_V\in\Sym^3(P^*)\).
\item There is a scalar $\kappa\in\C^\times$ such that
\(q_V(u)=\kappa (u,u) (u\in P),\) where $(\,\cdot\,,\,\cdot\,)$ is the invariant bilinear form of $V$ restricted to $P$. In particular, $q_V$ is nondegenerate.
\end{enumerate}
\end{lem}

\begin{proof}
By Theorem~\ref{thm:k4-Szero}, the map \(\Theta:V(1,4)\otimes P\to V\) from the preceding subsection kills $s_5\otimes P$. Proposition~\ref{prop:c1-verma} identifies the quotient of $V(1,4)$ by the submodule generated by $s_5$ with $X_2=L(1,4)$. Hence $\Theta$ factors through a Virasoro homomorphism \(\overline\Theta:X_2\otimes P\longrightarrow V\) whose restriction to the highest-weight space is the identity $P\to P_4(V)$. Since $X_2\otimes P$ is a finite direct sum of copies of the simple module $X_2$, it is semisimple; any nonzero submodule meets its highest-weight space. Thus $\overline\Theta$ is injective. We henceforth identify $X_2\otimes P$ with its image in $V$.

Let \(U_P:=\langle L(1,0),X_2\otimes P\rangle_{\mathrm{VOA}}.\) Theorem~\ref{thm:square-core} gives $U_P\in\IndCMY(\Csq)$. Set \(B_P:=\widehat\Phi^{-1}(U_P)\in\CAlg(\Rep^{\mathrm{rat}}G).\) Since $U_P$ is generated by $L(1,0)$ and the finite direct sum $X_2\otimes P$, Lemma~\ref{lem:finite-generation} shows that $B_P$ is generated by the corresponding $G$-submodule $E\otimes P$. Moreover, full faithfulness and $\widehat\Phi(\C)=X_0=L(1,0)$ give \(B_P^G=\Hom_G(\C,B_P) \cong \Hom_{L(1,0)}(X_0,U_P)=\C,\) because $(U_P)_0=\C\one$ and the vacuum embedding supplies the unique copy of $X_0$.

Fix $u\in P$. The equivariant map \(j_u:E^*\to B_P\), \(j_u(\xi)=\iota_E(\xi)\otimes u\), extends uniquely to a $G$-algebra homomorphism \(\pi_u:\Sym(E^*)\to B_P\); its image is the unital $G$-subalgebra $B_u$ generated by $E\otimes u$. On the VOA side, put \(W_u=\langle L(1,0),u\rangle_{\mathrm{VOA}}.\) Theorem~\ref{thm:square-core} gives $W_u\in\IndCMY(\Csq)$. Let \(\widetilde B_u:=\widehat\Phi^{-1}(W_u)\subseteq B_P\); the inclusion $W_u\hookrightarrow U_P$ corresponds to a monomorphism of commutative $G$-algebras. Since $\widetilde B_u$ contains $E\otimes u$, minimality of $B_u$ gives $B_u\subseteq\widetilde B_u$.

Conversely, $B_u\hookrightarrow B_P$ is a monomorphism of commutative algebra objects. Its image under $\widehat\Phi$ is an $L(0)$-stable submodule of the grading-restricted VOA $U_P$, so it is lower bounded with finite-dimensional homogeneous spaces. Proposition~\ref{prop:subalgebra-interface} therefore identifies $\widehat\Phi(B_u)$ with a vertex operator subalgebra of $U_P$. It contains $L(1,0)$ and the copy of $X_2$ whose highest-weight vector is $u$, hence it contains $W_u$. Applying $\widehat\Phi^{-1}$ gives $\widetilde B_u\subseteq B_u$. Therefore $B_u=\widetilde B_u$, proving~(i).

Since $q,d\in\Sym(E^*)^G$ and $B_P^G=\C\one$, there are unique scalars $q_V(u),d_V(u)$ such that \(\pi_u(q)=q_V(u)\one\) and \(\pi_u(d)=d_V(u)\one\). The map $u\mapsto j_u$ is linear. Hence, for a homogeneous polynomial $f\in\Sym^r(E^*)$, the scalar obtained from $\pi_u(f)$ is a homogeneous polynomial of degree $r$ in $u$. Thus $q_V\in\Sym^2(P^*)$ and $d_V\in\Sym^3(P^*)$, proving~(ii).

It remains to identify $q_V$. Because $B_P$ is a semisimple rational $G$-module and $B_P^G=\C\one$, let $p_0:B_P\to\C\one$ be the $G$-equivariant projection onto the trivial isotypic component. Write $\tau(x,y)=\operatorname{tr}(xy)$ for the fixed trace form on $E$. Since \(\Hom_G(E\otimes E,\C)=\C\tau\) and $G$ acts trivially on $P$, there is a unique bilinear form $\beta_P$ on $P$ such that the unit component of multiplication satisfies \(p_0\bigl((x\otimes u)(y\otimes v)\bigr) =\tau(x,y)\beta_P(u,v)\one (x,y\in E,\ u,v\in P)\). Commutativity of $B_P$ and symmetry of $\tau$ imply that $\beta_P$ is symmetric.

Under the fixed tensor equivalence \(\Phi:\Rep G\overset\sim\longrightarrow\Csq\), the nondegenerate contraction $\tau:E\otimes E\to\C$ is an evaluation morphism for the self-duality of $E$. A symmetric monoidal equivalence preserves duality, so $\Phi(\tau):X_2\boxtimes X_2\to X_0$ is an evaluation morphism for a self-duality of $X_2$. The corresponding invariant pairing on $X_2$ is nondegenerate. Since distinct conformal-weight spaces are orthogonal, its restriction to the one-dimensional highest-weight space of $X_2$ is nonzero. Consequently there is a scalar $\gamma_E\in\C^\times$, depending only on the fixed normalization of $\tau$ and the tensor equivalence, such that for $u,v\in P$ the vacuum component of their product in $U_P$ is \(\gamma_E\beta_P(u,v)z^{-8}\one\). Lemma~\ref{lem:vacuum-channel}, applied in $V$ with $h=4$, gives the same coefficient as \((u,v)z^{-8}\one\). Hence \(\beta_P(u,v)=\gamma_E^{-1}(u,v) (u,v\in P).\)

Finally, under the trace self-duality $\iota_E:E^*\overset\sim\to E$, the tensor \(\check q_E:=(\iota_E\otimes\iota_E)(q)\in(\Sym^2E)^G\) is the inverse Casimir tensor for $\tau$, so \(\tau(\check q_E)=\dim E=5\). Since $\pi_u(q)$ is already scalar, applying $p_0$ to its expression through the multiplication of $B_P$ gives \(q_V(u)=5\beta_P(u,u)=5\gamma_E^{-1}(u,u).\) Thus~(iii) holds with $\kappa=5\gamma_E^{-1}\ne0$. The restriction of the invariant form to $P=P_4(V)$ is nondegenerate by Lemma~\ref{lem:k4-saturation8}; therefore $q_V$ is nondegenerate.
\end{proof}

\begin{lem}\label{lem:regular-primary}
Call $u\in P_4(V)$ \emph{regular} if \(\Delta_V(u):=\frac12q_V(u)^3-3d_V(u)^2\ne0.\) Then regular vectors exist in $P_4(V)$. Equivalently, if \(c_2(u):=q_V(u)\) and \(c_3(u):=d_V(u)\) are the scalar parameters associated with $B_u$, then $u$ is regular precisely when \(\frac12c_2(u)^3-3c_3(u)^2\ne0\).
\end{lem}

\begin{proof}
Set $P=P_4(V)$ and \(\Delta_V:=\frac12q_V^3-3d_V^2\in\Sym^6(P^*)=\C[P]\). Suppose, to the contrary, that $\Delta_V(u)=0$ for every $u\in P$. Since a polynomial function on the affine space $P$ that vanishes at all of its $\C$-points is the zero polynomial, we obtain the identity \(q_V^3=6d_V^2 \text{in }\C[P].\) By Lemma~\ref{lem:k4-saturation8}, $\dim P=2$, and by Lemma~\ref{lem:k4-parameters}, $q_V$ is nondegenerate. Hence, over $\C$, there are distinct nonproportional linear forms $\ell_1,\ell_2\in P^*$ and a scalar $a\in\C^\times$ such that \(q_V=a\ell_1\ell_2\). The coordinate ring \(\C[P]\cong\C[x,y]\) is a unique factorization domain. Thus the irreducible factor $\ell_1$ occurs with multiplicity three in $q_V^3$, whereas every irreducible factor of the square $6d_V^2$ occurs with even multiplicity, since $6\in\C^\times$ is a unit. This contradiction proves the claim.
\end{proof}

\begin{thm}\label{thm:k4-D2}
Let $u\in P_4(V)$ be regular in the sense of Lemma~\ref{lem:regular-primary}. Then
\[
\langle L(1,0),u\rangle_{\mathrm{VOA}}\cong V_{A_1}^{D_2}
\]
as vertex operator algebra extensions of $L(1,0)$, where $D_2\cong(\Z/2)^2$ is the coordinate-axis half-turn subgroup of $G\cong SO_3(\C)$ appearing in Lemma~\ref{lem:regular-fiber}.
\end{thm}

\begin{proof}
Set \(W_u:=\langle L(1,0),u\rangle_{\mathrm{VOA}}\), \(c_2:=q_V(u)\), and \(c_3:=d_V(u)\). By Lemma~\ref{lem:k4-parameters}\textnormal{(i)--(ii)}, the VOA extension $W_u$ corresponds under $\widehat\Phi$ to the unital commutative rational $G$-algebra $B_u$, and there is a surjective $G$-algebra homomorphism \(\pi_u:\Sym(E^*)\twoheadrightarrow B_u\) satisfying \(\pi_u(q)=c_2\one\) and \(\pi_u(d)=c_3\one\). Consequently the $G$-stable ideal \((q-c_2,d-c_3)\) is contained in $\ker\pi_u$, so $\pi_u$ factors as a surjective unital $G$-algebra homomorphism \(\overline\pi_u: A_u:=\frac{\Sym(E^*)}{(q-c_2,d-c_3)}\twoheadrightarrow B_u.\)

Since $u$ is regular, \(\frac12c_2^3-3c_3^2\ne0\). Lemma~\ref{lem:regular-fiber} therefore identifies the affine $G$-scheme $\Spec A_u$ $G$-equivariantly with $G/D_2$; equivalently, \(A_u\cong\OO(G/D_2)\) as commutative rational $G$-algebras. By Lemma~\ref{lem:G-simple-homogeneous}, $A_u$ has no nonzero proper $G$-stable ideals. The kernel of $\overline\pi_u$ is a $G$-stable ideal, and it is proper because $\overline\pi_u$ is unital. Hence \(\ker\overline\pi_u=0\), and therefore \(B_u\cong\OO(G/D_2)\) as commutative rational $G$-algebras.

The subgroup $D_2$ is finite, hence closed and reductive over $\C$, and $G/D_2$ is affine by Lemma~\ref{lem:regular-fiber}. Thus Corollary~\ref{cor:fixed-point-dictionary} applies and yields \(W_u\cong V_{A_1}^{D_2}\) as VOA extensions of $L(1,0)$.
\end{proof}

\begin{lem}\label{lem:D2-lattice}
Let $D_2\le G$ be the coordinate-axis Klein four subgroup fixed in Lemma~\ref{lem:regular-fiber}. Write $A_1=\Z\gamma$ with $\langle\gamma,\gamma\rangle=2$, and set \(2A_1:=2\Z\gamma\subset A_1.\) Then \(V_{A_1}^{D_2}\cong V_{2A_1}^+\) as vertex operator algebras.
\end{lem}

\begin{proof}
Let $\chi_{-1}\in T_{A_1}=\Hom(A_1,U(1))$ be the character determined by \(\chi_{-1}(\gamma)=-1\), and let $\sigma$ be the corresponding lattice-charge automorphism from Proposition~\ref{prop:rank-one-lattice-package}\textnormal{(v)}. Thus \(\sigma(u\otimes e^{n\gamma})=(-1)^n u\otimes e^{n\gamma} (u\in M(1),\ n\in\Z).\) Let $\theta$ be the standard lift of the lattice isometry $-1$ fixed in Subsection~\ref{subsec:rank-one-lattice}. By Proposition~\ref{prop:rank-one-lattice-package}\textnormal{(v)}, \(\theta\sigma\theta^{-1}=\sigma^{-1}=\sigma\), so $\sigma$ and $\theta$ commute. Both are nontrivial involutions, and they are distinct because $\sigma$ fixes $M(1)$ pointwise whereas $\theta$ acts as $-1$ on $M(1)_1$. Hence \(H:=\langle\sigma,\theta\rangle\cong(\Z/2)^2.\)

We next compare $H$ with the subgroup $D_2$ occurring in Lemma~\ref{lem:regular-fiber}. The rank-one lattice decomposition and the weight formula from Subsection~\ref{subsec:rank-one-lattice} give \((V_{A_1})_1 =\C h\oplus\C e^\gamma\oplus\C e^{-\gamma}, h:=\gamma(-1)\one.\) Set \(x:=e^\gamma+e^{-\gamma}\) and \(y:=e^\gamma-e^{-\gamma}.\) Then \(\sigma:(h,x,y)\longmapsto(h,-x,-y), \theta:(h,x,y)\longmapsto(-h,x,-y).\) The normalized invariant bilinear form on $V_{A_1}$ is preserved by every VOA automorphism and is nondegenerate on $(V_{A_1})_1$ by Proposition~\ref{prop:Li-invariant-package}. The three lines $\C h,\C x,\C y$ are the distinct simultaneous eigenspaces of the commuting involutions $\sigma$ and $\theta$; hence they are pairwise orthogonal. Because their orthogonal direct sum is all of $(V_{A_1})_1$, the form is nonzero on each line. After rescaling the three vectors, and changing the sign of one vector if necessary, they give an oriented orthonormal basis. Under the standard identification of the $G$-module $(V_{A_1})_1$ with $W_2\cong\C^3$ from Corollary~\ref{cor:McRae-A1-SO3} and Remark~\ref{rmk:compact-complexification}, the subgroup $H$ is therefore conjugate in $G\cong SO_3(\C)$ to
\[
\{\operatorname{diag}(\varepsilon_1,\varepsilon_2,\varepsilon_3):
\varepsilon_i\in\{\pm1\},\ \varepsilon_1\varepsilon_2\varepsilon_3=1\}
=D_2.
\]
Choose $g\in G$ with $gHg^{-1}=D_2$. Since $G$ acts on $V_{A_1}$ by VOA automorphisms, \(V_{A_1}^{D_2}=V_{A_1}^{gHg^{-1}}=g\bigl(V_{A_1}^{H}\bigr) \cong V_{A_1}^{H}.\)

It remains to compute the latter fixed-point algebra. By the charge decomposition of $V_{A_1}$, the $\sigma$-fixed subspace is precisely the sum of the even charge sectors:
\[
V_{A_1}^{\langle\sigma\rangle}=\bigoplus_{m\in\Z}M(1,2m\gamma)=V_{2A_1}.
\] The explicit formula for the standard lift $\theta$ in Subsection~\ref{subsec:rank-one-lattice} shows that its restriction to $V_{2A_1}$ is exactly the standard lattice inversion for the sublattice $2A_1$. Since $\sigma$ and $\theta$ commute,
\[
V_{A_1}^{H}
 =\left(V_{A_1}^{\langle\sigma\rangle}\right)^{\langle\theta\rangle}
 =V_{2A_1}^{\langle\theta\rangle}
 =V_{2A_1}^+.
\]
Combining the two identifications proves the lemma.
\end{proof}

\begin{thm}\label{thm:k4}
Let $V$ satisfy the hypotheses of Theorem~\ref{thm:main}. If \(\ch V=\ch V_{\Lambda_4}^+,\) then \(V\cong V_{\Lambda_4}^+.\)
\end{thm}

\begin{proof}
Under the character hypothesis, Theorem~\ref{thm:k4-Szero} shows that, for every $0\ne u\in P_4(V)$, the cyclic Virasoro module $U(\Vir)u$ is an embedded copy of $X_2=L(1,4)$. By Lemma~\ref{lem:regular-primary}, choose a regular vector $u\in P_4(V)$. Such a vector is automatically nonzero because $\Delta_V(0)=0$. Set \(W:=\langle L(1,0),u\rangle_{\mathrm{VOA}}.\) Theorem~\ref{thm:k4-D2} and Lemma~\ref{lem:D2-lattice} give \(W\cong V_{A_1}^{D_2}\cong V_{2A_1}^+.\) Here $A_1=\Z\gamma$ with $\langle\gamma,\gamma\rangle=2$, so the sublattice $2A_1=2\Z\gamma$ is generated by $2\gamma$ of squared length $8$. Thus $2A_1$ is isometric to $\Lambda_4$. By the standard functoriality of the lattice-VOA construction recalled in Subsection~\ref{subsec:rank-one-lattice}, this lattice isometry identifies the standard lifts of $-1$ and hence induces \(V_{2A_1}^+\cong V_{\Lambda_4}^+.\) Consequently \(\ch W=\ch V_{\Lambda_4}^+=\ch V.\)

By Theorem~\ref{thm:k4-D2}, $W$ is a VOA extension of $L(1,0)$; in particular it contains the conformal vector of $V$ and is therefore a conformal subVOA of $V$. Hence $W_n\subseteq V_n$ for every $n\ge0$. Equality of the vacuum characters gives \(\dim W_n=\dim V_n\) for every $n$, so $W_n=V_n$ for all $n\ge0$. Thus $W=V$, and the preceding isomorphism yields \(V\cong V_{\Lambda_4}^+.\)
\end{proof}

\begin{thm}\label{thm:VLplus-all}
Let $L$ be any positive-definite even lattice of rank one, and let $V$ satisfy the hypotheses of Theorem~\ref{thm:main}. If \(\ch V=\ch V_L^+,\) then \(V\cong V_L^+.\) No Virasoro complete reducibility assumption is required.
\end{thm}

\begin{proof}
Choose a generator $\alpha$ of $L$. Since $L$ is positive definite and even, there is a unique integer $k\in\Z_{>0}$ such that \(\langle\alpha,\alpha\rangle=2k.\) Let $\Lambda_k:=\Z\beta$ with \(\langle\beta,\beta\rangle=2k.\) Then $L\cong\Lambda_k$ as even lattices, and the standard lattice-VOA construction gives \(V_L^+\cong V_{\Lambda_k}^+.\) Accordingly, the character hypothesis may be written as \(\ch V=\ch V_{\Lambda_k}^+.\)

If $k>4$, Theorem~\ref{thm:VLplus-kgt4} gives \(V\cong V_{\Lambda_k}^+.\) For $k=1,2,3,4$, the same conclusion follows from Theorems~\ref{thm:k1}, \ref{thm:k2}, \ref{thm:k3}, and \ref{thm:k4}, respectively. These cases exhaust $k\in\Z_{>0}$. Transporting the resulting isomorphism along $\Lambda_k\cong L$ gives \(V\cong V_L^+.\)
\end{proof}

\section{The lattice branch and the final classification}\label{sec:final}

\subsection{The lattice character branch}

We first isolate the only remaining structural branch. The following is precisely the form of the effective-central-charge theorem needed below.

\begin{thm}[Dong--Mason]\label{thm:DM-rank}
Let $V$ be strongly rational, and equip $V_1$ with its standard Lie bracket \([a,b]=a_0b\). Then $V_1$ is a finite-dimensional reductive complex Lie algebra. Set \(\ell:=\operatorname{rank}V_1=\dim\mathfrak h,\) where $\mathfrak h$ is any Cartan subalgebra of $V_1$. Then \(\ell\leq\widetilde c(V).\) Moreover, the following are equivalent:
\begin{enumerate}[label=\textnormal{(\roman*)}]
\item \(\ell=\widetilde c(V)=c(V);\)
\item $V$ is isomorphic to the lattice vertex operator algebra $V_K$ of a
positive-definite even lattice $K$.
\end{enumerate}
In this case, \(\operatorname{rank}K=\ell\).
\end{thm}

\begin{proof}
In \cite[Section~3]{DM}, Dong--Mason call a VOA strongly rational when it is rational, $C_2$-cofinite, and of strong CFT type, the latter meaning that it is of CFT type and satisfies $L(1)V_1=0$. Under our convention, $V$ is simple, rational, $C_2$-cofinite, self-contragredient, and of CFT type. Proposition~\ref{prop:Li-invariant-package}(ii) therefore gives \(L(1)V_1=0\), so all of the Dong--Mason hypotheses are satisfied.

Now \cite[Theorem~1]{DM} gives the reductivity of $V_1$, \cite[Theorem~2]{DM} gives \(\ell\leq\widetilde c(V)\), and \cite[Theorem~3]{DM} gives the equivalence of \textnormal{(i)} and \textnormal{(ii)}. Finally, if \(V\cong V_K\), then the standard lattice construction gives \(c(V_K)=\operatorname{rank}K\); see, for example, \cite[Chapter~8]{FLM}. Combining this with \(c(V)=\ell\) from \textnormal{(i)} yields \(\operatorname{rank}K=\ell\).
\end{proof}

\begin{thm}\label{thm:lattice-rigidity}
Let $V$ satisfy the hypotheses of Theorem~\ref{thm:main}, and let $L$ be a positive-definite even lattice of rank one. If \(\ch V=\ch V_L\), then \(V\cong V_L\).
\end{thm}

\begin{proof}
Put $\mathfrak h=\C\otimes_{\Z}L$. By the standard lattice construction recalled in Subsection~\ref{subsec:rank-one-lattice}, the Heisenberg line $\mathfrak h(-1)\one$ is contained in $(V_L)_1$. Since $\dim\mathfrak h=1$, this gives $(V_L)_1\ne0$. Moreover, $c(V)=1$ by hypothesis and $c(V_L)=\operatorname{rank}L=1$ for a lattice VOA; see \cite[Chapter~8]{FLM}. By the definition of the vacuum character, \(\ch V=q^{-1/24}\sum_{n\ge0}(\dim V_n)q^n, \ch V_L=q^{-1/24}\sum_{n\ge0}(\dim (V_L)_n)q^n\). Hence \(\ch V=\ch V_L\) implies \(\dim V_1=\dim (V_L)_1>0\), so $V_1\ne0$.

By Theorem~\ref{thm:DM-rank}, $V_1$ is a complex reductive Lie algebra. Writing \(\ell=\operatorname{rank}V_1\), we therefore have $\ell\ge1$: indeed, every nonzero reductive Lie algebra has a nonzero Cartan subalgebra. On the other hand, the same theorem and \(\widetilde c(V)=1\) give \(\ell\le1\). Thus \(\ell=\widetilde c(V)=c(V)=1.\) The equality case of Theorem~\ref{thm:DM-rank} now yields \(V\cong V_K\) for a positive-definite even lattice $K$ of rank one.

Choose generators $K=\Z\beta$ and $L=\Z\alpha$. Since both lattices are positive definite and even, there are unique $r,s\in\Z_{>0}$ such that \(\langle\beta,\beta\rangle=2r\) and \(\langle\alpha,\alpha\rangle=2s\). Put \(P(q)=\prod_{m\ge1}(1-q^m)^{-1}\). The rank-one lattice decomposition from Subsection~\ref{subsec:rank-one-lattice} gives \(V_K=\bigoplus_{n\in\Z}M(1,n\beta), V_L=\bigoplus_{n\in\Z}M(1,n\alpha).\) The oscillator calculation in the proof of Proposition~\ref{prop:rank-one-lattice-package} gives \(\grch M(1)=P(q)\), while the top vectors $e^{n\beta}$ and $e^{n\alpha}$ have conformal weights $rn^2$ and $sn^2$, respectively. Since both lattice VOAs have central charge one, it follows that \(\ch V_K=q^{-1/24}P(q)\sum_{n\in\Z}q^{rn^2}, \ch V_L=q^{-1/24}P(q)\sum_{n\in\Z}q^{sn^2}\). Because $V\cong V_K$ and \(\ch V=\ch V_L\), these two formal series are equal. Multiplying by $q^{1/24}$ and cancelling the unit $P(q)\in\C[[q]]^{\times}$ gives \(\sum_{n\in\Z}q^{rn^2}=\sum_{n\in\Z}q^{sn^2}.\) The least positive exponent on the two sides is $r$ and $s$, respectively; hence $r=s$. Therefore the map $\beta\mapsto\alpha$ is an isometry $K\cong L$. The standard lattice-VOA construction carries lattice isometries to VOA isomorphisms; see \cite[Chapter~8]{FLM}. Thus \(V_K\cong V_L\), and consequently \(V\cong V_L\).
\end{proof}

\subsection{Proof of the classification theorem}

\begin{proof}[Proof of Theorem~\ref{thm:main}]
Let $V$ satisfy the hypotheses of Theorem~\ref{thm:main}. By Theorem~\ref{thm:character-reduction}, there exists a positive-definite even rank-one lattice $L$ such that
\[
\ch V\in\left\{\ch V_L,\ch V_L^+,\ch V_{A_1}^{A_4},\ch V_{A_1}^{S_4},\ch V_{A_1}^{A_5}\right\}.
\] We consider the cases permitted by this character list.

If \(\ch V=\ch V_L\), then Theorem~\ref{thm:lattice-rigidity} gives \(V\cong V_L\). If \(\ch V=\ch V_L^+\), then Theorem~\ref{thm:VLplus-all} gives \(V\cong V_L^+\).

Suppose instead that \(\ch V=\ch V_{A_1}^{H}\) for some \(H\in\{A_4,S_4,A_5\}\). Since the hypotheses of Theorem~\ref{thm:main} imply that $V$ is simple, self-contragredient, of CFT type, and of central charge one, Theorem~\ref{thm:exceptional-rigidity} applies and yields \(V\cong V_{A_1}^{H}\).

These cases cover every character allowed by Theorem~\ref{thm:character-reduction}, and hence $V$ is isomorphic to one of the vertex operator algebras listed in Theorem~\ref{thm:main}. Finally, Theorem~\ref{thm:character-reduction} and the three rigidity theorems invoked above do not assume that $V$ is completely reducible as an $L(1,0)$-module, or that $V$ is a direct sum of Virasoro highest-weight modules. Thus no such additional hypothesis is used, which proves the final assertion of Theorem~\ref{thm:main}.
\end{proof}

\begin{cor}\label{cor:no-Vir-splitting}
Under the hypotheses of Theorem~\ref{thm:main}, the classification conclusion holds without assuming that $V$ is a direct sum of highest-weight $L(1,0)$-modules. In particular, the additional Virasoro-splitting hypothesis appearing in earlier characterization arguments such as \cite{DJplusI,DJplusII,DJA4} is not needed here.
\end{cor}

\begin{proof}
The proof of Theorem~\ref{thm:character-reduction} uses the modularity input of Theorem~\ref{thm:DJ-modularity}, the CFT-type hypothesis, and the condition $c=\widetilde c=1$, but nowhere assumes that the ambient $L(1,0)$-module $V$ is semisimple or a direct sum of highest-weight modules. The lattice branch is then handled by Theorem~\ref{thm:lattice-rigidity}. Theorem~\ref{thm:VLplus-all} treats the rank-one orbifold branch by constructing only the semisimple subcores needed for reconstruction, and Theorem~\ref{thm:exceptional-rigidity} treats the three exceptional branches in the same way. Hence none of the branches in the proof of Theorem~\ref{thm:main} requires a global Virasoro decomposition of $V$.
\end{proof}

\begin{rmk}
The vertex operator algebra $V_L$ of a positive-definite even lattice $L$ is simple, regular, of CFT type, and self-contragredient; in particular, its standard invariant bilinear form is nondegenerate. The same statements hold for $V_{A_1}$. Thus the hypotheses of Carnahan--Miyamoto's solvable-orbifold theorem are satisfied \cite[Corollary~5.25]{CarnahanMiyamoto}. If $\theta$ denotes lattice inversion, then $\langle\theta\rangle$ has order two, while $A_4$ and $S_4$ are finite solvable groups. Consequently, \(V_L^+=V_L^{\langle\theta\rangle}, V_{A_1}^{A_4}, V_{A_1}^{S_4}\) are simple and regular, have nonnegative $L(0)$-spectrum, and are self-contragredient. Since the automorphism groups fix the vacuum and the degree-zero space of the ambient lattice VOA is $\C\one$, each fixed-point algebra has degree-zero space $\C\one$ and hence is of CFT type. For a VOA of CFT type, regularity is equivalent to rationality together with $C_2$-cofiniteness \cite[Theorem~4.5]{ABD}. It follows that $V_L$, $V_L^+$, $V_{A_1}^{A_4}$, and $V_{A_1}^{S_4}$ are strongly rational.

For the remaining icosahedral model, Xu proves directly that $V_{A_1}^{A_5}$ is strongly rational \cite{A5new}. Thus every vertex operator algebra occurring in Theorem~\ref{thm:main} is strongly rational.
\end{rmk}

\appendix
\section{Why ambient Virasoro semisimplicity cannot be used as a formal shortcut}\label{app:nonsplit}

For completeness, we record a module-theoretic obstruction explaining why the generated-core arguments used above cannot be replaced by a formal appeal to integral conformal weights or trivial categorical twist. Even at central charge one, semisimplicity of $L(0)$, integral conformal weights, and trivial twist do not imply complete reducibility as a Virasoro module.

\begin{prop}\label{prop:Em}
For every $m\in\Z_{\ge0}$, choose a nonzero singular-vector embedding \(\iota_m:V(1,(m+1)^2)\hookrightarrow V(1,m^2)\) whose image is the unique maximal proper submodule from Proposition~\ref{prop:c1-verma}, and set \(E_m:=V(1,m^2)\big/ (\iota_m\circ\iota_{m+1})V(1,(m+2)^2).\) Then $E_m$ is an object of $\Ovir$ and fits into a nonsplit exact sequence of $L(1,0)$-modules \(0\longrightarrow X_{m+1} \longrightarrow E_m \longrightarrow X_m \longrightarrow0, X_j=L(1,j^2)\). Moreover, $L(0)$ acts semisimply on $E_m$, all conformal weights of $E_m$ are integers, and its standard categorical twist \(\theta_{E_m}=e^{2\pi iL(0)}\) is the identity. Finally, \(\grch E_m=\grch X_m+\grch X_{m+1}.\)
\end{prop}

\begin{proof}
Write $M_j:=V(1,j^2)$ for $j\ge0$. By Proposition~\ref{prop:c1-verma}, $\iota_jM_{j+1}$ is the unique maximal proper submodule of $M_j$ and \(M_j/\iota_jM_{j+1}\cong X_j\). Hence \(M_m\supset \iota_mM_{m+1} \supset (\iota_m\circ\iota_{m+1})M_{m+2},\) and injectivity of $\iota_m$ gives \(\frac{\iota_mM_{m+1}} {(\iota_m\circ\iota_{m+1})M_{m+2}} \cong \frac{M_{m+1}}{\iota_{m+1}M_{m+2}} \cong X_{m+1}\). Together with \(M_m/\iota_mM_{m+1}\cong X_m\), the third isomorphism theorem therefore yields an exact sequence of Virasoro modules \(0\longrightarrow X_{m+1}\longrightarrow E_m \longrightarrow X_m\longrightarrow0.\)

The quotient $E_m$ is a highest-weight Virasoro module of central charge one: the image $\bar v_m$ of a highest-weight vector $v_m\in M_m$ is nonzero and generates $E_m$. Proposition~\ref{prop:CJORY-package}(i), together with the identification $\Vir_1=L(1,0)$ established immediately after that proposition, therefore equips $E_m$ with its natural $L(1,0)$-module structure. The simple modules $X_m$ and $X_{m+1}$ are already $L(1,0)$-modules. Since $L(1,0)$ is generated by its conformal vector, the Virasoro homomorphisms in the preceding exact sequence are $L(1,0)$-module homomorphisms. Thus the sequence is exact in the category of $L(1,0)$-modules.

It is nonsplit already as a sequence of Virasoro modules. Indeed, the submodule $(\iota_m\circ\iota_{m+1})M_{m+2}$ has lowest conformal weight $(m+2)^2$, whereas the weight-$m^2$ subspace of $M_m$ is the one-dimensional highest-weight line $\C v_m$. Hence \((E_m)_{m^2}=\C\bar v_m\). If the sequence split as Virasoro modules, there would be a Virasoro submodule $S\subset E_m$ mapping isomorphically onto $X_m$. Its highest-weight line has weight $m^2$, so it must equal $\C\bar v_m$. Thus $\bar v_m\in S$. Since $\bar v_m$ generates $E_m$, this forces $S=E_m$, contradicting the nonzero kernel $X_{m+1}$. Therefore no Virasoro splitting exists, and a fortiori no $L(1,0)$-module splitting exists.

Now \(m^2=h_{2m+1,1}(1)\) and \((m+1)^2=h_{2m+3,1}(1)\), so both simple composition factors lie in the degenerate set $H_1$. The exact sequence above shows that $E_m$ has finite length with composition factors $X_{m+1}$ and $X_m$. Hence Proposition~\ref{prop:CJORY-package}(iii) gives $E_m\in\Ovir$.

Finally, the PBW grading gives an algebraic direct-sum decomposition
\[
M_m=\bigoplus_{N\ge0}(M_m)_{m^2+N}
\]
into finite-dimensional $L(0)$-eigenspaces. Because $\iota_m\circ\iota_{m+1}$ commutes with $L(0)$, its image is a graded submodule, so the quotient $E_m$ inherits an algebraic direct sum of $L(0)$-eigenspaces. Consequently, $L(0)$ acts semisimply on $E_m$ and \(\operatorname{Spec}_{E_m}L(0)\subset m^2+\Z_{\ge0}\subset\Z\). In the braided tensor category underlying a vertex tensor category of VOA modules, the standard twist is the module automorphism \(\theta_W=e^{2\pi iL(0)}\); see \cite[Section~3.3]{CKM}. Therefore \(\theta_{E_m}=\id_{E_m}\). The exact sequence preserves the $L(0)$-grading, so additivity of dimensions in each finite-dimensional homogeneous subspace gives \(\grch E_m=\grch X_m+\grch X_{m+1}\), as claimed.
\end{proof}

\begin{rmk}
The module $E_m$ in Proposition~\ref{prop:Em} satisfies all of the following simultaneously: $L(0)$ acts semisimply, all conformal weights are integral, the categorical twist is trivial, and its graded character is the sum of the graded characters of its two simple composition factors. Nevertheless, $E_m$ is a nonsplit Virasoro extension. Thus these formal properties, even taken together, do not imply complete reducibility as a Virasoro module. Proposition~\ref{prop:Em} is only a module-theoretic obstruction; it does not assert that $E_m$ occurs inside any of the strongly rational vertex operator algebras classified above. In Sections~\ref{sec:squarecore}--\ref{sec:low}, we instead extract only the simple submodules needed for reconstruction and prove that the vertex subalgebras they generate remain in explicitly semisimple tensor subcategories.
\end{rmk}

\subsection*{AI Use Statement}

OpenAI's GPT-5.6 Sol model was used as an auxiliary tool for literature exploration, conceptual brainstorming, and language editing during the preparation of this manuscript. All mathematical arguments, references, and final formulations were independently checked by the author, who takes full responsibility for the content of the paper.

\begingroup
\linespread{1.00}\selectfont

\endgroup


\begin{thebibliography}{99}

\bibitem{ABD}
T.~Abe, G.~Buhl and C.~Dong,
\emph{Rationality, regularity, and $C_2$-cofiniteness},
Trans. Amer. Math. Soc. \textbf{356} (2004), 3391--3402.

\bibitem{ALY}
D.~Adamovi\'c, X.~Lin and J.~Yang,
\emph{Tensor category of $\mathbb Z_2$-orbifold of Heisenberg vertex operator algebra and its applications},
arXiv:2604.12120v1 (2026).

\bibitem{CarnahanMiyamoto}
S.~Carnahan and M.~Miyamoto,
\emph{Regularity of fixed-point vertex operator subalgebras},
arXiv:1603.05645 (2016), revised 2018.

\bibitem{CJORY}
T.~Creutzig, C.~Jiang, F.~Orosz Hunziker, D.~Ridout and J.~Yang,
\emph{Tensor categories arising from the Virasoro algebra},
Adv. Math. \textbf{380} (2021), 107601.

\bibitem{CKLR}
T.~Creutzig, S.~Kanade, A.~R.~Linshaw and D.~Ridout,
\emph{Schur--Weyl duality for Heisenberg cosets},
Transform. Groups \textbf{24} (2019), no.~2, 301--354.

\bibitem{CKM}
T.~Creutzig, S.~Kanade and R.~McRae,
\emph{Tensor categories for vertex operator superalgebra extensions},
Mem. Amer. Math. Soc. \textbf{295} (2024), no.~1472.

\bibitem{CMY}
T.~Creutzig, R.~McRae and J.~Yang,
\emph{Direct limit completions of vertex tensor categories},
Commun. Contemp. Math. \textbf{24} (2022), no.~2, 2150033.

\bibitem{DG98}
C.~Dong and R.~L.~Griess Jr.,
\emph{Rank one lattice type vertex operator algebras and their automorphism groups},
J. Algebra \textbf{208} (1998), 262--275.

\bibitem{DongNagatomo}
C.~Dong and K.~Nagatomo,
\emph{Classification of irreducible modules for the vertex operator algebra $M(1)^+$},
J. Algebra \textbf{216} (1999), 384--404.

\bibitem{DJA4}
C.~Dong and C.~Jiang,
\emph{A characterization of the vertex operator algebra $V_{L_2}^{A_4}$},
in: Conformal Field Theory, Automorphic Forms and Related Topics,
Contrib. Math. Comput. Sci. \textbf{8}, Springer, Heidelberg, 2014, 55--74.

\bibitem{DJplusI}
C.~Dong and C.~Jiang,
\emph{A characterization of vertex operator algebras $V_{\mathbb Z\alpha}^{+}$: I},
J. Reine Angew. Math. \textbf{709} (2015), 51--79.

\bibitem{DJplusII}
C.~Dong and C.~Jiang,
\emph{A characterization of the rational vertex operator algebra $V_{\mathbb Z\alpha}^{+}$: II},
Adv. Math. \textbf{247} (2013), 41--70.

\bibitem{DM}
C.~Dong and G.~Mason,
\emph{Rational vertex operator algebras and the effective central charge},
Int. Math. Res. Not. 2004, no.~56, 2989--3008.

\bibitem{FLM}
I.~B.~Frenkel, J.~Lepowsky and A.~Meurman,
\emph{Vertex Operator Algebras and the Monster},
Pure and Applied Mathematics, vol.~134,
Academic Press, Boston, 1988.

\bibitem{Hartshorne}
R.~Hartshorne,
\emph{Algebraic Geometry},
Graduate Texts in Mathematics, vol.~52,
Springer-Verlag, New York--Heidelberg, 1977.

\bibitem{SpringerLAG}
T.~A.~Springer,
\emph{Linear Algebraic Groups},
2nd ed., Birkh\"auser, Boston, 1998.

\bibitem{HKL}
Y.-Z.~Huang, A.~Kirillov Jr. and J.~Lepowsky,
\emph{Braided tensor categories and extensions of vertex operator algebras},
Comm. Math. Phys. \textbf{337} (2015), 1143--1159.

\bibitem{Kiritsis}
E.~Kiritsis,
\emph{Proof of the completeness of the classification of rational conformal theories with $c=1$},
Phys. Lett. B \textbf{217} (1989), 427--430.

\bibitem{LepowskyLi}
J.~Lepowsky and H.~Li,
\emph{Introduction to Vertex Operator Algebras and Their Representations},
Progress in Mathematics, vol.~227,
Birkh\"auser, Boston, 2004.

\bibitem{LiInvariant}
H.~Li,
\emph{Symmetric invariant bilinear forms on vertex operator algebras},
J. Pure Appl. Algebra \textbf{96} (1994), 279--297.

\bibitem{KR}
V.~Kac and A.~Raina,
\emph{Bombay Lectures on Highest Weight Representations of Infinite Dimensional Lie Algebras},
World Scientific, Singapore, 1987.

\bibitem{Milas}
A.~Milas,
\emph{Fusion rings for degenerate minimal models},
J. Algebra \textbf{254} (2002), no.~2, 300--335.

\bibitem{MFK}
D.~Mumford, J.~Fogarty and F.~Kirwan,
\emph{Geometric Invariant Theory},
3rd ed., Ergebnisse der Mathematik und ihrer Grenzgebiete, vol.~34,
Springer, Berlin, 1994.

\bibitem{Matsushima}
Y.~Matsushima,
\emph{Espaces homog\`enes de Stein des groupes de Lie complexes},
Nagoya Math. J. \textbf{16} (1960), 205--218.

\bibitem{BHC}
A.~Borel and Harish-Chandra,
\emph{Arithmetic subgroups of algebraic groups},
Ann. of Math. (2) \textbf{75} (1962), 485--535.

\bibitem{McRae}
R.~McRae,
\emph{On the tensor structure of modules for compact orbifold vertex operator algebras},
Math. Z. \textbf{296} (2020), 409--452.

\bibitem{MY}
R.~McRae and J.~Yang,
\emph{An $\mathfrak{sl}_2$-type tensor category for the Virasoro algebra at central charge $25$ and applications},
Math. Z. \textbf{303} (2023), Article 32.

\bibitem{A5new}
S.~Xu,
\emph{$C_2$-cofiniteness and rationality of the icosahedral orbifold $V_{L_2}^{A_5}$},
arXiv:2608.13872v1 (2026).

\bibitem{Zhu}
Y.~Zhu,
\emph{Modular invariance of characters of vertex operator algebras},
J. Amer. Math. Soc. \textbf{9} (1996), no.~1, 237--302.

\bibitem{ZhangDong}
W.~Zhang and C.~Dong,
\emph{$W$-algebra $W(2,2)$ and the vertex operator algebra $L(\frac12,0)\otimes L(\frac12,0)$},
Comm. Math. Phys. \textbf{285} (2009), 991--1004.

\end{thebibliography}
\end{document}